\documentclass[a4paper,10pt]{article}
\usepackage{courier}
\usepackage{blindtext} 
\usepackage{amsmath}
\usepackage{amssymb}
\usepackage[colorlinks=true,breaklinks]{hyperref}
\usepackage[hyphenbreaks]{breakurl}
\usepackage{xcolor}
\definecolor{c1}{rgb}{0,0,1}
\definecolor{c2}{rgb}{0,0.3,0.9}
\definecolor{c3}{rgb}{0.3,0.9}
\hypersetup{linkcolor={c1},citecolor={c2},urlcolor={c3}}
\usepackage{enumerate}
\usepackage{todonotes}
\usepackage{makeidx}
\makeindex
\usepackage[top=4cm,bottom=4cm,left=3cm,right=3cm]{geometry}
\usepackage{fancyhdr}

\def\XXint#1#2#3{{\setbox0=\hbox{$#1{#2#3}{\int}$ }
\vcenter{\hbox{$#2#3$ }}\kern-.6\wd0}}

\usepackage{amsthm}
\theoremstyle{plain}
\newtheorem{theorem}{Theorem}[section]

\theoremstyle{definition}

\theoremstyle{lemma}
\newtheorem{lemma}[theorem]{Lemma}

\theoremstyle{Remark}
\newtheorem{Remark}[theorem]{Remark}

\theoremstyle{proposition}

\theoremstyle{corollary}

\theoremstyle{example}

\theoremstyle{assumption}
\newtheorem{assumption}[theorem]{Assumption}

\usepackage{mathrsfs}
\usepackage{systeme}
\usepackage{colonequals}
\usepackage{comment}

\begin{document}

\title{Well-posedness and trend to equilibrium for the Vlasov-Poisson-Fokker-Planck system with a confining potential} 

\author{Gayrat Toshpulatov}




\maketitle


{\footnotesize
 \centerline{Institut f\"ur Analysis $\&$ Scientific Computing,} \centerline{Technische Universit\"at Wien, Wiedner Hauptstr. 8-10, A-1040 Wien, Austria,}
\centerline{gayrat.toshpulatov@tuwien.ac.at}} 
\medskip




\begin{abstract}
We study  well-posedness and long time behavior of  the  nonlinear Vlasov-Poisson-Fokker-Planck system with an external confining potential. The system describes the time evolution of particles (e.g.$\,\,$in a plasma) undergoing diffusion, friction and Coulomb interaction. 
We prove  existence and uniqueness of mild solutions  in a weighted Sobolev space. Moreover, we prove that the solutions converge to the global equilibrium exponentially. Our results hold for a wide class of external potentials and the  estimates on the rate of convergence are explicit and constructive.  The technique is based on  the construction of Lyapunov functionals,   new  short and long time estimates for the linearized system, and fixed point arguments.  
\end{abstract}
\tableofcontents

\section{Introduction}

This paper is devoted to the study of  well-posedness and long time behavior of  the nonlinear Vlasov-Poisson-Fokker-Planck system
\begin{equation}\label{VPFP}
\begin{cases}
\partial_t f+v\cdot\nabla_x f-(\nabla_x V+\nabla_x \phi)\cdot \nabla_v f=\nu\text{div}_v(vf)+\sigma\Delta_v f, \, \, \, \, x, v\in \mathbb{R}^d, \, \, t>0\\
-\Delta_x \phi=\displaystyle\int_{\mathbb{R}^d}f dv, \, \, \, \, f_{|t=0}=f_0.
\end{cases}
\end{equation}
 The system  is one  of the fundamental  models in plasma physics, for the derivation and applications we refer to \cite{Chandra1, Chandra2, Pit, 10.2307/24895438}.  The variables $t\geq 0,$ $x \in \mathbb{R}^d,$ and $v \in \mathbb{R}^d,$ respectively, stand for time, position, and velocity. The first unknown  $f=f(t,x,v)\geq 0$ describes the evolution of the phase space probability density of charged particles. The second unknown  $\phi =\phi(t,x)$ 
  determines  the self-consistent \textit{repulsive} electrostatic potential. Because of the Poisson equation in \eqref{VPFP} we have
 \begin{equation*}\label{phi}
\nabla_x\phi=\frac{1}{|\mathbb{S}^{d-1}|}\frac{x}{|x|^d}\ast \int_{\mathbb{R}^d} f dv,
\end{equation*}
where $|\mathbb{S}^{d-1}|$ is the area of the unit sphere $\mathbb{S}^{d-1}$ in $\mathbb{R}^d.$ $V=V(x)$ is a given external electrostatic confinement  potential  (i.e., $V(x)\to +\infty$ as $|x|\to \infty$).  The operator $v\cdot \nabla_x-(\nabla_x V+\nabla_x \phi)\cdot \nabla_v$ is the transport operator. 
The Fokker-Planck  operator $\nu\text{div}_v(v\, \cdot)+\sigma\Delta_v $  describes the collision effects of particles and the interaction with the environment. $\nu>0$ and $\sigma>0$ denote respectively the friction and diffusion parameters.

In the literature there exists also the Vlasov-Poisson-Fokker-Planck system with  the self-consistent \textit{attractive}  electrostatic potential,  and it  is  widely used
in stellar physics \cite{Pad}. For that case the self-consistent electrostatic potential $\phi$ is defined by a change of sign in the Poisson equation. In this paper, we only consider the repulsive case.

We mention that, if  there is an interaction with a fixed background of positive charges in a plasma, then  the self-consistent electrostatic potential $\phi$ is defined  by $-\Delta_x \phi=\int_{\mathbb{R}^d}f dv-n,$ where $n=n(x)$ is a given non-negative
function which describes the background density (e.g. of ions). In this case, we can still write the system in the form of \eqref{VPFP} by replacing $\phi$ with  $\tilde{\phi}$ defined by $-\Delta_x \tilde{\phi}=\int_{\mathbb{R}^d}f dv$ and replacing $V$ with  $\tilde{V}\colonequals V+V_0,$ where   $V_0$ is defined by $ \Delta_x V_0= n.$   Therefore, without loss of generality, we assume $n= 0.$

The system has several properties following standard physical consideration. First, the Fokker-Planck  operator \textit{acts only on the velocity $v.$} It reflects the physical fact that collisions are localized in space.

Whenever $f(t,x,v)$ is a (well-behaved) solution,  we have \textit{conservation of mass}
\begin{equation*}   \displaystyle \int_{\mathbb{R}^{2d}}f(t,x,v)dxdv=\int_{\mathbb{R}^{2d}}f_0(x,v)dxdv, \, \, \, \, \, \, \forall t>0.
\end{equation*}
 Therefore, without loss of generality, we shall assume 
$f_0\geq 0$ and $ \displaystyle \int_{\mathbb{R}^{2d}}f_0(x,v)dxdv=1.$

If $V$ grows fast enough as $|x|\to \infty,$
 the system has a unique normalized \textit{steady state} or \textit{global equilibrium} \cite{Dolbeault1991STATIONARYSI, GLZ}
\begin{equation*}\label{steady state}
f_{\infty}(x,v)=
\rho_{\infty}(x)M(v),
\end{equation*}
where $$\rho_{\infty}(x)\colonequals \displaystyle \frac{e^{-\frac{\nu}{\sigma}[V(x)+\phi_{\infty}(x)]}}{ \int_{\mathbb{R}^{d}}e^{-\frac{\nu}{\sigma}[V(x')+\phi_{\infty}(x')]}dx' },\, \, \, \, \, \, \, \displaystyle M(v)\colonequals \frac{e^{-\frac{\nu}{\sigma}|v|^2/2}}{ (2\pi \sigma/\nu)^{d/2}},$$ and $\phi_{\infty}$ is a solution of  the  Poisson-Boltzmann-Emden equation \cite{PBE}
\begin{equation}\label{eq.steady state}
-\Delta_x \phi_{\infty}(x)=\displaystyle \frac{e^{-\frac{\nu}{\sigma}[V(x)+\phi_{\infty}(x)]}}{ \int_{\mathbb{R}^{d}}e^{-\frac{\nu}{\sigma}[V(x')+\phi_{\infty}(x')]}dx' }.
\end{equation}

The system is \textit{dissipative} in the sense that the following  relative entropy or free energy functional decreases under the time-evolution of $f$ \cite{H-theorem, Dol}: let $\mathrm{H}$ be a functional defined on the space of probability densities by
\begin{equation*}\label{H}
f\to \mathrm{H}[f]\colonequals \int_{\mathbb{R}^{2d}} f \ln{\frac{f}{f_{\infty}}}dxdv+\int_{\mathbb{R}^{d}}|\nabla_x \phi-\nabla_x \phi_{\infty}|^2dx,
\end{equation*}
with $\phi$ given by the Poisson equation $\displaystyle -\Delta_x \phi=\int_{\mathbb{R}^d}f dv.$ 
By  the Csisz\'ar-Kullback-Pinsker inequality \cite{Cs}
\begin{equation*}
\mathrm{H}[f]\geq \frac{1}{2} ||f -f_{\infty}||^2_{L^1(\mathbb{R}^{2d})}+\int_{\mathbb{R}^{d}}|\nabla_x \phi-\nabla_x \phi_{\infty}|^2dx\geq 0,
\end{equation*}
the minimum of $\mathrm{H}$ is zero and  it is attained at $f_{\infty}$ (i.e., $\mathrm{H}[f]\geq 0$ for all probability densities $f$ and  $\mathrm{H}[f_{\infty}]=0).$   If $f=f(t,x,v)$ solves   \eqref{VPFP} and has sufficient smoothness and decay properties (as $|(x,v)^T|\to \infty$), we have $$\displaystyle \frac{d}{dt}\mathrm{H}[f(t)] 
 \leq 0. $$
 This decay of the functional $\mathrm{H}$ reminds us of the famous Boltzmann $H-$theorem stated for the Boltzmann equation \cite{Cer}. This similarity is  expected since the Fokker-Planck operator can be considered as a linear variant of Boltzmann's collision operator \cite{Cer, Vil.H}.

 On the basis of the decay of the functional $\mathrm{H},$  one can expect that  $\mathrm{H}[f(t)]$ decreases to its minimum (which is zero) as $t\to \infty.$ Since this minimum   is attained at $f_{\infty},$ one can argue that $f(t)$ converges to the equilibrium distribution $f_{\infty}$ as $t\to \infty.$
 Clearly, before proving this convergence, we first need to establish the well-posedness of the system \eqref{VPFP}.
  Therefore, we get  an important problem:
   to prove  existence and uniqueness of  the solution $f,$ then
   to prove  the convergence $f(t)\to f_{\infty} $ as $t\to \infty.$

  If we have a reasonable solution $f$ and if it satisfies  some a-priori bounds, using the decay of the functional $\mathrm{H}$ above and compactness tools, we can prove that $f(t)$ does indeed converge to  $f_{\infty} $ as $t\to \infty$ \cite{Bouch.Dol, Dol}. But this method based on compactness gives no information on the rate of convergence and it is non-constructive. We are interested in the study of rates of convergence  and we want to derive  constructive bounds for this convergence, because explicit and constructive estimates are essential for applications  in physics (e.g. equilibration process, numerical simulations).

  There are many works dealing with this problem.
When the system \eqref{VPFP} does not have a confining potential (i.e., $V=0$)   existence,  uniqueness and  asymptotic behavior have been studied comprehensively: Degond considered the frictionless  system (i.e., $\nu=0$) in \cite{Degond} and showed global  existence of classical solutions  in dimension $d\leq 2,$ see also \cite{Ono, Ono.Str}. The long time behavior of the frictionless  system  was studied in \cite{Car.Sol.Vaz, Carpio, kagei}. With non-zero friction (i.e., $\nu>0$), global existence of  classical solutions in dimension $d\leq 2$ and  local in time existence in dimension $d\geq 3$ were obtained by Victory and O'Dwyer in \cite{10.2307/24895438}.  Bouchut \cite{BOUCHUT1993239} proved global existence of classical solutions in dimension $d=3.$  He also showed in \cite{BOUCHUT1995225} that the system has smoothing properties.  Then  global existence of weak solutions in dimension $d=3$ was studied in  \cite{Vic, Car.Sol}.
Using the micro-macro strategy   Hwang and Jang \cite{Hwang.Jang} obtained  exponential decay in a close-to-equilibrium regime.  There are recent studies \cite{Tris, Bedro} on torus (i.e., $x\in \mathbb{T}^d$) concerning the long time behavior and Landau damping in a weak collisional regime, i.e., if $\nu$ and $\sigma$ are sufficiently small.

When the equation has a non-zero confining potential $V,$  there are only few studies: 
  When the self-consistent interaction  is sufficiently small (i.e., the non-linear term  $\nabla_x \phi\cdot \nabla_v f$  is replaced with  $\varepsilon\nabla_x \phi\cdot \nabla_v f$ and $\varepsilon$ is sufficiently small) and $\partial^2_{x_ix_j} V\in \bigcap_{p=1}^{\infty} W^{p, \infty}$ for all $i,j \in \{1,...,d\},$ H\'erau and Thomann \cite{Herau.Thomann}  proved a global existence result in dimensions $d=2$ and $d=3.$ They also showed that the solution converges to the steady state exponentially. Their proof relies on the hypocoercive and hypoelliptic properties  of  the linear kinetic Fokker-Planck equation obtained in \cite{Herau.Nier,Herau} and a fixed point argument. Recently,  Abddala, Dolbeault et al.\,\,\cite{addala2021text} studied the linearized Vlasov-Poisson-Fokker-Planck system  around the steady state $f_{\infty}.$ They used a hypocoercive method developed in \cite{dolbeault2009hypocoercivity, dolbeault2015hypocoercivity} and proved exponential stability of the linearized system. 


  In this paper, we shall improve these previous results when there is a non-zero potential $V.$ For the full system \eqref{VPFP} with a non-zero potential $V$ and the parameters $\nu>0, $ $\sigma>0$ we shall prove existence, uniqueness and convergence $f(t)\to f_{\infty} $ as $t\to \infty$ for a wide class of potentials $V.$  Moreover,  our decay rates  are explicit and constructive.

 The organization of this paper is as follows. In Section 2 we present the assumptions on the potential, define functional  spaces and state the main results. In Section 3 we show existence and regularity of the steady state, establish Poincar\'e type inequalities and gather some estimates for the Poisson equation. Section 4 contains the analysis of the linearized Vlasov-Poisson-Fokker-Planck system: existence, uniqueness, exponential stability, and hypoelliptic regularity. The final section presents some estimates on the semigroup of the linearized system and  the proof of the main results concerning the nonlinear Vlasov-Poisson-Fokker-Planck system.

\section{Setting and main results}
We make the following assumptions on the external potential $V.$
\begin{assumption}
\begin{itemize}
\item[(A1)]
$V\in C^{\infty}(\mathbb{R}^d)$ is bounded from below,  \begin{equation*}\label{A1}
e^{-\frac{\nu}{\sigma}V} \in L^1(\mathbb{R}^d)\, \, \, \,  \text{ and } \, \, \, \,  |\nabla_x V|e^{-\frac{\nu}{\sigma}V} \in L^r(\mathbb{R}^d),\, \, r>d.
\end{equation*}
\item[(A2)] There exists a constant $c_1>0$ such that
 \begin{equation}\label{hess<grad2}
 \left| \left|\frac{\partial^2 V(x)}{\partial x^2}\right|\right|_F\leq c_1(1+|\nabla_x V(x)|), \, \, \, \, \, \,  \forall x\in \mathbb{R}^d,
\end{equation}
where $||\cdot||_F$ denotes the Frobenius norm.
\item[(A3)] There exists a constant $ \kappa_1>0$ such that  the Poincar\'e inequality
\begin{equation}\label{Poin.V}
 \int_{\mathbb{R}^d}h^2e^{-\frac{\nu}{\sigma}V}dx-\left(\int_{\mathbb{R}^d}he^{-\frac{\nu}{\sigma}V}dx\right)^2\leq \kappa_1\int_{\mathbb{R}^d}|\nabla_xh|^2 e^{-\frac{\nu}{\sigma}V}dx
\end{equation}
holds for all $h$ with $\displaystyle \int_{\mathbb{R}^d}h^2e^{-\frac{\nu}{\sigma}V}dx<\infty$ and $\displaystyle \int_{\mathbb{R}^d}|\nabla_x h|^2e^{-\frac{\nu}{\sigma}V}dx<\infty.$
\end{itemize}
\end{assumption}

There are a lot of studies and sufficient conditions implying the Poincar\'e inequality \eqref{Poin.V}. For example, if $V$ is uniformly convex (Bakry-Emery criterion) or $ \liminf_{|x|\to \infty} \left(a|\nabla V(x)|^2-\Delta V(x)\right)> 0
$   for some
 $a \in (0,1),$
then the Poincar\'e inequality holds,  for more information  see \cite[Chapter 4]{bakry2014analysis}, \cite{Bakry2008ASP}. We note that Assumption \ref{A1} includes the potentials $V$ considered in \cite{addala2021text, Herau.Thomann}.
It is possible to make weaker regularity hypothesis on the potential $V,$ but we maintain the assumption that $V \in C^{\infty} $ to keep the presentation simple. We note that the potentials of the form $$V(x)=r|x|^{k}+V_0(x),$$
where $r>0,$ $k>1$ and $V_0\colon \mathbb{R}^n \to \mathbb{R}$ is a polynomial of degree  $j<k,$ satisfy our assumptions. In particular, it includes the double-well potentials of the form $V(x)=r_1 |x|^4-r_2|x|^2,$ $r_1, r_2>0.$

We define  the following weighted spaces 
$$L^2(\mathbb{R}^{2d}, f_{\infty})\colonequals \left\{ g :\mathbb{R}^{2d}\to \mathbb{R}: \, \, \int_{\mathbb{R}^{2d}}g^2f_{\infty}dxdv<\infty \right\}$$
and
$$H^1(\mathbb{R}^{2d}, f_{\infty})\colonequals \left\{ g\in L^2(\mathbb{R}^{2d}, f_{\infty}): \int_{\mathbb{R}^{2d}}|\nabla_x g|^2f_{\infty}dxdv +\int_{\mathbb{R}^{2d}}|\nabla_v g|^2f_{\infty}dxdv<\infty \right\}.$$
The corresponding  norms are
\begin{equation*}\label{L^2 norm}||g||_{ L^2(\mathbb{R}^{2d},f_{\infty})}\colonequals \sqrt{\int_{\mathbb{R}^{2d}}g^2f_{\infty}dxdv}
\end{equation*}
and $$||g||_{ H^1(\mathbb{R}^{2d},f_{\infty})}\colonequals \sqrt{ \int_{\mathbb{R}^{2d}} g^2f_{\infty}dxdv+\int_{\mathbb{R}^{2d}}|\nabla_x g|^2f_{\infty}dxdv+\int_{\mathbb{R}^{2d}}|\nabla_v g|^2f_{\infty}dxdv}. $$

We define the generalized Sobolev space or the Bessel potential space  \cite[Section V.3]{stein1970singular}, \cite[Section 1.2.6]{Adams}
$$\mathscr{L}^{p}_{\alpha}(\mathbb{R}^d)\colonequals \{g:\mathbb{R}^d \to \mathbb{R}: \, \, \, (1-\Delta_x)^{\frac{\alpha}{2}}g\in L^p(\mathbb{R}^d)\}, \, \, \, \, \,1<p<\infty, \, \, \, \, \alpha \in \mathbb{R}, $$
where  $(1-\Delta_x)^{\frac{\alpha}{2}}g\colonequals \mathcal{F}^{-1}\left((1+4\pi^2 |\xi|^2)^{\frac{\alpha}{2}}\mathcal{F}{g}\right)$ with the Fourier transform $\displaystyle \mathcal{F}g(\xi)\colonequals \int_{\mathbb{R}^d}g(x)e^{-2\pi x\cdot \xi}dx.$
The norm on $\mathscr{L}^{p}_{\alpha}(\mathbb{R}^d)$ is $$||g||_{\mathscr{L}^{p}_{\alpha}(\mathbb{R}^d)}\colonequals ||(1-\Delta_x)^{\frac{\alpha}{2}}g||_{L^{p}(\mathbb{R}^d)}.$$ For $\alpha\in \mathbb{N},$ $\mathscr{L}^{p}_{\alpha}(\mathbb{R}^d)$ coincides with the usual Sobolev space $W^{\alpha, p}(\mathbb{R}^d),$  $1<p<\infty.$ Let  $\alpha_1,\alpha_2\in \mathbb{R},$ $1<p_1,p_2<\infty,$ $\theta\in (0,1),$ $\alpha=(1-\theta)\alpha_1+\theta \alpha_2$ and $p=(1-\theta)p_1+\theta p_2,$ then $\mathscr{L}^{p}_{\alpha}(\mathbb{R}^d)$ is the complex interpolation space  \cite[Chapter 6]{interpolation} between $\mathscr{L}^{p_1}_{\alpha_1}(\mathbb{R}^d)$ and $\mathscr{L}^{p_2}_{\alpha_2}(\mathbb{R}^d),$ i.e., $\mathscr{L}^{p}_{\alpha}(\mathbb{R}^d)=\left(\mathscr{L}^{p_1}_{\alpha_1}(\mathbb{R}^d), \mathscr{L}^{p_2}_{\alpha_2}(\mathbb{R}^d)\right)_{[\theta]}.$

We  define a weighted fractional Sobolev space
$$H^{\alpha}_x(\mathbb{R}^{2d}, f_{\infty})\colonequals \{ g\in L^2(\mathbb{R}^{2d}, f_{\infty}): \, 
(1-\Delta_x)^{\frac{\alpha}{2}}(gf^{1/2}_{\infty})\in L^2(\mathbb{R}^{2d})\},\, \, \, \, \alpha \in [0,1].$$
 The corresponding norm is
$$||g||_{ H^{\alpha}_x(\mathbb{R}^{2d},f_{\infty})}\colonequals
||(1-\Delta_x)^{\frac{\alpha}{2}}(gf^{1/2}_{\infty})||_{ L^{2}(\mathbb{R}^{2d})}
.$$
By the Plancherel theorem
$$||g||_{ H^{0}_x(\mathbb{R}^{2d},f_{\infty})}=||g||_{ L^2(\mathbb{R}^{2d},f_{\infty})}$$
and
\begin{equation}\label{H_x^1}||g||_{ H^{1}_x(\mathbb{R}^{2d},f_{\infty})}= \sqrt{\displaystyle \int_{\mathbb{R}^{2d}} g^2f_{\infty}dxdv+\int_{\mathbb{R}^{2d}}|\nabla_x (gf_{\infty}^{1/2})|^2dxdv}.
\end{equation}


We also define $$H^{1}_v(\mathbb{R}^{2d}, f_{\infty})\colonequals \{ g\in L^2(\mathbb{R}^{2d}, f_{\infty}): \, |\nabla_v g| \in  L^2(\mathbb{R}^{2d}, f_{\infty}) \}$$ with the norm
$$||g||_{ H^{1}_v(\mathbb{R}^{2d},f_{\infty})}\colonequals  \sqrt{\displaystyle \int_{\mathbb{R}^{2d}} g^2f_{\infty}dxdv+\int_{\mathbb{R}^{2d}}|\nabla_v g|^2f_{\infty}dxdv}.$$

  Let $\displaystyle  h \colonequals\frac{f-f_{\infty}}{f_{\infty}},$  $ \displaystyle \psi \colonequals \phi-\phi_{\infty}$ and  $ \displaystyle h_0 \colonequals \frac{f_0-f_{\infty}}{f_{\infty}}.$ 
  Then, we  write the system   \eqref{VPFP} as
  \begin{equation}\label{hVPFP}
  \begin{small}
\begin{cases}
\partial_t h+v\cdot\nabla_x h-\nabla_x( V+\phi_{\infty})\cdot \nabla_v h+v\cdot\nabla_x \psi-\sigma \Delta_v h+\nu v\cdot \nabla_v h=\nabla_x \psi\cdot (\nabla_v h-\frac{\nu}{\sigma}v h)  \\
-\Delta_x \psi=\displaystyle\int_{\mathbb{R}^d}hf_{\infty} dv, \, \, \, \, \, \, h_{|t=0}=h_0.  
\end{cases}
\end{small}
\end{equation}
We note that $\displaystyle \int_{\mathbb{R}^{2d}}f_0dxdv=\int_{\mathbb{R}^{2d}}f_{\infty}dxdv=1$ implies $$\int_{\mathbb{R}^{2d}}h_0 f_{\infty}dxdv=0.$$
It is obvious that the existence of a unique solution $f(t)$ to \eqref{VPFP} and the convergence $f(t)\to f_{\infty}$ as $t\to \infty$ are respectively equivalent to the existence of a unique solution $h(t)$ to \eqref{hVPFP} and  the convergence $h(t)\to 0$ as $t\to \infty.$

 The term $\nabla_x \psi\cdot \nabla_v h-\frac{\nu}{\sigma}v\cdot\nabla_x \psi h $ appearing on the right hand side of \eqref{hVPFP} is  nonlinear. If we drop it,  we  obtain the linearized Vlasov-Poisson-Fokker-Planck system  around the steady state $f_{\infty}$
\begin{equation}\label{linVPFP}
\begin{cases}
\partial_t h+v\cdot\nabla_x h-\nabla_x (V+\phi_{\infty})\cdot \nabla_v h+v\cdot\nabla_x \psi-\sigma\Delta_v h+\nu v\cdot \nabla_v h =0\\
-\Delta_x \psi=\displaystyle\int_{\mathbb{R}^d}hf_{\infty} dv,\, \, \, \, \, \, h_{|t=0}=h_0.  
\end{cases}
\end{equation}
We first study this linearized system  in dimension $d\geq 3.$
We prove that the linearized  system \eqref{linVPFP} is well-posed in $C\left([0,\infty); L^2(\mathbb{R}^{2d}, f_{\infty})\right)$ and  has  regularizing properties which is called \textit{hypoellipticity} \cite{Hor}. More precisely, even if the initial data $h_0$ is in $L^2(\mathbb{R}^{2d}, f_{\infty}),$ the solution $h(t)$ is in $H^1(\mathbb{R}^{2d}, f_{\infty})$ for $t>0,$ and we obtain short time estimates for this $L^2(\mathbb{R}^{2d}, f_{\infty})\to H^1(\mathbb{R}^{2d}, f_{\infty})$ regularization. We also prove that the solutions of \eqref{linVPFP} decay exponentially to zero as $t\to \infty$ in
$H^1(\mathbb{R}^{2d}, f_{\infty}):$
\begin{theorem}[\textbf{The linerized Vlasov-Poisson-Fokker-Planck system}]\label{th:lin}
Let $d\geq 3$ and  $h_0 \in L^2(\mathbb{R}^{2d}, f_{\infty})
.$  
\begin{itemize}
\item[(i)] Let
$V\in C^{\infty}(\mathbb{R}^d)$ be bounded from below and
$e^{-\frac{\nu}{\sigma}V}\in L^1(\mathbb{R}^d).$ Then, the system  \eqref{linVPFP} admits a unique mild solution $$h\in C\left([0,\infty); L^2(\mathbb{R}^{2d}, f_{\infty})\right)$$ and $$|\nabla_x \psi| \in  C\left([0,\infty); {L^{\frac{pd}{d-p}}}(\mathbb{R}^d)\right), \, \, \, \, \, \forall\, p\in (1, 2].$$
\item[(ii)]Let the assumptions $(A1)$ and $(A2)$ hold. Then,  for any $t_0>0,$  there are explicitly computable constants $C_1>0$ and $C_2>0$ (independent of $h_0$) such that
\begin{equation}\label{t^3}
\int_{\mathbb{R}^{2d}}|\nabla_x h(t)|^2f_{\infty}dxdv\leq \frac{C_1}{t^3} \int_{\mathbb{R}^{2d}}h^2_0f_{\infty}dxdv
\end{equation} and
\begin{equation}\label{t^2}
\int_{\mathbb{R}^{2d}}|\nabla_v h(t)|^2f_{\infty}dxdv\leq \frac{C_2}{t} \int_{\mathbb{R}^{2d}}h^2_0f_{\infty}dxdv
\end{equation}
hold for all $t \in (0,t_0].$
\item[(iii)] Let  $\displaystyle \int_{\mathbb{R}^{2d}}h_0 f_{\infty} dxdv=0,$ the assumptions $(A1),$ $(A2)$ and $(A3)$ hold. Then, there are explicitly computable  constants $\lambda>0,$    $C_3>0$ and $C_4>0$  (independent of $h_0$) such that
\begin{equation}\label{decay} ||h(t)||_{ H^1(\mathbb{R}^{2d},f_{\infty})}\leq C_3e^{-\lambda t}||h_0||_{ L^2(\mathbb{R}^{2d},f_{\infty})}
\end{equation}
and \begin{equation}\label{psi W}
{||\nabla_x \psi(t)||_{L^{\frac{pd}{d-p}}(\mathbb{R}^d)}}+||\nabla_x \psi(t)||_{W^{1,\frac{2d}{d-2}}(\mathbb{R}^d)} \leq C_4 e^{-\lambda t}||h_0||_{ L^2(\mathbb{R}^{2d},f_{\infty})}
\end{equation}
hold for all $t\geq t_0>0$ {and $p\in (1,2].$}
\end{itemize}
\end{theorem}
\begin{Remark}
\begin{enumerate}
\item We think that Theorem \ref{th:lin}$\,(ii)$  is the first  regularity result for the linearized system.  This  can be considered as a generalization of the regularity results for the linear kinetic Fokker-Planck equation \cite[Theorem 1.1]{Herau}, \cite[Theorem A.8]{article}. 
\item  Theorem \ref{th:lin}$\,(iii)$  extends  the work of Abddala, Dolbeault et al.  in \cite[Theorem 1]{addala2021text}, since they  obtained the exponential decay only in $L^2(\mathbb{R}^{2d}, f_{\infty}).$ 
\item Theorem \ref{th:lin} holds for any parameters $\nu>0$ and $\sigma>0.$ It can be obtained when $x\in \mathbb{T}^d$ by similar computation.  Hence,  Theorem \ref{th:lin} extends the result of Landau damping  for the linearized system in \cite[Theorem 3.1]{Tris} where  $\nu$ and $\sigma$ are required to be small.
\end{enumerate}
\end{Remark}

Next, we pass to the nonlinear Vlasov-Poisson-Fokker-Planck system in dimension $d=3.$
We define operators
$$Kh \colonequals v\cdot\nabla_x h-\nabla_x (V+ \phi_{\infty})\cdot \nabla_v h+v\cdot\nabla_x \psi-\sigma\Delta_v h+\nu v\cdot \nabla_v h$$
and
$$R[h]\colonequals \nabla_x \psi\cdot \nabla_v h-\frac{\nu}{\sigma}v\cdot\nabla_x \psi h. $$
 Since $\nabla_x \psi$  can be expressed by $h$ as
\begin{equation*}\label{grad psi}\nabla_x\psi=\frac{1}{|\mathbb{S}^{d-1}|}\frac{x}{|x|^d}\ast \int_{\mathbb{R}^d} hf_{\infty} dv,
\end{equation*}
 we consider $K$ and $R$ as operators acting only on $h.$
It shows that $K$ is linear and $R$ is nonlinear with respect to $h.$
Then the linearized system \eqref{linVPFP} can be written as
\begin{equation*}
\partial_t h+Kh=0,
\end{equation*}
while the nonlinear system \eqref{hVPFP} can be written as
\begin{equation*}
\partial_t h+Kh=R[h].
\end{equation*}
From Theorem \ref{th:lin}$\,(i)$ we obtain that  $K$ generates a $C_0$ semigroup  $e^{-t K }$ on $L^2(\mathbb{R}^{2d},\\ f_{\infty}).$ Then
the Duhamel principle suggests to convert  this nonlinear system  to an integral equation
\begin{equation}\label{INTe}
h(t)=e^{-t K } h_0+\int_0^t e^{-(t-s) K } R [h(s)] ds.
\end{equation}
We mention that a function $h$ satisfying \eqref{INTe} is called a \textit{mild solutions} to \eqref{hVPFP}, see \cite[Section 6.1]{Pazy}. Using the properties of  $e^{-t K }$ and fixed point arguments we show that there is a unique solution to this integral equation: 

\begin{theorem}[\textbf{Local well-posedness}]\label{Short time}
Let $d=3,$ $ \alpha \in \left(\frac{1}{2}, \frac{2}{3}\right),$ the assumptions $(A1)$ and $ (A2)$ hold. Then, for every  $h_0\in H^{\alpha}_x(\mathbb{R}^{6}, f_{\infty})\cap H^{1}_v(\mathbb{R}^{6}, f_{\infty}),$ there is a $t_{max}\in (0, \infty]$ such that \eqref{hVPFP} has a unique mild solution  $$h\in C \left([0, t_{max}); H_x^{\alpha} (\mathbb{R}^{6}, f_{\infty})\right)\cap C \left([0, t_{max}); H_v^{1} (\mathbb{R}^{6}, f_{\infty})\right)$$
and $$|\nabla_x \psi| \in C \left([0,t_{max}); \mathscr{L}^{6}_{\alpha}(\mathbb{R}^3)\right).$$
Moreover, if $t_{max}<\infty,$ then at least one of the limits $$\lim_{t\nearrow t_{max}}||h(t)||_{ H^{\alpha}_x(\mathbb{R}^{6},f_{\infty})} \, \, \, \,\text{ and } \, \, \, \, \, \, \,  \lim_{t\nearrow t_{max}}||h(t)||_{ H_v^1(\mathbb{R}^{6},f_{\infty})}$$ is infinite.
\end{theorem}
If the initial data $h_0$ is small (i.e. $f_0$ is close to $f_{\infty}),$ then there is a unique global solution and it decays exponentially:
\begin{theorem}[\textbf{Global well-posedness}]\label{th:VPFP}
Let $d=3,$ $ \alpha \in \left(\frac{1}{2}, \frac{2}{3}\right),$  the assumptions $(A1),$ $(A2)$ and $(A3)$ hold. Let  $h_0\in H^{\alpha}_x(\mathbb{R}^{6}, f_{\infty})\cap H^{1}_v(\mathbb{R}^{6}, f_{\infty}), $  $\displaystyle \int_{\mathbb{R}^{6}}h_0 f_{\infty} dxdv=0,$    $$||  h_0||_{H^{\alpha }_x (\mathbb{R}^{6}, f_{\infty})}\leq \delta_1 \, \, \, \, \, \text{and} \, \, \, \, \, ||  h_0||_{H_v^{1 } (\mathbb{R}^{6}, f_{\infty})}
\leq \delta_2$$
for  explicitly computable constants $\delta_1, \, \delta_2>0$ (given in the proof).  Then \eqref{hVPFP} has  a unique global mild solution  $$h\in C \left([0, \infty); H_x^{\alpha} (\mathbb{R}^{6}, f_{\infty})\right)\cap C \left([0, \infty); H_v^{1} (\mathbb{R}^{6}, f_{\infty})\right)$$
and $$|\nabla_x \psi| \in C \left([0,\infty); \mathscr{L}^{6}_{\alpha}(\mathbb{R}^3)\right).$$  Moreover, for any $\lambda_1 \in (0,\lambda),$ there are explicitly computable constants  $C_5>0,$ $C_6>0$ and $C_7>0$ (independent of $h_0,$ but depending on $\delta_1$ and $\delta_2$) such that
$$||  h(t)||_{H_x^{\alpha} (\mathbb{R}^{6}, f_{\infty})}\leq C_5 e^{-\lambda_1 t},$$
$$||  h(t)||_{H_v^{1} (\mathbb{R}^{6}, f_{\infty})}\leq  C_6 e^{-\lambda_1 t}$$
and
$$|| \nabla_x \psi(t)||_{\mathscr{L}^{6}_{\alpha}(\mathbb{R}^3)}\leq C_7 e^{-\lambda_1 t}$$
hold, where $\lambda$ is the decay rate obtained for the linearized system in Theorem \ref{th:lin}.  
\end{theorem}

In Theorem \ref{th:VPFP} we assume the initial data $\displaystyle h_0=(f_0-f_{\infty})/f_{\infty}$ is in a neighborhood of zero in $H^{\alpha}_x(\mathbb{R}^{2d}, f_{\infty}))\cap H^{1}_v(\mathbb{R}^{2d}, f_{\infty}),$ and  the radius of this neighborhood can be estimated explicitly. We have to make this assumption because of the difficulties coming from the nonlinearity of the system. The smallness of $\displaystyle (f_0-f_{\infty})/f_{\infty}$  is a common assumption to study various nonlinear kinetic equations (e.g.  \cite{ BGK, Bedro, Tris, Hwang.Jang, Mou.Neu, Des.Vil, Guo}) and in some cases this assumptions  is necessary. It is often required the smallness of  $\displaystyle (f_0-f_{\infty})/{f_{\infty}}$ in more regular Sobolev spaces, for example in $H^s(\mathbb{R}^{2d}, f_{\infty})$ with $s\geq d.$ While we require this assumption in a larger and less regular space (i.e., $H^{\alpha}_x(\mathbb{R}^{2d}, f_{\infty}))\cap H^{1}_v(\mathbb{R}^{2d}, f_{\infty}))$ and so our result improves the previous works in this respect.   It would be interesting to generalize our results for arbitrary large initial data, away from the steady state. Yet, this extension is not within reach so far and will be a matter of further study.

We believe that Theorem \ref{Short time} and Theorem \ref{th:VPFP} are the first well-posedness and exponential stability results for  a  large class of potentials $V$ (i.e, the potentials satisying Assumption \ref{A1}).   Most of the previous results were obtained when $V=0.$  Our results hold for any parameters $\nu>0 $ and $\sigma>0,$ they do not need to be small or large as in  \cite{Herda, Bedro, Tris}.

\section{Preliminaries}
\subsection{Steady state}
In this section, we show that there is a unique  solution to \eqref{eq.steady state} and we establish some regularity estimates  in the Sobolev spaces.
\begin{lemma}[{\cite[Section 2]{Dolbeault1991STATIONARYSI}}] \label{s.t.1}
 Let $d\geq 3$ and  $e^{-\frac{\nu}{\sigma}V}\in L^1(\mathbb{R}^d).$
Then \eqref{eq.steady state} has a unique solution $\phi_{\infty}\geq 0$ such that $$\phi_{\infty} \in L^{\frac{d}{d-2},\infty}(\mathbb{R}^d)\,\,\, \text{ and } \, \, \, |\nabla\phi_{\infty}|\in L^{\frac{d}{d-1},\infty}(\mathbb{R}^d),$$
where $L^{p,\infty}(\mathbb{R}^d)\colonequals\left\{g \in L^{1}_{loc}(\mathbb{R}^d): \sup_{\lambda> 0}\left[\lambda^p \text{meas}(\{x\in \mathbb{R}^d:\, \, g(x)>\lambda\})\right]< \infty\right\},$  $  p>1.$
\end{lemma}
As we have the existence by Lemma \ref{s.t.1}, we next establish some regularity for $\phi_{\infty}.$ The important tool here is the Hardy-Littlewood-Sobolev inequality:
 \begin{theorem}[{\cite[Theorem 7.25]{giaquinta2013introduction}}]
Let $p,q \in (1,\infty),$ $a\in (0,d)$ such that $\frac{1}{q}-\frac{1}{p}+\frac{a}{d}=0.$ There exists a constant $C_{HLS}>0$ such that, for all $g\in L^p(\mathbb{R}^d),$
\begin{equation*}\label{HLS}
\left|\left|g\ast \frac{1}{|x|^{d-a}}\right|\right|_{L^q(\mathbb{R}^d)}\leq C_{HLS}||g||_{L^p(\mathbb{R}^d)}.
\end{equation*}
\end{theorem}
\begin{lemma}\label{s.t.}
Let $d\geq 3, $  $V$ be bounded from below and $e^{-\frac{\nu}{\sigma}V}\in L^1(\mathbb{R}^d).$
Then the solution $\phi_{\infty}$ of \eqref{eq.steady state} satisfies
\begin{equation*}
\phi_{\infty}\in W^{2,q}(\mathbb{R}^d)
\end{equation*}
 for all $ q \in
 (\frac{d}{d-2},\infty).$
  Moreover, if $V\in C^1(\mathbb{R}^d)$  and
$|\nabla_x V|e^{-\frac{\nu}{\sigma}V}\in L^r(\mathbb{R}^d)$ for some $r\in
 (\frac{d}{d-2},\infty),$ then
\begin{equation*}
\phi_{\infty}\in W^{3,r}(\mathbb{R}^d).
\end{equation*}
 In particular, if $r>d,$ then $\phi_{\infty}\in W^{2,\infty}(\mathbb{R}^d).$
\end{lemma}
\begin{proof}
The boundedness of $V $  from below, $e^{-\frac{\nu}{\sigma}V}\in L^1(\mathbb{R}^d)$ and $\phi_{\infty}\geq 0 $ imply that \begin{equation}\label{in L^p}
e^{-\frac{\nu}{\sigma}[V+\phi_{\infty}]}\in L^p(\mathbb{R}^d), \, \, \, \, \, \, \forall\, p\in [1,\infty]
\end{equation}
and so
\begin{equation}\label{lap.L^p}
 -\Delta_x \phi_{\infty}= \frac{e^{-\frac{\nu}{\sigma}[V+\phi_{\infty}]}}{ \int_{\mathbb{R}^{d}}e^{-\frac{\nu}{\sigma}[V(x')+\phi_{\infty}(x')]}dx' }\in L^p(\mathbb{R}^d), \, \, \, \, \, \, \forall\, p\in [1,\infty].
\end{equation}
We present $\phi_{\infty}$ as
\begin{equation*}
\phi_{\infty}=
\frac{1}{(d-2)|\mathbb{S}^{d-1}|}\frac{1}{|x|^{d-2}}\ast \left(\frac{e^{-\frac{\nu}{\sigma}[V+\phi_{\infty}]}}{ \int_{\mathbb{R}^{d}}e^{-\frac{\nu}{\sigma}[V(x')+\phi_{\infty}(x')]}dx' }\right).
\end{equation*}
\eqref{in L^p} and the Hardy-Littlewood-Sobolev inequality  show
\begin{equation}\label{psi in L^p p>d/d-2}
\phi_{\infty} \in L^q(\mathbb{R}^d), \, \, \,  \, \, \,\forall\,  q\in \left(\frac{d}{d-2},\infty\right).
\end{equation}
\eqref{lap.L^p} and \eqref{psi in L^p p>d/d-2} yield $-\Delta_x \phi_{\infty}+\phi_{\infty} \in L^q(\mathbb{R}^d) \, \, \, \text{ for all } \, \, \, q\in (\frac{d}{d-2},\infty).$ Thus,  the  elliptic regularity \cite[Section 7.2 and Section 7.3]{giaquinta2013introduction}  shows
\begin{equation*}
\phi_{\infty}\in W^{2,q}(\mathbb{R}^d), \, \, \, \, \forall\, q\in \left(\frac{d}{d-2},\infty\right).
\end{equation*}
In particular, by the Sobolev embedding theorem
\begin{equation}\label{bound d}
\phi_{\infty}\in L^{\infty}(\mathbb{R}^d), \, \, \, \, |\nabla_x \phi_{\infty}|\in L^{\infty}(\mathbb{R}^d).
\end{equation}
We use the bootstrap argument. Because of the assumption $|\nabla_x V|e^{-\frac{\nu}{\sigma}V}\in L^r(\mathbb{R}^d)$ and \eqref{bound d}, we have
\begin{equation*}
 -\Delta_x (\partial_{x_i}\phi_{\infty})=- \frac{\frac{\nu}{\sigma}(\partial_{x_i}V+\partial_{x_i}\phi_{\infty})e^{-\frac{\nu}{\sigma}[V+\phi_{\infty}]}}{ \int_{\mathbb{R}^{d}}e^{-\frac{\nu}{\sigma}[V(x')+\phi_{\infty}(x')]}dx' }\in L^r(\mathbb{R}^d)
\end{equation*}
   for all $i\in \{1,...,d\}.$
Again using the elliptic regularity we obtain $\partial_{x_i} \phi_{\infty}\in W^{2,r}(\mathbb{R}^d).$ If $r>d,$ then the Sobolev embedding theorem provides that $\phi_{\infty}\in W^{2,\infty}(\mathbb{R}^d).$

\end{proof}

\subsection{Poincar\'e type inqualities}
In this section, we present some sufficient conditions on the potential $V$ such that $\rho_{\infty}$ satisfies   Poincar\'e type-inequalities.

\begin{lemma}\label{lemma.Poin.}
Let $V$  be bounded from below, $e^{-\frac{\nu}{\sigma}V} \in L^1(\mathbb{R}^d)$  and $e^{-\frac{\nu}{\sigma}V}$ satisfy the Poincar\'e inequality \eqref{Poin.V}. Then,  there exists a positive constant $\kappa_2$ such that
\begin{equation}\label{Poin.inq.}
 \int_{\mathbb{R}^{2d}}h^2f_{\infty}dxdv-\left(\int_{\mathbb{R}^{2d}}hf_{\infty}dxdv\right)^2\leq \kappa_2 \int_{\mathbb{R}^{2d}}(|\nabla_xh|^2 +|\nabla_v h|^2 )f_{\infty}dxdv
\end{equation}
holds for all $ h\in H^1(\mathbb{R}^{2d}, f_{\infty}).$
\end{lemma}
\begin{proof}
 $\phi_{\infty}$ is bounded  by Lemma \ref{s.t.}. Then the Holley-Stroock perturbation argument \cite{Holley1987LogarithmicSI} implies that $\rho_{\infty}$ satisfies the Poincar\'e inequality $$ \int_{\mathbb{R}^d}h^2\rho_{\infty}dx-\left(\int_{\mathbb{R}^d}h\rho_{\infty}dx\right)^2\leq \kappa'_1\int_{\mathbb{R}^d}|\nabla_xh|^2 \rho_{\infty}dx$$ for some constant $\kappa'_1>0.$
Since the Gaussian distribution $
 M(v)=\frac{e^{-\frac{\nu}{\sigma}|v|^2/2}}{ (2\pi \sigma/\nu)^{d/2}}$ satisfies the Poincar\'e inequality,  \cite[Proposition 4.3.1]
{bakry2014analysis} shows that $f_{\infty}=\rho_{\infty}M$ satisfies \eqref{Poin.inq.}.
\end{proof}
\begin{lemma}\label{Vill estimates}
\begin{itemize}
\item[(i)] Let $V$ satisfy the assumptions (A1) and (A2). 
Then there exist $\kappa_3>0$ and $\kappa_4>0$ such that, for all $g \in H^1(\mathbb{R}^d,\rho_{\infty}),$
\begin{equation}\label{Vil.s lemma}
\int_{\mathbb{R}^d}g^2 \left| \left|\frac{\partial^2 (V+\phi_{\infty})}{\partial x^2}\right|\right|_F^2\rho_{\infty}dx\leq \kappa_3 \left(\int_{\mathbb{R}^d} g^2 \rho_{\infty}dx+\int_{\mathbb{R}^d}|\nabla_x  g|^2 \rho_{\infty}dx\right),
\end{equation}
\begin{equation}\label{Vil.s lemma1}
\int_{\mathbb{R}^d} g^2 |\nabla_x (V+\phi_{\infty})|^2\rho_{\infty}dx\leq \kappa_4 \left(\int_{\mathbb{R}^d} g^2 \rho_{\infty}dx+\int_{\mathbb{R}^d}|\nabla_x  g|^2 \rho_{\infty}dx\right).
\end{equation}
\item[(ii)]There exist $\kappa'_4>0$ such that, for all $g \in H^1(\mathbb{R}^d,M),$
\begin{equation}\label{Vil.s lemma2}
\int_{\mathbb{R}^d}|v|^2g^2 Mdv\leq \kappa'_4 \left(\int_{\mathbb{R}^d} g^2 Mdv+\int_{\mathbb{R}^d}|\nabla_v  g|^2 Mdv\right).
\end{equation}
\end{itemize}
\end{lemma}
\begin{proof} $(i)$
We first prove that there exists $c_2>0$ such that  \begin{equation}\label{hess<grad2'}
 \frac{\nu}{\sigma}\left|\left|\frac{\partial^2 (V(x)+\phi_{\infty}(x))}{\partial x^2}\right|\right|_F\leq c_2\left(1+\frac{\nu}{\sigma}|\nabla_x \left(V(x)+\phi_{\infty}(x)\right)|\right), \, \, \, \, \, \forall x \in \mathbb{R}^d.
\end{equation}
Lemma \ref{s.t.} provides   $ \phi_{\infty}\in W^{2,\infty}(\mathbb{R}^d).$ Then, \eqref{hess<grad2'} follows by \eqref{hess<grad2} and the following estimates:
\begin{align*}
 \left| \left|\frac{\partial^2 (V+\phi_{\infty})}{\partial x^2}\right|\right|_F &\leq \left|\left|\frac{\partial^2 \phi_{\infty}}{\partial x^2}\right|\right|_F+ \left|\left|\frac{\partial^2 V}{\partial x^2}\right|\right|_F\\
 &\leq \left|\left|\frac{\partial^2 \phi_{\infty}}{\partial x^2}\right|\right|_F+ c_1(1+|\nabla_x V|)\\
& \leq \left|\left|\frac{\partial^2 \phi_{\infty}}{\partial x^2}\right|\right|_F+c_1|\nabla_x \phi_{\infty}|+c_1+c_1|\nabla_x (V+\phi_{\infty})|\\
& \leq c_2\left(\frac{\sigma}{\nu}+|\nabla_x (V+\phi_{\infty})|\right),
\end{align*}
where $c_2\colonequals \max\left\{c_1,\frac{\nu}{\sigma}\left|\left|\,||\frac{\partial^2 \phi_{\infty}}{\partial x^2}||_F+c_1|\nabla_x \phi_{\infty}|+c_1\,\right|\right|_{L^{\infty}(\mathbb{R}^d)}\right\}.$\\
 Then, \eqref{hess<grad2'} and \cite[Lemma A.24]{article} (by replacing $V$ with $\frac{\nu}{\sigma}[V+\phi_{\infty}]$) provide \eqref{Vil.s lemma} and \eqref{Vil.s lemma1}.

$(ii)$ The proof follows from  \cite[Lemma A.24]{article} (by replacing $V$ with ${|v|^2}/{2}).$
\end{proof}
\begin{lemma}
Let $V$ satisfy the assumptions (A1) and (A2). Then there is a positive constant $\kappa_5$ such that, for all $g \in H^1_x(\mathbb{R}^d,f_{\infty}),$
\begin{equation}\label{norm H_x^1}
||g||_{ H^{1}_x(\mathbb{R}^{2d},f_{\infty})}\leq \kappa_5 \sqrt{\displaystyle \int_{\mathbb{R}^{2d}} g^2f_{\infty}dxdv+\int_{\mathbb{R}^{2d}}|\nabla_x g|^2f_{\infty}dxdv}.
\end{equation}
\end{lemma}
\begin{proof}
 Using \eqref{H_x^1} we estimate
$$||g||^2_{ H^{1}_x(\mathbb{R}^{2d},f_{\infty})}= \displaystyle \int_{\mathbb{R}^{2d}} g^2f_{\infty}dxdv+\int_{\mathbb{R}^{2d}}\left|\nabla_x gf^{1/2}_{\infty}-\frac{\nu}{2\sigma}\nabla_x (V+\phi_{\infty})gf^{1/2}_{\infty}\right|^2dxdv$$
$$\leq \displaystyle \int_{\mathbb{R}^{2d}} g^2f_{\infty}dxdv+2\int_{\mathbb{R}^{2d}}|\nabla_x g|^2f_{\infty}dxdv+\frac{\nu^2}{2\sigma^2}\int_{\mathbb{R}^{2d}}g^2|\nabla_x (V+\phi_{\infty})|^2f_{\infty}dxdv. $$
Applying \eqref{Vil.s lemma1} to the last term  we obtain
\begin{align*}
\displaystyle ||g||^2_{ H^{1}_x(\mathbb{R}^{2d},f_{\infty})}&\leq \left(1+\frac{\kappa_4\nu^2}{2\sigma^2}\right) \int_{\mathbb{R}^{2d}} g^2f_{\infty}dxdv+\left(2+\frac{\kappa_4\nu^2}{2\sigma^2}\right) \int_{\mathbb{R}^{2d}}|\nabla_x g|^2f_{\infty}dxdv
 \\
 &\leq \left(2+\frac{\kappa_4\nu^2}{2\sigma^2}\right)\left( \int_{\mathbb{R}^{2d}} g^2f_{\infty}dxdv+\int_{\mathbb{R}^{2d}}|\nabla_x g|^2f_{\infty}dxdv\right).
 \end{align*}
Thus, \eqref{norm H_x^1} holds with $\kappa_5\colonequals\sqrt{2+\frac{\kappa_4\nu^2}{2\sigma^2}}.$
\end{proof}


\subsection{The Poisson equation}
In this section, we present  some estimates for   the Poisson equation
\begin{equation}\label{Pois.eq.}
-\Delta_x \psi=\int_{\mathbb{R}^d}hf_{\infty}dv.
\end{equation}
We define
$L^p(\mathbb{R}^{2d}, f_{\infty})\colonequals \left\{ g :\mathbb{R}^{2d}\to \mathbb{R}: \, \, \int_{\mathbb{R}^{2d}}|g|^pf_{\infty}dxdv<\infty \right\}$
with the norm
\begin{equation*}||g||_{ L^p(\mathbb{R}^{2d},f_{\infty})}\colonequals \left(\int_{\mathbb{R}^{2d}}|g|^pf_{\infty}dxdv\right)^{1/p}.
\end{equation*}

\begin{lemma}\label{lemma0}
Let $V$ be bounded from below and $e^{-\frac{\nu}{\sigma}V}\in L^1(\mathbb{R}^d).$
\begin{enumerate}[~~i)]
\item[(i)] Let $p\in [1,2]$ and $h\in L^p(\mathbb{R}^{2d},f_{\infty}).$ Then
\begin{equation}\label{est.int h}
\left|\left|\int_{\mathbb{R}^d}hf_{\infty}dv\right|\right|_{L^p(\mathbb{R}^d)}\leq
  \left|\left|\rho_{\infty}\right|\right|^{1-\frac{1}{p}}_{L^{\infty}(\mathbb{R}^d)} ||h||_{L^p(\mathbb{R}^{2d},f_{\infty})}.
\end{equation}
\item[(ii)] Let $h \in L^2(\mathbb{R}^{2d},f_{\infty}).$ Then, we have  $h \in L^p(\mathbb{R}^{2d},f_{\infty})$ for all $p\in [1,2]$ and
\begin{equation}\label{p<2}
||h||_{L^p(\mathbb{R}^{2d}, f_{\infty})}\leq ||h||_{L^2(\mathbb{R}^{2d}, f_{\infty})}.
\end{equation}
\item[(iii)] Let $\alpha\in [0,1]$ 
and $h\in  H^{\alpha}_x(\mathbb{R}^{2d},f_{\infty}).$
If $V$ satisfies the assumption (A2), then there is a constant   $\mathcal{A}>0$ (independent of $h$) such that
\begin{equation}\label{est.int der h}
\left|\left|\int_{\mathbb{R}^d}hf_{\infty}dv\right|\right|_{\mathscr{L}^{2}_{\alpha}(\mathbb{R}^d)}\leq \mathcal{A} ||h||_{ H^{\alpha}_x(\mathbb{R}^{2d},f_{\infty})}.
\end{equation}
\end{enumerate}
\end{lemma}
\begin{proof}  The assumptions on $V$ provide that $f_{\infty}$ is well-defined and bounded.

$(i)$ If $p=1,$ \eqref{est.int h} follows by
\begin{equation*}
\int_{\mathbb{R}^d}\left|\int_{\mathbb{R}^d}hf_{\infty}dv\right|dx\leq \int_{\mathbb{R}^{2d}}|h|f_{\infty}dvdx.
\end{equation*}
If $p\in(1,2],$ the H\"older inequality implies
\begin{align*}
\int_{\mathbb{R}^d}\left|\int_{\mathbb{R}^d}hf_{\infty}dv\right|^pdx&=\int_{\mathbb{R}^d}\left|\int_{\mathbb{R}^d}(hf^{1/p}_{\infty})f^{1-1/p}_{\infty}dv\right|^pdx\\ &\leq
\int_{\mathbb{R}^d}
\left[\left(\int_{\mathbb{R}^d}|h|^pf_{\infty}dv\right)^{1/p}
\left(\int_{\mathbb{R}^d}f_{\infty}dv\right)^{1-1/p}\right]^pdx
\\
&\leq
\left|\left|\rho_{\infty}\right|\right|^{{p-1}}_{L^{\infty}(\mathbb{R}^d)} ||h||^p_{L^p(\mathbb{R}^{2d},f_{\infty})}.
\end{align*}

$(ii)$  The H\"older inequality and $ \int_{\mathbb{R}^{2d}}f_{\infty}dxdv=1$  show
$$||h||_{L^2(\mathbb{R}^{2d}, f_{\infty})}^{p}=\left( \int_{\mathbb{R}^{2d}}(|h|^{p}f^{\frac{p}{2}}_{\infty})^{\frac{2}{p}}dxdv\right)^{\frac{p}{2}}\left( \int_{\mathbb{R}^{2d}}f_{\infty}dxdv\right)^{1-\frac{p}{2}}\geq \int_{\mathbb{R}^{2d}}|h|^{p}f_{\infty}dxdv. $$

$(iii)$
  If $\alpha=0,$ then the Plancherel theorem and \eqref{est.int h} with $p=2$ yield \eqref{est.int der h} with $\mathcal{A}\colonequals \left|\left|\rho_{\infty}\right|\right|^{\frac{1}{2}}_{L^{\infty}(\mathbb{R}^d)} . $  If  $\alpha=1,$ by the Plancherel theorem
  \begin{equation}\label{Plan}
  \left|\left|\int_{\mathbb{R}^d}hf_{\infty}dv\right|\right|^2_{\mathscr{L}^{2}_{1}(\mathbb{R}^d)}=\int_{\mathbb{R}^d}\left|\int_{\mathbb{R}^d}hf_{\infty}dv\right|^2 dx+\int_{\mathbb{R}^d}\left|\nabla_x \int_{\mathbb{R}^d}hf_{\infty}dv\right|^2 dx.
  \end{equation}
  We estimate the second term on the right using $\nabla_x f_{\infty}^{1/2}=-\frac{\nu}{2\sigma}\nabla_x(V+\phi_{\infty})f_{\infty}^{1/2}$ 
\begin{align}\label{h_x}
\int_{\mathbb{R}^d}\left|\nabla_x\int_{\mathbb{R}^d}hf_{\infty}dv\right|^2 dx &=\int_{\mathbb{R}^d}\left|\int_{\mathbb{R}^d}\nabla_x(hf_{\infty}^{1/2})f_{\infty}^{1/2}dv-\frac{\nu}{2\sigma}\int_{\mathbb{R}^d}h \nabla_x(V+\phi_{\infty})f_{\infty}dv\right|^2 dx \nonumber
\\
& \leq 2\int_{\mathbb{R}^d}\left|\int_{\mathbb{R}^d}\nabla_x (hf_{\infty}^{1/2})f_{\infty}^{1/2}dv\right|^2dx+\frac{\nu^2}{2\sigma^2}\int_{\mathbb{R}^d}\left|\int_{\mathbb{R}^d}h\nabla_x(V+\phi_{\infty})f_{\infty}dv\right|^2 dx.
\end{align}
We estimate the first integral in \eqref{h_x}
\begin{align*}\int_{\mathbb{R}^d}\left|\int_{\mathbb{R}^d}\nabla_x(hf_{\infty}^{1/2})f_{\infty}^{1/2}dv\right|^2dx&\leq \int_{\mathbb{R}^d}\left(\int_{\mathbb{R}^d}|\nabla_x(hf_{\infty}^{1/2})|^2dv\right)\left(\int_{\mathbb{R}^d}f_{\infty}dv\right)dx\\
&\leq ||\rho_{\infty}||_{L^{\infty}}\int_{\mathbb{R}^{2d}}|\nabla_x(hf_{\infty}^{1/2})|^2dxdv.
\end{align*}
We estimate the second integral in \eqref{h_x}
\begin{align*}
\int_{\mathbb{R}^d}\left|\int_{\mathbb{R}^d}h\nabla_x(V+\phi_{\infty})f_{\infty}dv\right|^2 dx 
&\leq \int_{\mathbb{R}^d}\left(\int_{\mathbb{R}^d}(hf_{\infty}^{1/2})^2|\nabla_x(V+\phi_{\infty})|^2dv\right)\left(\int_{\mathbb{R}^d}f_{\infty}dv\right) dx\\
&=\int_{\mathbb{R}^{2d}}(hf_{\infty}^{1/2})^2|\nabla_x(V+\phi_{\infty})|^2\rho_{\infty} dxdv.
\end{align*}
These estimates yield
\begin{align*}
\int_{\mathbb{R}^d}\left|\nabla_x\int_{\mathbb{R}^d}hf_{\infty}dv\right|^2 dx &\leq 2||\rho_{\infty}||_{L^{\infty}}\int_{\mathbb{R}^{2d}}|\nabla_x(hf_{\infty}^{1/2})|^2dxdv\\
&\quad +\frac{\nu^2}{2\sigma^2}\int_{\mathbb{R}^{2d}}(hf_{\infty}^{1/2})^2|\nabla_x(V+\phi_{\infty})|^2\rho_{\infty} dxdv.
\end{align*}
Applying \eqref{Vil.s lemma1} to the last term  we obtain
\begin{multline}\label{grad int}
\int_{\mathbb{R}^d}\left|\nabla_x\int_{\mathbb{R}^d}hf_{\infty}dv\right|^2 dx\leq 2||\rho_{\infty}||_{L^{\infty}}\int_{\mathbb{R}^{2d}}|\nabla_x(hf_{\infty}^{1/2})|^2dxdv\\+\frac{\kappa_4\nu^2}{2\sigma^2}\int_{\mathbb{R}^{2d}}h^2f_{\infty}\rho_{\infty}dxdv+\frac{\kappa_4\nu^2}{2\sigma^2}\int_{\mathbb{R}^{2d}}|\nabla_x(hf_{\infty}^{1/2})|^2\rho_{\infty}dxdv\\
\leq \frac{\kappa_4||\rho_{\infty}||_{L^{\infty}}\nu^2}{2\sigma^2}\int_{\mathbb{R}^{2d}}h^2f_{\infty}dxdv+\frac{(4\sigma^2+\kappa_4\nu^2)||\rho_{\infty}||_{L^{\infty}} }{2\sigma^2}\int_{\mathbb{R}^{2d}}|\nabla_x(hf_{\infty}^{1/2})|^2dxdv.
\end{multline}
 \eqref{Plan}, \eqref{est.int h} with $p=2$ and \eqref{grad int} provide
\begin{align*}
  \left|\left|\int_{\mathbb{R}^d}hf_{\infty}dv\right|\right|^2_{\mathscr{L}^{2}_{1}(\mathbb{R}^d)}
  &\leq \left(||\rho_{\infty}||_{L^{\infty}}+\frac{\kappa_4||\rho_{\infty}||_{L^{\infty}}\nu^2}{2\sigma^2} \right)\int_{\mathbb{R}^{2d}}h^2f_{\infty}dxdv
  \\&\, \, \, \, \, \, \, +\frac{(4\sigma^2+\kappa_4\nu^2)||\rho_{\infty}||_{L^{\infty}} }{2\sigma^2} \int_{\mathbb{R}^{2d}}|\nabla_x(hf_{\infty}^{1/2})|^2dxdv
  \\ &\leq \frac{(4\sigma^2+\kappa_4\nu^2)||\rho_{\infty}||_{L^{\infty}} }{2\sigma^2} ||h||^2_{ H^{1}_x(\mathbb{R}^{2d},f_{\infty})}.
  \end{align*}
This estimate  proves \eqref{est.int der h} 
when $\alpha=1.$ By interpolation it holds for all $\alpha\in (0,1).$
\end{proof}
\begin{lemma}\label{lem.1}
Let $d\geq 3,$ $V$ be bounded from below  and $e^{-\frac{\nu}{\sigma}V}\in L^1(\mathbb{R}^d).$
Let $h \in L^2(\mathbb{R}^{2d},f_{\infty})$ and $\psi$ satisfy  \eqref{Pois.eq.}. Then
\begin{enumerate}
\item[(i)]  
There is a positive constant $\theta_1$  such that, for all $p\in (1, 2] ,$ \begin{equation}\label{gr.psi in L^s}
||\nabla_x \psi||_{ L^{\frac{pd}{d-p}}(\mathbb{R}^d)}\leq \theta_1||h||_{ L^{p}(\mathbb{R}^{2d},f_{\infty})}\leq \theta_1||h||_{ L^2(\mathbb{R}^{2d},f_{\infty})}.
\end{equation}
\item[(ii)] If  $h \in  H^{\alpha}_x(\mathbb{R}^{2d},f_{\infty})$ for some $\alpha \in [0,1]$ and $V$ satisfies the assumption (A2), then there is a positive constant $\theta_2$ such that
\begin{equation}\label{gr.psi in Lalpha}
||\nabla_x \psi||_{\mathscr{L}^{\frac{2d}{d-2}}_{\alpha}(\mathbb{R}^d)}\leq \theta_2||h||_{ H^{\alpha}_x(\mathbb{R}^{2d},f_{\infty})}.
\end{equation}
Moreover, if $d=3$ and $\alpha \in (\frac{1}{2},1],$ then there is a positive constant $\theta_3$ such that
\begin{equation}\label{gr.psi in L^infty}
||\nabla_x \psi||_{L^{\infty}(\mathbb{R}^3)}
\leq \theta_3||h||_{ H^{\alpha}_x(\mathbb{R}^{6},f_{\infty})}.
\end{equation}
  \end{enumerate}
\end{lemma}
\begin{proof}
 $(i)$ Applying the Hardy-Littlewood-Sobolev inequality  to the right hand sight of
$$\displaystyle |\nabla_x\psi|=\frac{1}{|\mathbb{S}^{d-1}|}\left|\frac{x}{|x|^d}\ast \int_{\mathbb{R}^d}hf_{\infty}dv\right|\leq \frac{d}{|\mathbb{S}^{d-1}|} \frac{1}{|x|^{d-1}}\ast \left|\int_{\mathbb{R}^d}hf_{\infty}dv\right|,$$
we obtain that there is $C>0$ such that
$$||\nabla_x \psi||_{ L^{\frac{pd}{d-p}}(\mathbb{R}^d)}\leq C\left|\left|\int_{\mathbb{R}^d}hf_{\infty}dv\right|\right|_{ L^{p}(\mathbb{R}^{d})}$$ holds for all $p\in (1, 2].$ Then, \eqref{est.int h} and \eqref{p<2} implies \eqref{gr.psi in L^s}.

$(ii)$ \eqref{Pois.eq.} shows 
 $$-(1-\Delta_x)^{\frac{\alpha}{2}}\Delta_x \psi=-\Delta_x \left((1-\Delta_x)^{\frac{\alpha}{2}}\psi\right)=(1-\Delta_x)^{\frac{\alpha}{2}} \int_{\mathbb{R}^d}hf_{\infty}dv. $$ Applying the Hardy-Littlewood-Sobolev inequality with $p=2$ to the right hand sight of
\begin{align*}\displaystyle |\nabla_x\left((1-\Delta_x)^{\frac{\alpha}{2}}\psi\right)|&=\frac{1}{|\mathbb{S}^{d-1}|}\left|\frac{x}{|x|^d}\ast\left[ (1-\Delta_x)^{\frac{\alpha}{2}}\int_{\mathbb{R}^d}hf_{\infty}dv\right]\right|\\
&\leq \frac{d}{|\mathbb{S}^{d-1}|} \frac{1}{|x|^{d-1}}\ast \left|(1-\Delta_x)^{\frac{\alpha}{2}}\int_{\mathbb{R}^d}hf_{\infty}dv\right|,\end{align*} we get 
\begin{multline*}
\left|\left|\nabla_x \left((1-\Delta_x)^{\frac{\alpha}{2}}\psi\right)\right|\right|_{ L^{\frac{2d}{d-2}}(\mathbb{R}^d)}=||\nabla_x \psi||_{\mathscr{L}^{\frac{2d}{d-2}}_{\alpha}(\mathbb{R}^d)}\\
\leq C\left|\left|(1-\Delta_x)^{\frac{\alpha}{2}}\int_{\mathbb{R}^d}hf_{\infty}dv \right|\right|_{ L^{2}(\mathbb{R}^{d})}=C\left|\left|\int_{\mathbb{R}^d}hf_{\infty}dv \right|\right|_{ \mathscr{L}^{2}_{\alpha}(\mathbb{R}^{d})}.
\end{multline*}
Then, \eqref{gr.psi in Lalpha} follows by \eqref{est.int der h}.

Let $d=3$ and $\alpha \in (\frac{1}{2},1].$ Since $\frac{2d\alpha}{d-2}=6\alpha>3$ for $\alpha\in (1/2, 1],$ the Sobolev embedding \cite[Theorem 1.2.4]{Adams} provides \eqref{gr.psi in L^infty}.

\end{proof}


\section{The linearized Vlasov-Poisson-Fokker-Planck system}
 In this section, we analyze the linearized Vlasov-Poisson-Fokker-Planck system \eqref{linVPFP}. We  first show existence and uniqueness. 
\subsection{Existence and uniqueness}
 We write the linearized system \eqref{linVPFP}  as
\begin{equation*}\begin{cases}
\partial_t h+Kh=0\\
 h_{|t=0}=h_0,
\end{cases}
\end{equation*}
where $Kh \colonequals v\cdot\nabla_x h-\nabla_x (V+\phi_{\infty})\cdot \nabla_v h+v\cdot\nabla_x \psi-\sigma\Delta_v h+\nu v\cdot \nabla_v h.$
 Clearly, $K$ depends on  $\nabla_x \psi.$ But we consider $K$ as an operator acting only on $h,$  since $\nabla_x \psi$  can be expressed by $h$ as
\begin{equation*}\nabla_x\psi=\frac{1}{|\mathbb{S}^{d-1}|}\frac{x}{|x|^d}\ast \int_{\mathbb{R}^d} hf_{\infty} dv.
\end{equation*}
\begin{theorem}\label{th:linEXIS}
 Let
$V\in C^{\infty}(\mathbb{R}^d)$ be bounded from below and
$e^{-\frac{\nu}{\sigma}V}\in L^1(\mathbb{R}^d).$ Then  $K$ generates a $C_0$ semigroup $e^{-t K}$ on $L^2(\mathbb{R}^{2d}, f_{\infty}).$ In particular, for any $h_0\in L^2(\mathbb{R}^{2d}, f_{\infty}),$ the linearized system  \eqref{linVPFP} has a unique mild solution
$$h \in C\left([0,\infty); L^2(\mathbb{R}^{2d}, f_{\infty})\right)$$ and $$|\nabla_x \psi| \in  C\left([0,\infty); {L^{\frac{pd}{d-p}}}(\mathbb{R}^d)\right), \, \, \, \,  \, \, \,\forall\,  p\in (1, 2].$$
\end{theorem}
 \begin{proof}
 By Lemma \ref{s.t.}, we have  $\phi_{\infty}\in W^{2,q}(\mathbb{R}^d)$ for all $q \in (\frac{d}{d-2}, \infty),$ and so $\phi_{\infty}\in C^{1}(\mathbb{R}^d).$ Since we have $V\in C^{\infty}(\mathbb{R}^d),$ we can show $\phi_{\infty}\in C^{\infty}(\mathbb{R}^d)$ by a bootstrap argument, see \cite[Theorem 5.20]{giaquinta2013introduction}.
 We consider the following  equation without the Poisson equation
 \begin{equation}\label{kinFP}
\begin{cases}
\partial_t h+v\cdot\nabla_x h-\nabla_x (V+\phi_{\infty})\cdot \nabla_v h-\sigma \Delta_v h+\nu v\cdot \nabla_v h=0\\
h_{|t=0}=h_0.
\end{cases}
\end{equation}
Let $Lh \colonequals v\cdot\nabla_x h-\nabla_x (V+\phi_{\infty})\cdot \nabla_v h-\sigma\Delta_v h+\nu v\cdot \nabla_v h.$ Then \eqref{kinFP} can be written as \begin{equation*}
\begin{cases}
\partial_t h+Lh=0\\
h_{|t=0}=h_0.
\end{cases}
\end{equation*}
Since we have $V+\phi_{\infty}\in C^{\infty}(\mathbb{R}^d),$  \cite[Section 5.2]{HelNi} shows that $L$ generates a $C^{\infty}$ regularizing contraction semigroup in
$L^2(\mathbb{R}^{2d}, f_{\infty} ).$
 $K$ differs from $L$ in the term $v\cdot\nabla_x \psi$ coming from the Poisson equation.  \eqref{gr.psi in L^s} with $p=\frac{2d}{d+2}$ shows
\begin{align*}
\int_{\mathbb{R}^{2d}}|v\cdot \nabla_x \psi|^2f_{\infty}dxdv&\leq \left(\int_{\mathbb{R}^{d}}|v|^2M(v)dv\right)||\rho_{\infty}||_{L^{\infty}}\int_{\mathbb{R}^{d}}| \nabla_x \psi|^2dx\\
&\leq \theta_1^2||\rho_{\infty}||_{L^{\infty}} \left(\int_{\mathbb{R}^{d}}|v|^2M(v)dv\right) \int_{\mathbb{R}^{2d}}h^2f_{\infty}dxdv.
\end{align*}
Therefore, we can consider $h
\to v\cdot \nabla_x \psi$ as a bounded operator from $L^2(\mathbb{R}^{2d},f_{\infty})$ to $L^2(\mathbb{R}^{2d},f_{\infty}).$
 This implies that $K$ is a bounded perturbation of $L$ in $L^2(\mathbb{R}^{2d}, f_{\infty}).$  Then \cite[Chapter 3]{Pazy} provides that $K$ generates a $C_0$ semigroup $e^{-t K} $ on $L^2(\mathbb{R}^{2d}, f_{\infty}).$ Also,  $e^{-t K}h_0 \in C\left([0,\infty); L^2(\mathbb{R}^{2d}, f_{\infty})\right)$ is the unique mild solution to \eqref{linVPFP} by \cite[Chapter 4]{Pazy}.
Then   \eqref{gr.psi in L^s} implies that the absolute value of
$$\nabla_x\psi(t)=\frac{1}{|\mathbb{S}^{d-1}|}\frac{x}{|x|^d}\ast \int_{\mathbb{R}^d} e^{-t K}h_0\,f_{\infty} dv$$
is in $L^{ \frac{pd}{d-p} }(\mathbb{R}^d)$ for all $p\in (1, 2].$ Moreover,    $ ||\nabla_x \psi(t)||_{L^{ \frac{pd}{d-p} }(\mathbb{R}^d)}$ is continuous function of $t$ as $e^{-t K}h_0 \in C\left([0,\infty); L^2(\mathbb{R}^{2d}, f_{\infty})\right).$   

 \end{proof}
\subsection{Exponential stability in $H^1(\mathbb{R}^{2d}, f_{\infty})$}
 In this subsection, we shall construct a Lypunov functional for  the linearized system \eqref{linVPFP}. This functional will help us to show that the solutions of \eqref{linVPFP} are exponentially stable in $H^1(\mathbb{R}^{2d}, f_{\infty}).$

We introduce  a  norm $$||h||^2\colonequals\int_{\mathbb{R}^{2d}}h^2f_{\infty}dxdv+\int_{\mathbb{R}^{d}}|\nabla_x \psi|^2dx.$$

\begin{lemma}\label{lemma ||h||^2}
Let $h$ be the solution of \eqref{linVPFP}. Then, for all $t>0,$
\begin{equation*}
\frac{d}{dt}||h(t)||^2=-2\sigma\int_{\mathbb{R}^{2d}}|\nabla_v h|^2f_{\infty}dxdv.
\end{equation*}
In particular, we have $||h(t)||\leq ||h_0||$ for all $t\geq 0.$
\end{lemma}
\begin{proof}
First, we compute
\begin{align*}
\frac{d}{dt} \int_{\mathbb{R}^{2d}}h^2f_{\infty}dxdv&=2\int_{\mathbb{R}^{2d}}h\partial_t h f_{\infty}dxdv\\
&=-2\int_{\mathbb{R}^{2d}}\left(v\cdot \nabla_x h-\nabla_x(V+\phi_{\infty})\cdot \nabla_v h\right)hf_{\infty}dxdv\\&\, \, \, \, \, \, +2\int_{\mathbb{R}^{2d}}(\sigma\Delta_v h-\nu v\cdot \nabla_v h)h f_{\infty}dxdv\\
&\quad-2 \int_{\mathbb{R}^{2d}}v \cdot \nabla_x \psi hf_{\infty}dxdv.
\end{align*}
 Integrating by parts we get
 \begin{multline*}
 -2\int_{\mathbb{R}^{2d}}\left(v\cdot \nabla_x h-\nabla_x(V+\phi_{\infty})\cdot \nabla_v h\right)hf_{\infty}dxdv\\=-\int_{\mathbb{R}^{2d}}\left(v\cdot \nabla_x h^2-\nabla_x(V+\phi_{\infty})\cdot \nabla_v h^2\right)f_{\infty}dxdv=0,
 \end{multline*}
 and
 \begin{equation*}
2\int_{\mathbb{R}^{2d}}(\sigma\Delta_v h-\nu v\cdot \nabla_v h)h f_{\infty}dxdv=-2\sigma\int_{\mathbb{R}^{2d}}|\nabla_v h|^2 f_{\infty}dxdv.
 \end{equation*}
 Hence \begin{equation}\label{dt h^2}
\frac{d}{dt} \int_{\mathbb{R}^{2d}}h^2f_{\infty}dxdv=-2\sigma\int_{\mathbb{R}^{2d}}|\nabla_v h|^2 f_{\infty}dxdv-2 \int_{\mathbb{R}^{2d}}v \cdot \nabla_x \psi hf_{\infty}dxdv.
\end{equation}
 Secondly, we compute
\begin{align}\label{dt psi}
\frac{d}{dt}\int_{\mathbb{R}^{d}}|\nabla_x \psi|^2dx =&2\int_{\mathbb{R}^{d}}\nabla_x (\partial_t \psi) \cdot \nabla_x \psi dx\notag\\=&-2\int_{\mathbb{R}^{d}} \psi \Delta_x (\partial_t \psi)dx=2\int_{\mathbb{R}^{d}} \psi  \left(\int_{\mathbb{R}^{d}} \partial_t h f_{\infty}dv\right)dx\notag\\\notag
=&2\int_{\mathbb{R}^{2d}}\psi\left[\nabla_x (V+\phi_{\infty})\cdot \nabla_v h-v\cdot \nabla_x h+\sigma\Delta_v h-\nu v\cdot \nabla_v h-v\cdot \nabla_x \psi\right]f_{\infty}dxdv\\
=&2\int_{\mathbb{R}^{2d}}v \cdot \nabla_x \psi hf_{\infty}dxdv {-2\int_{\mathbb{R}^{2d}}\psi \nabla_x \psi \cdot v f_{\infty}dxdv}=2\int_{\mathbb{R}^{2d}}v \cdot \nabla_x \psi hf_{\infty}dxdv,
\end{align}
where we integrated by parts and used $\nabla_x f_{\infty}=-\frac{\nu}{\sigma}\nabla_x(V+\phi_{\infty})f_{\infty}$ and $\nabla_vf_{\infty}=-\frac{\nu}{\sigma}vf_{\infty}.$\\
\eqref{dt h^2} and \eqref{dt psi} provide the claimed equality.
\end{proof}
Let  $P \in \mathbb{R}^{2n \times 2n}$ be a constant, symmetric, positive definite matrix.  We define
\begin{align*} 
&\mathrm{S}_P[h]\colonequals \int_{\mathbb{R}^{2d}}\begin{pmatrix}
\nabla_x (h+\psi)\\ \nabla_v (h+\psi)
\end{pmatrix}^T  P\begin{pmatrix}
\nabla_x (h+\psi)\\
\nabla_v (h+\psi)
\end{pmatrix} f_{\infty}dxdv\\ &\quad\quad\quad =\int_{\mathbb{R}^{2d}}\begin{pmatrix}
\nabla_x (h+\psi)\\ \nabla_v h
\end{pmatrix}^T  P\begin{pmatrix}
\nabla_x (h+\psi)\\\nabla_v h
\end{pmatrix} f_{\infty}dxdv.
\end{align*}
\begin{lemma}\label{lemma main}
Let $h$ be the  solution of \eqref{linVPFP}. Then, for all $t>0,$
\begin{align}\label{derivS}
\frac{d}{dt} \mathrm{S}_P[h(t)]=&-2\sigma\int_{\mathbb{R}^{2d}}\left\{\sum_{i=1}^d \begin{pmatrix}
\nabla_x ( \partial_{v_i} h)\\ \nabla_v(\partial_{v_i} h)
\end{pmatrix}^T  P\begin{pmatrix}
\nabla_x (\partial_{v_i} h)\\\nabla_v(\partial_{v_i} h)
\end{pmatrix} \right\}f_{\infty}dxdv\notag\\&-\int_{\mathbb{R}^{2d}}  \begin{pmatrix}
\nabla_x (h+\psi)\\ \nabla_v h
\end{pmatrix}^T  \left\{QP+PQ^T\right\} \begin{pmatrix}
\nabla_x (h+\psi)\\ \nabla_v h
\end{pmatrix} f_{\infty}dxdv\notag\\&
-2\int_{\mathbb{R}^{2d}} \begin{pmatrix}
\nabla_x (h+\psi)\\ \nabla_v h
\end{pmatrix}^T P \begin{pmatrix} \nabla_x \partial_t  \psi\\0
\end{pmatrix} f_{\infty}dxdv,
\end{align}
where  $Q=Q(x)\colonequals \begin{pmatrix}
0&I\\
-\frac{\partial^2 (V(x)+\phi_{\infty}(x))}{\partial x^2}&\nu I
\end{pmatrix}.$
\end{lemma}
\begin{proof} Since $\psi$ does not depends on $v,$ we write \eqref{linVPFP} as
\begin{equation*}
\partial_t (h+\psi)=-v\cdot\nabla_x (h+\psi)+\nabla_x (V+ \phi_{\infty})\cdot \nabla_v (h+\psi)+\sigma\Delta_v (h+\psi)-\nu v\cdot \nabla_v (h+\psi)+\partial_t \psi. \end{equation*}
We denote $u\colonequals \begin{pmatrix}
\nabla_x (h+\psi)\\ \nabla_v (h+\psi)
\end{pmatrix}, $ $\displaystyle u_1\colonequals \nabla_x (h+\psi),$ $\displaystyle u_2\colonequals \nabla_v (h+\psi).$ Then   $u_1$ and $u_2$ satisfy
\begin{multline*}\partial_t u_1= \sigma\Delta_v u_1-\nu \sum_{i=1}^d v_i  \partial_{v_i}u_1+\sum_{i=1}^d \partial_{x_i} (V
+\phi_{\infty}) \partial_{v_i} u_1-\sum_{i=1}^d v_i \partial_{x_i} u_1\\+\frac{\partial^2 (V
+\phi_{\infty})}{\partial x^2}u_2+\nabla_x \partial_t  \psi,
\end{multline*}
\begin{equation*}
 \partial_t u_2= \sigma \Delta_v u_2-\nu \sum_{i=1}^d v_i  \partial_{v_i}u_2+\sum_{i=1}^d \partial_{x_i} (V
+\phi_{\infty}) \partial_{v_i} u_2-\sum_{i=1}^d v_i \partial_{x_i} u_2-u_1- \nu u_2.
 \end{equation*}
 These equations can be written  with respect to $u=\begin{pmatrix}
 u_1\\u_2
\end{pmatrix}:  $
 \begin{equation*}\partial_t u= \sigma\Delta_v u-\nu \sum_{i=1}^d v_i  \partial_{v_i}u+\sum_{i=1}^d \partial_{x_i} (V
+\phi_{\infty}) \partial_{v_i} u-\sum_{i=1}^d v_i \partial_{x_i} u-Q^Tu+\begin{pmatrix}
\nabla_x \partial_t \psi\\0
\end{pmatrix}.
\end{equation*}
 It allows us to compute the time derivative
   \begin{align}\label{time.derS}
\frac{d}{dt} \mathrm{S}_P[h(t)]\notag=&2\int_{\mathbb{R}^{2d}}
  u^T P  \partial_t u
 f_{\infty}dxdv\notag\\=& 2\sigma\int_{\mathbb{R}^{2d}} u^T P
 \Delta_v u  f_{\infty}dxdv-2 \nu \sum_{i=1}^d\int_{\mathbb{R}^{2d}} u^T P
   \partial_{v_i}u v_i f_{\infty}dxdv\notag\\&+ 2\sum_{i=1}^d \int_{\mathbb{R}^{2d}} u^TP
 \partial_{v_i} u \partial_{x_i} (V
+\phi_{\infty}) f_{\infty}dxdv-2\sum_{i=1}^d\int_{\mathbb{R}^{2d}} u^TP
   \partial_{x_i}u v_i f_{\infty}dxdv\notag\\&-\int_{\mathbb{R}^{2d}} u^T \{QP+PQ^T\}u f_{\infty}dxdv-2\int_{\mathbb{R}^{2d}} u^T P\begin{pmatrix}
\nabla_x \partial_t  \psi\\0
\end{pmatrix} f_{\infty}dxdv.
 \end{align}
  First, we consider the term in the second line of \eqref{time.derS} and use $\partial_{v_i}f_{\infty}=-\frac{\nu}{\sigma}v_if_{\infty}:$
  \begin{multline}\label{term1}
   2 \sigma\sum_{i=1}^d\int_{\mathbb{R}^{2d}} u^TP
  \partial^2_{v_iv_i} u f_{\infty}dxdv -2 \nu \sum_{i=1}^d\int_{\mathbb{R}^{2d}} u^TP
  \partial_{v_i}u  v_i f_{\infty}dxdv\\=-2 \sigma \sum_{i=1}^d \int_{\mathbb{R}^{2d}}\partial_{v_i}u^TP\partial_{v_i}u f_{\infty}dxdv.
 \end{multline}
  Next, we consider the terms in the third line of \eqref{time.derS}: \begin{multline}\label{third term1}
  2\sum_{i=1}^d \int_{\mathbb{R}^{2d}} u^TP
 \partial_{v_i} u \partial_{x_i}(V
+\phi_{\infty}) f_{\infty}dxdv\\= -2\sum_{i=1}^d \int_{\mathbb{R}^{2d}} u^TP
 \partial_{v_i} u \partial_{x_i} (V
+\phi_{\infty}) f_{\infty}dxdv\\+{\frac{2\nu}{\sigma}}\sum_{i=1}^d \int_{\mathbb{R}^{2d}} u^TP
 u  \partial_{x_i} (V
+\phi_{\infty}) v_i f_{\infty}dxdv
 \end{multline}
 and
 \begin{multline}\label{third term2}
 -2\sum_{i=1}^d\int_{\mathbb{R}^{2d}} u^TP
 \partial_{x_i}u  v_i  f_{\infty}dxdv\\=2\sum_{i=1}^d\int_{\mathbb{R}^{2d}} u^TP
   \partial_{x_i}u v_i f_{\infty}dxdv-{\frac{2\nu}{\sigma}}\sum_{i=1}^d\int_{\mathbb{R}^{2d}} u^TP
   u \partial_{x_i}(V
+\phi_{\infty}) v_i f_{\infty}dxdv.
 \end{multline}
 \eqref{third term1} and \eqref{third term2} show that the third line of  \eqref{time.derS} equals to zero.
 Combining \eqref{time.derS} and \eqref{term1},  we obtain the statement \eqref{derivS}.
\end{proof}
\begin{lemma}
Let $h$ be the  solution of \eqref{linVPFP} and $\psi$ be the solution of $-\Delta_x \psi=\int_{\mathbb{R}^d}hf_{\infty}dv.$ 
 Then, for all $t>0,$
\begin{equation}\label{psi tx}
\int_{\mathbb{R}^{2d}}|\nabla_x \partial_t  \psi|^2f_{\infty}dxdv\leq \frac{\sigma^2||\rho_{\infty}||^2_{L^{\infty}}}{\nu^2}\int_{\mathbb{R}^{2d}}|\nabla_v h|^2f_{\infty}dxdv.
\end{equation}
\end{lemma}
\begin{proof}
We compute
\begin{align*}
-\Delta_x(\partial_t  \psi)&=\int_{\mathbb{R}^d}\partial_t hf_{\infty}dv\\
&=\int_{\mathbb{R}^d}[-v\cdot\nabla_x h+\nabla_x (V+ \phi_{\infty})\cdot \nabla_v h\\
&\quad +\sigma\Delta_v h-\nu v\cdot \nabla_v h-v\cdot\nabla_x \psi]f_{\infty}dv\\
&=\int_{\mathbb{R}^d}[-v\cdot\nabla_x h+\nabla_x (V+ \phi_{\infty})\cdot \nabla_v h]f_{\infty}dv\\
&=\int_{\mathbb{R}^d}[-v\cdot\nabla_x h+\frac{\nu}{\sigma} v\cdot \nabla_x (V+ \phi_{\infty})h]f_{\infty}dv\\&=-\text{div}_x \int_{\mathbb{R}^d}v hf_{\infty}dv
=-\frac{\sigma}{\nu}\text{div}_x \int_{\mathbb{R}^d}\nabla_v hf_{\infty}dv,
\end{align*}
where we integrated by parts and used $\nabla_v f_{\infty}=-\frac{\nu}{\sigma}vf_{\infty}.$
It lets us compute
\begin{align*}
\int_{\mathbb{R}^{2d}}| \nabla_x \partial_t \psi|^2f_{\infty}dxdv &=\int_{\mathbb{R}^{d}}|\nabla_x  \partial_t \psi|^2\rho_{\infty}dx\leq ||\rho_{\infty}||_{L^{\infty}}\int_{\mathbb{R}^{d}}|\nabla_x \partial_t  \psi|^2dx\\&=-||\rho_{\infty}||_{L^{\infty}}\int_{\mathbb{R}^{d}}\partial_t  \psi \Delta_x (\partial_t \psi)dx
\\&=-\frac{\sigma ||\rho_{\infty}||_{L^{\infty}}}{\nu}\int_{\mathbb{R}^{2d}}\partial_t  \psi \text{div}_x(\nabla_v hf_{\infty})dxdv\\
&=\frac{\sigma ||\rho_{\infty}||_{L^{\infty}}}{\nu}\int_{\mathbb{R}^{2d}}  \nabla_x \partial_t \psi \cdot \nabla_v hf_{\infty}dxdv \\
&\leq \frac{1}{2}\int_{\mathbb{R}^{2d}}| \nabla_x \partial_t \psi|^2f_{\infty}dxdv+\frac{\sigma^2 ||\rho_{\infty}||^2_{L^{\infty}}}{2\nu^2}\int_{\mathbb{R}^{2d}}|\nabla_v h|^2f_{\infty}dxdv.
\end{align*}
By simplifying this inequality we get \eqref{psi tx}.
\end{proof}

Let $\gamma> 0.$  We consider  a functional
\begin{align*}
\mathrm{E}[h]& \colonequals \gamma||h||^2+\mathrm{S}_P[h]\\
&\, \, \, =\gamma\left[ \int_{\mathbb{R}^{2d}}h^2f_{\infty}dxdv+\int_{\mathbb{R}^{d}}|\nabla_x \psi|^2dx\right]
+\int_{\mathbb{R}^{2d}}
\begin{pmatrix}\nabla_x (h+\psi)\\ \nabla_v h\end{pmatrix}^T  P\begin{pmatrix}\nabla_x (h+\psi)\\\nabla_v h\end{pmatrix} f_{\infty}dxdv.
\end{align*}
It is clear that $\mathrm{E}$ depends on the parameter $\gamma$ and the matrix $P,$ we will fix them later. We show that $\mathrm{E}$ is equivalent to the $H^1-$norm.
\begin{lemma}\label{lem:E equv. H1} Let $V$ be bounded from below and $e^{-\frac{\nu}{\sigma}V}\in L^1(\mathbb{R}^d).$   Let  $p_1$ and $p_2$ be the smallest  and the largest eigenvalues of $P,$ respectively. Then, for all  $h\in H^1(\mathbb{R}^{2d},f_{\infty}),$
\begin{equation}\label{E equiv}
\frac{\mathrm{E}[h]}{ \max\left\{\gamma+\theta^2_1(\gamma+2p_2||\rho_{\infty}||_{L^{\infty}}),\, 2p_2\right\}}
\leq ||h||^2_{H^1(\mathbb{R}^{2d},f_{\infty})}
\leq \frac{\mathrm{E}[h]}{\min\left\{\gamma,\, \frac{\gamma p_1}{\gamma+p_1||\rho_{\infty}||_{L^{\infty}}}\right\}},
\end{equation}
where $\theta_1$ is the constants appearing in \eqref{gr.psi in L^s}.
Moreover, if $e^{-\frac{\nu}{\sigma}V}$ satisfies the Poincar\'e inequality \eqref{Poin.V} and $\displaystyle \int_{\mathbb{R}^{2d}}hf_{\infty}dxdv=0,$ then
\begin{equation}\label{E<S}
 \mathrm{S}_P[h]\leq \mathrm{E}[h]\leq \frac{p_1+\gamma \kappa_2}{p_1}\mathrm{S}_P[h],
\end{equation}
where  $\kappa_2$ is the constants appearing in  \eqref{Poin.inq.}.
\end{lemma}
\begin{proof}
 As  $P$ is positive definite, we have $0<p_1\leq  p_2 .$ We estimate $\mathrm{E}[h]$ from above by using   $ P\leq  p_2 I:$
\begin{align*}
\mathrm{E}[h]&\leq \gamma\left[ \int_{\mathbb{R}^{2d}}h^2f_{\infty}dxdv+\int_{\mathbb{R}^{d}}|\nabla_x \psi|^2dx\right]\\&\, \, \, \, \, \, + p_2\left[\int_{\mathbb{R}^{2d}}
|\nabla_x (h+\psi)|^2 f_{\infty}dxdv+\int_{\mathbb{R}^{2d}}
|\nabla_v h|^2 f_{\infty}dxdv\right]\\
&\leq \gamma \int_{\mathbb{R}^{2d}}h^2f_{\infty}dxdv+(\gamma+2p_2||\rho_{\infty}||_{L^{\infty}})\int_{\mathbb{R}^{d}}|\nabla_x \psi|^2dx\\
&\, \, \, \, \, \, + 2p_2\int_{\mathbb{R}^{2d}}
|\nabla_x h|^2 f_{\infty}dxdv+p_2\int_{\mathbb{R}^{2d}}
|\nabla_v h|^2 f_{\infty}dxdv.
\end{align*}
We use \eqref{gr.psi in L^s} with $p=\frac{2d}{d+2}$ to get
\begin{align}\label{Eee1}
\mathrm{E}[h]
\leq \max\{\gamma+\theta^2_1(\gamma+2p_2||\rho_{\infty}||_{L^{\infty}}), 2p_2\}||h||^2_{H^1(\mathbb{R}^{2d},f_{\infty})}.
\end{align}
We estimate $\mathrm{E}[h]$ from below by using $p_1 I\leq P$
  \begin{align*}
\mathrm{E}[h]\geq &\gamma\left[ \int_{\mathbb{R}^{2d}}h^2f_{\infty}dxdv+\int_{\mathbb{R}^{d}}|\nabla_x \psi|^2dx\right] \nonumber\\&+ p_1\left[\int_{\mathbb{R}^{2d}}
|\nabla_x (h+\psi)|^2 f_{\infty}dxdv+\int_{\mathbb{R}^{2d}}
|\nabla_v h|^2 f_{\infty}dxdv\right].
\end{align*}
By the H\"older inequality
\begin{align*}\int_{\mathbb{R}^{2d}}
|\nabla_x (h+\psi)|^2 f_{\infty}dxdv\geq& \frac{\gamma}{\gamma+p_1||\rho_{\infty}||_{L^{\infty}}}\int_{\mathbb{R}^{2d}}
|\nabla_x h|^2 f_{\infty}dxdv\\ &-\frac{\gamma}{p_1||\rho_{\infty}||_{L^{\infty}}}\int_{\mathbb{R}^{2d}}
|\nabla_x \psi|^2 f_{\infty}dxdv\\
\geq& \frac{\gamma}{\gamma+p_1|\rho_{\infty}||_{L^{\infty}}}\int_{\mathbb{R}^{2d}}|\nabla_x h|^2 f_{\infty}dxdv-\frac{\gamma}{p_1}\int_{\mathbb{R}^{d}}|\nabla_x \psi|^2 {dx}.
\end{align*}
Using the last two estimates
\begin{align}\label{Eee2}
\mathrm{E}[h]\geq& \gamma \int_{\mathbb{R}^{2d}}h^2f_{\infty}dxdv+ \frac{\gamma p_1}{\gamma+p_1||\rho_{\infty}||_{L^{\infty}}}\int_{\mathbb{R}^{2d}}
|\nabla_x h|^2 f_{\infty}dxdv+p_1 \int_{\mathbb{R}^{2d}}
|\nabla_v h|^2 f_{\infty}dxdv\notag\\
\geq & \min\left\{\gamma,\, \frac{\gamma p_1}{\gamma+p_1||\rho_{\infty}||_{L^{\infty}}}\right\}\left[
\int_{\mathbb{R}^{2d}}h^2f_{\infty}dxdv+\int_{\mathbb{R}^{2d}}
|\nabla_x h|^2 f_{\infty}dxdv+\int_{\mathbb{R}^{2d}}
|\nabla_v h|^2 f_{\infty}dxdv \right]
\notag\\ =& \min\left\{\gamma,\, \frac{\gamma p_1}{\gamma+p_1||\rho_{\infty}||_{L^{\infty}}}\right\}||h||^2_{H^1(\mathbb{R}^{2d},f_{\infty})}.
\end{align}
\eqref{Eee1} and \eqref{Eee2} provide \eqref{E equiv}.

We now prove \eqref{E<S}.
The definition of $ \mathrm{E}$ implies $\mathrm{S}_P[h]\leq \mathrm{E}[h].$  Since $P\geq p_1 I,$ we have
 \begin{equation*}\mathrm{S}_P[h]\geq p_1\left[\int_{\mathbb{R}^{2d}}
|\nabla_x (h+\psi)|^2f_{\infty}dxdv+ \int_{\mathbb{R}^{2d}}|\nabla_v (h+\psi)|^2f_{\infty}dxdv\right].
\end{equation*}
  Using the Poincar\'e inequality \eqref{Poin.inq.} and $\int_{\mathbb{R}^{2d}} h f_{\infty}dxdv=0$
 \begin{equation*}\mathrm{S}_P(h)\geq \frac{p_1}{ \kappa_2}\left[\int_{\mathbb{R}^{2d}}
 (h+\psi)^2f_{\infty}dxdv-\left( \int_{\mathbb{R}^{2d}} \psi f_{\infty}dxdv\right)^2\right].
\end{equation*}
The H\"older inequality and  $-\Delta_x \psi=\int_{\mathbb{R}^{d}} h f_{\infty}dv$ imply
\begin{align*}\int_{\mathbb{R}^{2d}}
 (h+\psi)^2f_{\infty}dxdv-&\left( \int_{\mathbb{R}^{2d}} \psi f_{\infty}dxdv\right)^2 \geq \int_{\mathbb{R}^{2d}}
 h^2f_{\infty}dxdv+2\int_{\mathbb{R}^{2d}}
 h\psi f_{\infty}dxdv\\ &=\int_{\mathbb{R}^{2d}}
 h^2f_{\infty}dxdv-2\int_{\mathbb{R}^{d}}
 \psi \Delta_x \psi dx\\&=\int_{\mathbb{R}^{2d}}
 h^2f_{\infty}dxdv+2\int_{\mathbb{R}^{d}}
 |\nabla_x\psi|^2 dx\geq ||h||^2.
 \end{align*}
Thus, $\mathrm{S}_P[h]\geq \frac{p_1}{ \kappa_2}||h||^2 $ and \eqref{E<S} follows.
\end{proof}
We now prove the main result of this subsection.
\begin{theorem}\label{th:hypocoercivity}
 Let $h$ be the solution of \eqref{linVPFP} with  an initial data  $h_0 \in H^1(\mathbb{R}^{2d}, f_{\infty})
$ such that $\int_{\mathbb{R}^{2d}}h_0 f_{\infty} dxdv=0.$
Let the assumptions $(A1),$ $(A2)$ and $(A3)$ hold. Then,  there exist a positive constant $\gamma$ and a constant, symmetric, positive definite matrix $P$ such that
\begin{equation}\label{dt E<-labmda E}
\frac{d}{dt}\mathrm{E}[h(t)]\leq -2\lambda \mathrm{E}[h(t)], \, \, \, \, \, t\geq 0
\end{equation}
holds for some $\lambda>0$ depending     $\gamma$ and $P.$ In particular, \begin{equation*}\label{Cor:H^1 decay}\mathrm{E}[h(t)]\leq e^{-2\lambda t} \mathrm{E}[h_0], \, \, \, \, \, t\geq 0
.\end{equation*}
\end{theorem}
\begin{proof}
 Lemma \ref{lemma ||h||^2} and Lemma \ref{lemma main} show that
\begin{align}\label{dt E}
\frac{d }{dt}\mathrm{E}[h(t)]=&-2\sigma\gamma \int_{\mathbb{R}^{2d}}|\nabla_v h|^2f_{\infty}dxdv\notag\\&
-2\sigma\int_{\mathbb{R}^{2d}}\left\{\sum_{i=1}^d \begin{pmatrix}
\nabla_x ( \partial_{v_i} h)\\ \nabla_v(\partial_{v_i} h)
\end{pmatrix}^T  P\begin{pmatrix}
\nabla_x (\partial_{v_i} h)\\\nabla_v(\partial_{v_i} h)
\end{pmatrix} \right\}f_{\infty}dxdv\notag\\&-\int_{\mathbb{R}^{2d}}  \begin{pmatrix}
\nabla_x (h+\psi)\\ \nabla_v h
\end{pmatrix}^T  \left\{QP+PQ^T\right\} \begin{pmatrix}
\nabla_x (h+\psi)\\ \nabla_v h
\end{pmatrix} f_{\infty}dxdv\notag\\&
-2\int_{\mathbb{R}^{2d}} \begin{pmatrix}
\nabla_x (h+\psi)\\ \nabla_v h
\end{pmatrix}^T P \begin{pmatrix}
\partial_t \nabla_x \psi\\ 0
\end{pmatrix} f_{\infty}dxdv.
\end{align}
We choose the matrix
$P\colonequals \begin{pmatrix}
\varepsilon^3I& \varepsilon^2I\\
\varepsilon^2I& 2\varepsilon I
\end{pmatrix}
$
with $\varepsilon>0$ which will be fixed later. It is easy to check that $P$ is positive definite. We denote $W\colonequals V+\phi_{\infty}.$ Then, we have
$$QP+PQ^T=\begin{pmatrix}
-2\varepsilon^2I& \varepsilon^3 \frac{\partial^2 W}{\partial x^2}-(\nu\varepsilon^2+2\varepsilon)I\\ \varepsilon^3 \frac{\partial^2 W}{\partial x^2}-(\nu\varepsilon^2+2\varepsilon)I& 2\varepsilon^2 \frac{\partial^2 W}{\partial x^2}-4\nu\varepsilon I
\end{pmatrix}.$$
This helps to compute
\begin{align}\label{es.QP+}
&-\int_{\mathbb{R}^{2d}}  \begin{pmatrix}
\nabla_x (h+\psi)\notag\\ \nabla_v h
\end{pmatrix}^T  \left\{QP+PQ^T\right\} \begin{pmatrix}
\nabla_x (h+\psi)\\ \nabla_v h
\end{pmatrix} f_{\infty}dxdv\notag\\
=&-2\varepsilon^2\int_{\mathbb{R}^{2d}}|\nabla_x (h+\psi)|^2f_{\infty}dxdv\notag\\&+2\int_{\mathbb{R}^{2d}}\nabla_x^T (h+\psi)\left(\varepsilon^3 \frac{\partial^2 W}{\partial x^2}-(\nu \varepsilon^2+2\varepsilon)I\right)\nabla_v h f_{\infty}dxdv\notag\\&+\int_{\mathbb{R}^{2d}}\nabla_v^T h\left(2\varepsilon^2 \frac{\partial^2 W}{\partial x^2}-4\nu\varepsilon I\right)\nabla_v h f_{\infty}dxdv.
\end{align}
We work on the  terms which contain $ \frac{\partial^2 W}{\partial x^2}.$ We use the H\"older inequality and \eqref{Vil.s lemma} to get
\begin{align}\label{xWv}
&2\varepsilon^3\int_{\mathbb{R}^{2d}}\nabla_x^T (h+\psi) \frac{\partial^2 W}{\partial x^2}\nabla_v h f_{\infty}dxdv\notag\\\leq &2\varepsilon^3\int_{\mathbb{R}^{2d}}|\nabla_x (h+\psi)|\left|\left| \frac{\partial^2 W}{\partial x^2}\right|\right|_F|\nabla_v h| f_{\infty}dxdv\notag\\ \leq &\varepsilon^2 \int_{\mathbb{R}^{2d}}|\nabla_x (h+\psi)|^2  f_{\infty}dxdv+{\varepsilon^4}\int_{\mathbb{R}^{2d}}\left|\left|\frac{\partial^2 W}{\partial x^2}\right|\right|_F^2|\nabla_v h|^2 f_{\infty}dxdv\notag\\ \leq & \varepsilon^2\int_{\mathbb{R}^{2d}}|\nabla_x (h+\psi)|^2  f_{\infty}dxdv+{\varepsilon^4\kappa_3}\int_{\mathbb{R}^{2d}}|\nabla_v h |^2 f_{\infty}dxdv\notag\\&+{\varepsilon^4\kappa_3}\sum_{i=1}^d\int_{\mathbb{R}^{2d}}|\nabla_x  (\partial_{v_i} h)|^2 f_{\infty}dxdv.
\end{align}
Similarly,
\begin{multline}\label{vWv}
2\varepsilon^2\int_{\mathbb{R}^{2d}}\nabla_v^T h \frac{\partial^2 W}{\partial x^2}\nabla_v h f_{\infty}dxdv\leq 2\varepsilon^2\int_{\mathbb{R}^{2d}}|\nabla_v h| \left| \left|\frac{\partial^2 W}{\partial x^2}\right|\right|_F |\nabla_v h| f_{\infty}dxdv\\ \leq\int_{\mathbb{R}^{2d}}|\nabla_v h|^2  f_{\infty}dxdv+\varepsilon^4 \int_{\mathbb{R}^{2d}}\left|\left|\frac{\partial^2 W}{\partial x^2}\right|\right|_F^2|\nabla_v h|^2 f_{\infty}dxdv\\  \leq (1+\varepsilon^4 \kappa_3)\int_{\mathbb{R}^{2d}}|\nabla_v h|^2  f_{\infty}dxdv+{\varepsilon^4 \kappa_3}\sum_{i=1}^d\int_{\mathbb{R}^{2d}}|\nabla_x  (\partial_{v_i} h)|^2 f_{\infty}dxdv.
\end{multline}
Combining \eqref{es.QP+}, \eqref{xWv} and \eqref{vWv} we get
{\small \begin{equation}\label{simp.QP+..}
\begin{split}
&-\int_{\mathbb{R}^{2d}}  \begin{pmatrix}
\nabla_x (h+\psi)\\ \nabla_v h
\end{pmatrix}^T  \left\{QP+PQ^T\right\} \begin{pmatrix}
\nabla_x (h+\psi)\\ \nabla_v h
\end{pmatrix} f_{\infty}dxdv\\
\leq &2{\varepsilon^4 \kappa_3}\sum_{i=1}^d\int_{\mathbb{R}^{2d}}|\nabla_x  (\partial_{v_i} h)|^2 f_{\infty}dxdv\\
&-\int_{\mathbb{R}^{2d}}  \begin{pmatrix}
\nabla_x (h+\psi)\\ \nabla_v h
\end{pmatrix}^T\begin{pmatrix}
\varepsilon^2 I& (\nu \varepsilon^2+2\varepsilon)I\\ (\nu \varepsilon^2+2\varepsilon)I& (4\nu \varepsilon-1-2\varepsilon^4\kappa_3)I
\end{pmatrix} \begin{pmatrix}
\nabla_x (h+\psi)\\ \nabla_v h
\end{pmatrix} f_{\infty}dxdv.
\end{split}
\end{equation}}
Next, we estimate  the last term of \eqref{dt E}:
\begin{multline}\label{simhPpsi}
-2\int_{\mathbb{R}^{2d}} \begin{pmatrix}
\nabla_x (h+\psi)\\ \nabla_v h
\end{pmatrix}^T P \begin{pmatrix}
\nabla_x \partial_t  \psi\\0
\end{pmatrix} f_{\infty}dxdv\\=-2\varepsilon^3\int_{\mathbb{R}^{2d}} \nabla_x (h+\psi) \cdot  \nabla_x \partial_t \psi f_{\infty}dxdv-2\varepsilon^2\int_{\mathbb{R}^{2d}} \nabla_v h\cdot  \nabla_x \partial_t  \psi  f_{\infty}dxdv\\ \leq \varepsilon^4 \int_{\mathbb{R}^{2d}}|\nabla_x (h+\psi)|^2f_{\infty}dxdv+{\varepsilon^2}\int_{\mathbb{R}^{2d}}|\nabla_x \partial_t  \psi|^2f_{\infty}dxdv\\+2\varepsilon^2\sqrt{\int_{\mathbb{R}^{2d}} |\nabla_v h|^2 f_{\infty}dxdv}\sqrt{\int_{\mathbb{R}^{2d}} |\nabla_x \partial_t  \psi |^2 f_{\infty}dxdv}\\
\leq \varepsilon^4 \int_{\mathbb{R}^{2d}}|\nabla_x (h+\psi)|^2f_{\infty}dxdv\\+\varepsilon^2\left(\frac{\sigma^2||\rho_{\infty}||^2_{L^{\infty}}}{\nu^2}+2\frac{\sigma ||\rho_{\infty}||_{L^{\infty}}}{\nu}\right)\int_{\mathbb{R}^{2d}}| \nabla_v h|^2f_{\infty}dxdv,
\end{multline}
where we used the H\"older inequality and \eqref{psi tx}.
We gather  \eqref{simp.QP+..} and \eqref{simhPpsi} to estimate \eqref{dt E}
\begin{multline}\label{sim dtE}
\frac{d }{dt}\mathrm{E}[h(t)]
\leq -2\sigma\int_{\mathbb{R}^{2d}}\left\{\sum_{i=1}^d \begin{pmatrix}
\nabla_x ( \partial_{v_i} h)\\ \nabla_v(\partial_{v_i} h)
\end{pmatrix}^T
 \begin{pmatrix}
 (\varepsilon^3-\frac{\varepsilon^4 \kappa_3}{\sigma})I & \varepsilon^2 I\\
 \varepsilon^2 I& 2\varepsilon I
 \end{pmatrix}
\begin{pmatrix}
\nabla_x (\partial_{v_i} h)\\\nabla_v(\partial_{v_i} h)
\end{pmatrix} \right\}f_{\infty}dxdv\\-\int_{\mathbb{R}^{2d}}  \begin{pmatrix}
\nabla_x (h+\psi)\\ \nabla_v h
\end{pmatrix}^T
P_1
 \begin{pmatrix}
\nabla_x (h+\psi)\\ \nabla_v h
\end{pmatrix} f_{\infty}dxdv,
\end{multline}
where $$P_1\colonequals \begin{pmatrix}
(\varepsilon^2-\varepsilon^4) I& (\nu\varepsilon^2+2\varepsilon)I\\ (\nu\varepsilon^2+2\varepsilon)I& (2\gamma\sigma-1+4\nu \varepsilon-\varepsilon^2(\frac{\sigma^2||\rho_{\infty}||^2_{L^{\infty}}}{\nu^2}+\frac{2\sigma||\rho_{\infty}||_{L^{\infty}}}{\nu})-2\varepsilon^4\kappa_3)I
\end{pmatrix}.$$
We choose $\gamma>0$ and $\varepsilon>0$ such that   the  matrices in the first and second lines of \eqref{sim dtE} satisfy $$\begin{pmatrix}
 (\varepsilon^3-\frac{\varepsilon^4 \kappa_3}{\sigma})I & \varepsilon^2 I\\
 \varepsilon^2 I& 2\varepsilon I
 \end{pmatrix}\geq 0 \, \, \, \, \text{ and } \, \, \, \, P_1>0.$$
It is possible to choose such $\gamma$ and $\varepsilon,$ for example, if $\varepsilon$ is small enough and $\gamma$ is large enough, then the conditions above are satisfied. Moreover, there is $\tilde{\lambda}=\tilde{\lambda}(\varepsilon,\gamma)>0$ such that
$$P_1\geq 2\tilde{\lambda}P.
$$ Using this estimate in \eqref{sim dtE} we get
 \begin{equation}\label{dt E<S}\frac{d}{dt}\mathrm{E}[h(t)]\leq -2\tilde{\lambda} \mathrm{S}_P[h(t)].
 \end{equation}
By using  \eqref{E<S}
$$\frac{d}{dt}\mathrm{E}[h(t)]\leq -2\lambda \mathrm{E}[h(t)]\, \, \, \, \, \text{ for all} \, \, \,  t\geq 0$$
with $\displaystyle \lambda\colonequals \tilde{\lambda} \frac{p_1}{p_1+\gamma \kappa_2}.$ Then, the Gr\"onwall inequality implies $$\mathrm{E}[h(t)]\leq e^{-2\lambda t} \mathrm{E}[h_0].$$ 
\end{proof}
\begin{Remark}\emph{\begin{itemize}
\item[1.]
 We note that $\gamma$ and $P$ such that Theorem \ref{th:hypocoercivity} holds are not unique. But the decay rate $\lambda$ depends on $\gamma$ and $P.$ To get a better rate, one has to optimize $\lambda=\lambda(\gamma, P)$ with respect to $\gamma$ and $P.$
  \item[2.] The Poincar\'e inequality is essential to get the inequality \eqref{dt E<-labmda E} and so the exponential decay. When the Poincar\'e inequality is not valid, we only get  the bound $\mathrm{E}[h(t)]\leq  \mathrm{E}[h_0]$ by \eqref{dt E<S}.
\end{itemize}}\end{Remark}
\subsection{Hypoelliptic regularity}
In this section we show that, for any initial data  $h_0 \in L^2(\mathbb{R}^{2d}, f_{\infty}),$ the solution $h(t)$ of the linearized equation is in $ H^1(\mathbb{R}^{2d}, f_{\infty})$ for all $t>0.$
\begin{theorem}\label{th:hypoellip}
 Let $h$ be the solution of \eqref{linVPFP} with  an initial data  $h_0 \in L^2(\mathbb{R}^{2d}, f_{\infty})
.$
Under the assumptions $(A1)$ and $(A2)$,  for any $t_0>0,$  there are explicitly computable constants $C_1>0$ and $C_2>0$  such that
\begin{equation}\label{hypel1}
\int_{\mathbb{R}^{2d}}|\nabla_x h(t)|^2f_{\infty}dxdv\leq \frac{C_1}{t^3}\int_{\mathbb{R}^{2d}}h^2_0f_{\infty}dxdv
\end{equation} and
\begin{equation}\label{hypoel2}
\int_{\mathbb{R}^{2d}}|\nabla_v h(t)|^2f_{\infty}dxdv\leq \frac{C_2}{t}\int_{\mathbb{R}^{2d}}h^2_0f_{\infty}dxdv
\end{equation}
hold for all $t \in (0,t_0].$
\end{theorem}
\begin{proof}
In order to prove the short-time regularization of \eqref{hypel1} and \eqref{hypoel2} we consider now  the functional $\mathrm{E}$ with  a matrix $P$ which depends explicitly on time $t,$ i.e.
 $$P=P(t)\colonequals \begin{pmatrix}
\varepsilon^3 t^3 I& \varepsilon^2 t^2I\\
\varepsilon^2 t^2I& {2\varepsilon t I}
\end{pmatrix}.
$$
 We shall fix $\varepsilon>0$ later. It is easy to check
 \begin{equation}\label{P>}
 P(t)\geq\begin{pmatrix}
 {\frac{\varepsilon^3 t^3}{3} I}&0\\
 0&  {\frac{\varepsilon t}{2}I}
 \end{pmatrix}
 \end{equation} which implies that $P(t)$ is positive definite for all $t>0.$  Our goal is to show that $\mathrm{E}[h(t)]$  decreases. To this end we compute the time derivative of $\mathrm{E}[h(t)].$ We follow the proofs of Lemma \ref{lemma ||h||^2} and Lemma \ref{lemma main} to compute the time derivative of $\mathrm{E},$ but we need to take into account that $P$ depends on time $t:$ 
\begin{align*}
\frac{d }{dt}\mathrm{E}[h(t)]=&-2\sigma \gamma \int_{\mathbb{R}^{2d}}|\nabla_v h|^2f_{\infty}dxdv\\
&-2\sigma\int_{\mathbb{R}^{2d}}\left\{\sum_{i=1}^d \begin{pmatrix}
\nabla_x ( \partial_{v_i} h)\\ \nabla_v(\partial_{v_i} h)
\end{pmatrix}^T  P\begin{pmatrix}
\nabla_x (\partial_{v_i} h)\\\nabla_v(\partial_{v_i} h)
\end{pmatrix} \right\}f_{\infty}dxdv\\&-\int_{\mathbb{R}^{2d}}  \begin{pmatrix}
\nabla_x (h+\psi)\\ \nabla_v h
\end{pmatrix}^T  \left\{QP+PQ^T\right\} \begin{pmatrix}
\nabla_x (h+\psi)\\ \nabla_v h
\end{pmatrix} f_{\infty}dxdv\\
&-2\int_{\mathbb{R}^{2d}} \begin{pmatrix}
\nabla_x (h+\psi)\\ \nabla_v h
\end{pmatrix}^T P \begin{pmatrix}
\partial_t \nabla_x \psi\\ 0
\end{pmatrix} f_{\infty}dxdv\\&+\int_{\mathbb{R}^{2d}}  \begin{pmatrix}
\nabla_x h\\ \nabla_v h
\end{pmatrix}^T  \partial_t P \begin{pmatrix}
\nabla_x h\\ \nabla_v h
\end{pmatrix} f_{\infty}dxdv.
\end{align*}
We estimate the terms on the right as \eqref{xWv}-\eqref{simhPpsi} (where we need to replace $\varepsilon$ to $\varepsilon t$) and obtain
\begin{multline}\label{sim dtE t}
\frac{d }{dt}\mathrm{E}[h(t)]\leq 
-2\sigma\int_{\mathbb{R}^{2d}}\left\{\sum_{i=1}^d \begin{pmatrix}
\nabla_x ( \partial_{v_i} h)\\ \nabla_v(\partial_{v_i} h)
\end{pmatrix}^T
 \begin{pmatrix}
 (\varepsilon^3t^3-\frac{\varepsilon^4t^4 \kappa_3}{\sigma})I & \varepsilon^2t^2 I\\
 \varepsilon^2t^2I& 2\varepsilon t I
 \end{pmatrix}
\begin{pmatrix}
\nabla_x (\partial_{v_i} h)\\\nabla_v(\partial_{v_i} h)
\end{pmatrix} \right\}f_{\infty}dxdv\\-\int_{\mathbb{R}^{2d}} \begin{pmatrix}
\nabla_x (h+\psi)\\ \nabla_v h
\end{pmatrix}^T
[P_1-\partial_t P]
 \begin{pmatrix}
\nabla_x (h+\psi)\\ \nabla_v h
\end{pmatrix} f_{\infty}dxdv,
\end{multline}
where
\begin{small}
$$P_1\colonequals \begin{pmatrix}
(\varepsilon^2t^2-\varepsilon^4 t^4) I& (\nu\varepsilon^2 t^2+2\varepsilon t)I\\ (\nu\varepsilon^2 t^2+2\varepsilon t)I& \left(2\gamma \sigma-1+4\nu \varepsilon t-\frac{\sigma^2 ||\rho_{\infty}||^2_{L^{\infty}}+2\sigma \nu||\rho_{\infty}||_{L^{\infty}}}{\nu^2} \varepsilon^2 t^2-2\varepsilon^4 t^4 \kappa_3\right)I
\end{pmatrix}.$$
\end{small}
Since $\partial_t P=
\begin{pmatrix}
3\varepsilon^3 t^2 I & {2\varepsilon^2 t I}\\
2\varepsilon^2 t I & 2 \varepsilon I
\end{pmatrix},$ we have\\
$P_1-\partial_t P=$\\

$\begin{pmatrix}
([\varepsilon^2-3\varepsilon^3]t^2-\varepsilon^4 t^4) I& (\nu\varepsilon^2 t^2+2[\varepsilon-\varepsilon^2] t)I\\ (\nu\varepsilon^2 t^2+2[\varepsilon-\varepsilon^2] t)I& \left(2(\gamma\sigma-\varepsilon)-1 +4\nu\varepsilon t- \frac{\sigma^2 ||\rho_{\infty}||^2_{L^{\infty}}+2\sigma \nu||\rho_{\infty}||_{L^{\infty}}}{\nu^2} \varepsilon^2 t^2 -2\varepsilon^4 t^4 \kappa_3\right)I
\end{pmatrix}.$
We choose $\gamma$ and $\varepsilon$ such that, for all $t \in [0,t_0],$  the  matrices in the first and the second lines of \eqref{sim dtE t} are positive semi-definite, i.e. $$\begin{pmatrix}
 (\varepsilon^3t^3-\frac{\varepsilon^4t^4 \kappa_3}{\sigma})I & \varepsilon^2t^2 I\\
 \varepsilon^2t^2I& 2\varepsilon t I
 \end{pmatrix}\geq 0, \, \, \, \, \, \, P_1-\partial_t P\geq 0.$$
It is possible to choose such $\gamma$ and $\varepsilon,$ for example, if $\gamma$ is large and  $\varepsilon$ is small enough, then these inequalities hold.
Then we get $$\frac{d }{dt}\mathrm{E}[h(t)]\leq 0, \, \, \, \, \, t\in (0,t_0].$$ This yields that  $\mathrm{E}[h(t)]$ is decreasing in $[0,t_0].$   $\mathrm{E}[h(t=0)]= \gamma ||h_0||^2$ and \eqref{gr.psi in L^s} with $p=\frac{2d}{d+2}$ show
\begin{equation}\label{E<gamma ||||} \mathrm{E}[h(t)]\leq \gamma ||h_0||^2\leq \gamma(1+\theta_1^2)\int_{\mathbb{R}^{2d}}h_0^2f_{\infty}dxdv, \, \, \,  \, \, \forall t \in [0,t_0].
\end{equation}
On the other hand, we have by \eqref{P>} that
\begin{align*}
\mathrm{E}[h(t)] &\geq \gamma\int_{\mathbb{R}^{2d}}h^2(t)f_{\infty}dxdv+\gamma \int_{\mathbb{R}^{2d}}|\nabla_x \psi(t)|^2dx\\ &\, \, \, \,\, \,  +\frac{\varepsilon^3t^3}{3}\int_{\mathbb{R}^{2d}}|\nabla_x h(t)+\nabla_x \psi(t)|^2f_{\infty}dxdv+\frac{\varepsilon t}{2}\int_{\mathbb{R}^{2d}}|\nabla_v h(t)|^2f_{\infty}dxdv. 
\end{align*}
If we use  the estimate \begin{align*}
\int_{\mathbb{R}^{2d}}
|\nabla_x (h+\psi)|^2 f_{\infty}dxdv &\geq \frac{1}{2}\int_{\mathbb{R}^{2d}}
|\nabla_x h|^2 f_{\infty}dxdv-\int_{\mathbb{R}^{2d}}
|\nabla_x \psi|^2 f_{\infty}dxdv\\ &\geq \frac{1}{2}\int_{\mathbb{R}^{2d}}
|\nabla_x h|^2 f_{\infty}dxdv-||\rho_{\infty}||_{L^{\infty}}\int_{\mathbb{R}^{d}}
|\nabla_x \psi|^2 dv,
\end{align*}
we get
$$ \mathrm{E}[h(t)]\geq \gamma\int_{\mathbb{R}^{2d}}|h(t)|^2f_{\infty}dxdv+\left(\gamma-\frac{||\rho_{\infty}||_{L^{\infty}}\varepsilon^3t^3}{3}\right)\int_{\mathbb{R}^{2d}}|\nabla_x \psi(t)|^2dx$$
$$\hspace{20pt}+\frac{\varepsilon^3t^3}{6}\int_{\mathbb{R}^{2d}}|\nabla_x h(t)|^2f_{\infty}dxdv+\frac{\varepsilon t}{2}\int_{\mathbb{R}^{2d}}|\nabla_v h(t)|^2f_{\infty}dxdv.$$
 If we take $\gamma$ large enough so that $\gamma-\frac{||\rho_{\infty}||_{L^{\infty}}\varepsilon^3t_0^3}{3}\geq 0,$ then 
  \begin{equation}\label{E>||.||}\mathrm{E}[h(t)]\geq \frac{\varepsilon^3t^3}{6}\int_{\mathbb{R}^{2d}}|\nabla_x h(t)|^2f_{\infty}dxdv+\frac{\varepsilon t}{2}\int_{\mathbb{R}^{2d}}|\nabla_v h(t)|^2f_{\infty}dxdv, \, \, \, \, \, \,  \, \,  \forall t\in (0,t_0].
 \end{equation} \eqref{E<gamma ||||} and \eqref{E>||.||} show that the statement of the theorem holds with constants $C_1\colonequals \frac{6\gamma(1+\theta_1^2)}{\varepsilon^3} $ and $C_2\colonequals \frac{2\gamma(1+\theta_1^2)}{\varepsilon}. $
\end{proof}

 Now we are ready to prove Theorem \ref{th:lin} concerning the linearized Vlasov-Poissson- Fokker-Planck equation.

\begin{proof}[\textbf{Proof of Theorem \ref{th:lin}}]
The proofs of Theorem \ref{th:lin}$\,(i)$ and Theorem \ref{th:lin}$\,(ii)$  follow from Theorem \ref{th:linEXIS} and  Theorem \ref{th:hypoellip}.

  We use  Theorem \ref{th:hypocoercivity} to prove Theorem \ref{th:lin}$\,(iii).$ Let $\mathrm{E}$ be the functional  in Theorem \ref{th:hypocoercivity}. Then, \eqref{dt E<-labmda E} can be written as
 $$\frac{d}{dt}\left(e^{2\lambda t}\mathrm{E}[h(t)]\right)\leq 0.$$ For any $t>t_0,$ we integrate this inequality in $[t_0, t]$ to get
 \begin{equation}\label{E_0}\mathrm{E}[h(t)]\leq e^{-2\lambda (t-t_0)} \mathrm{E}[h(t_0)].
 \end{equation}
 By \eqref{E<S} we have
 \begin{align*}
 {\mathrm{E}[h(t_0)]}& \leq  \frac{p_1+\gamma \kappa_2}{p_1}\mathrm{S}_P[h(t_0)]\\
 & \leq \frac{p_2(p_1+\gamma \kappa_2)}{p_1}\left[\int_{\mathbb{R}^{2d}}
|\nabla_x h(t_0)+\nabla_x \psi(t_0)|^2 f_{\infty}dxdv+\int_{\mathbb{R}^{2d}}
|\nabla_v h(t_0)|^2 f_{\infty}dxdv\right]\\
&\leq \frac{p_2(p_1+\gamma \kappa_2)}{p_1}\left[2\int_{\mathbb{R}^{2d}}
|\nabla_x h(t_0)|^2f_{\infty}+\int_{\mathbb{R}^{2d}}
|\nabla_v h(t_0)|^2 f_{\infty}dxdv\right]\\
& +2|\rho_{\infty}||_{L^{\infty}}\frac{p_2(p_1+\gamma \kappa_2)}{p_1}\int_{\mathbb{R}^{2d}}|\nabla_x \psi(t_0)|^2 dxdv,
\end{align*}
where $p_1$ and $p_2$ are the smallest and the largest eigenvalues of the matrix $P$ which we defined in the proof of Theorem \ref{th:hypocoercivity}.  Because of \eqref{gr.psi in L^s} with $p=\frac{2d}{d+2}$ and the Poincar\'e inequality \eqref{Poin.inq.}, the integral
$\int_{\mathbb{R}^{d}}
|\nabla_x \psi(t_0)|^2dx$ is bounded by  $$\int_{\mathbb{R}^{2d}}
|\nabla_x h(t_0)|^2 f_{\infty}dxdv+\int_{\mathbb{R}^{2d}}
|\nabla_v h(t_0)|^2 f_{\infty}dxdv.$$
Thus, there is a constant $C>0$ such that
$$
 {\mathrm{E}[h(t_0)]}\leq C\left[\int_{\mathbb{R}^{2d}}
|\nabla_x h(t_0)|^2 f_{\infty}dxdv+\int_{\mathbb{R}^{2d}}
|\nabla_v h(t_0)|^2 f_{\infty}dxdv\right],$$
and by Theorem \ref{th:hypoellip}
\begin{equation}\label{E_01}
 {\mathrm{E}[h(t_0)]}\leq C\max\{C_1t_0^{-3},{C_2}{t_0^{-1}}\} \int_{\mathbb{R}^{2d}}h^2_0f_{\infty}dxdv.
 \end{equation}
If we combine  \eqref{E_0}, \eqref{E_01} and \eqref{E equiv}, we obtain
$$\int_{\mathbb{R}^{2d}}h^2(t)f_{\infty}dxdv+\int_{\mathbb{R}^{2d}}
|\nabla_x h(t)|^2 f_{\infty}dxdv+\int_{\mathbb{R}^{2d}}
|\nabla_v h(t)|^2 f_{\infty}dxdv$$ $$\leq  \frac{\mathrm{E}[h(t)]}{ \min\left\{\gamma,\, \frac{\gamma p_1}{\gamma+p_1||\rho_{\infty}||_{L^{\infty}}}\right\}} \leq  \frac{e^{-2\lambda (t-t_0)} \mathrm{E}[h(t_0)]}{ \min\left\{\gamma,\, \frac{\gamma p_1}{\gamma+p_1||\rho_{\infty}||_{L^{\infty}}}\right\}}
$$ $$ \leq  \frac{Ce^{2 \lambda t_0}\max\{C_1t_0^{-3},{C_2}{t_0^{-1}}\}}{ \min\left\{\gamma,\, \frac{\gamma p_1}{\gamma+p_1||\rho_{\infty}||_{L^{\infty}}}\right\}} e^{-2\lambda t} \int_{\mathbb{R}^{2d}}h^2_0f_{\infty}dxdv. $$
This proves \eqref{decay} with the constant $\displaystyle C_3\colonequals\sqrt{ \frac{Ce^{2 \lambda t_0}\max\{C_1t_0^{-3},{C_2}{t_0^{-1}}\}}{ \min\left\{\gamma,\, \frac{\gamma p_1}{\gamma+p_1||\rho_{\infty}||_{L^{\infty}}}\right\}}}.$

We now prove \eqref{psi W}. We have from \eqref{gr.psi in L^s}
$$||\nabla_x \psi(t)||_{ L^{\frac{pd}{d-p}}(\mathbb{R}^d)}\leq \theta_1||h(t)||_{ L^2(\mathbb{R}^{2d},f_{\infty})}\leq \theta_1||h(t)||_{ H^1 (\mathbb{R}^{2d},f_{\infty})}, \, \, \, \, \, \forall\, p\in (1,2].$$
 The relation $\mathscr{L}^{\frac{2d}{d-2}}_{1}(\mathbb{R}^d)=W^{1, \frac{2d}{d-2}}(\mathbb{R}^d)$ (see \cite{Adams}) and \eqref{gr.psi in Lalpha} with $\alpha=1$  show $$||\nabla_x \psi(t)||_{W^{1, \frac{2d}{d-2}}(\mathbb{R}^d)}\leq \theta_2||h(t)||_{ H^{1}_x(\mathbb{R}^{2d},f_{\infty})}\leq \theta_2||h(t)||_{ H^{1}(\mathbb{R}^{2d},f_{\infty})}$$ for all $t>0.$ These estimates and \eqref{decay} imply
\begin{align*}||\nabla_x \psi(t)||_{ L^{\frac{pd}{d-p}}(\mathbb{R}^d)}+||\nabla_x \psi(t)||_{W^{1, \frac{2d}{d-2}}(\mathbb{R}^d)} & \leq (\theta_1+\theta_2)||h(t)||_{ H^1 (\mathbb{R}^{2d},f_{\infty})}\\
&\leq C_3(\theta_1+\theta_2) e^{-\lambda t}||h_0||_{ L^{2}(\mathbb{R}^{2d},f_{\infty})}
\end{align*}
for all $t\geq t_0.$
This proves \eqref{psi W} with the constant $C_4\colonequals  C_3(\theta_1+\theta_2).$

\end{proof}

\section{The nonlinear Vlasov-Poisson-Fokker-Planck system in dimension three}
In this section we work on the nonlinear Vlasov-Poisson-Fokker-Planck system \eqref{hVPFP} in  dimension $d=3$. We mention that we get \eqref{hVPFP} from \eqref{VPFP} by taking $ h \colonequals \frac{f}{f_{\infty}}-1$ and $\psi\colonequals \phi-\phi_{\infty}.$
To prove Theorem \ref{Short time} and Theorem \ref{th:VPFP}, we will consider the integral version of the system \eqref{INTe}, and then apply a fixed point argument to find a candidate for a solution, then prove existence, uniqueness and stability.
\subsection{Semigroup estimates }

  We recall that the linearized system  \eqref{linVPFP} can be written as
$$\partial_t h+Kh=0$$
with the operator $Kh \colonequals v\cdot\nabla_x h-\nabla_x (V+ \phi_{\infty})\cdot \nabla_v h+v\cdot\nabla_x \psi-\sigma\Delta_v h+\nu v\cdot \nabla_v h.$ By Theorem \ref{th:linEXIS} $K$ generates a $C_0$ semigroup $e^{tK}$ on $ L^2(\mathbb{R}^{2d}, f_{\infty}).$ 
We define a subspace of $L^2(\mathbb{R}^{2d}, f_{\infty})$ $$\mathcal{H}\colonequals \{g \in L^2(\mathbb{R}^{2d}, f_{\infty}): \int_{\mathbb{R}^{2d}}gf_{\infty}dxdv=0\}.$$
 Here the norm of  $\mathcal{H}$ is the norm of $L^2(\mathbb{R}^{2d}, f_{\infty}).$
  Let  $h_0\in \mathcal{H},$ then $h(t)=e^{-t K }h_0, \, \, t>0$ is the solution of \eqref{linVPFP} and it is in  $H^1(\mathbb{R}^{2d}, f_{\infty})$ by Theorem \ref{th:lin} $(ii).$
 By integrating \eqref{linVPFP}, we obtain
$$\int_{\mathbb{R}^{2d}}e^{-t K }h_0f_{\infty}dxdv=\int_{\mathbb{R}^{2d}}h_0f_{\infty}dxdv=0.$$
 Therefore, $e^{-t K }$ maps  $\mathcal{H}$ into $\mathcal{H}\cap H^1(\mathbb{R}^{2d}, f_{\infty})$ for all $t>0.$\\

We will need the following estimates on $e^{-t K }.$
\begin{lemma}\label{lem: e^tK0}
Let $d\geq 3,$ $\alpha \in [0,1],$ $t_0>0, $ the assumption $(A1)$ and $ (A2)$ hold.   There are positive constants $\mathscr{C}_1,$  $\mathscr{C}_2,$  $\mathscr{C}_3,$  $\mathscr{C}_4$ and  $\mathscr{C}_5$  such that
\begin{itemize}
\item[(i)]\begin{equation}\label{e:L^20}
||e^{-t K }h_0||_{L^2(\mathbb{R}^{2d}, f_{\infty})}\leq  \mathscr{C}_1 ||h_0||_{L^2(\mathbb{R}^{2d}, f_{\infty})},   \, \forall\, t\in [0, t_0],\,  \forall\, h_0 \in L^2(\mathbb{R}^{2d}, f_{\infty}).
\end{equation}
\item[(ii)]
\begin{equation}\label{e:H^alpha0}
|| e^{-t K }h_0||_{H^{\alpha }_x(\mathbb{R}^{2d}, f_{\infty})}\leq  \mathscr{C}_2  ||h_0||_{H^{\alpha }_x(\mathbb{R}^{2d}, f_{\infty})}, \,  \forall\, t\in [0, t_0],  \,  \forall\, h_0 \in  H^{\alpha}_x(\mathbb{R}^{2d}, f_{\infty}).
\end{equation}
\item[(iii)]
 \begin{equation}\label{ee:H_x^alpha0}
||e^{-t K }h_0||_{H^{\alpha}_x(\mathbb{R}^{2d}, f_{\infty})}\leq  \mathscr{C}_3 (1+t^{-\frac{3\alpha}{2}})||h_0||_{L^2(\mathbb{R}^{2d}, f_{\infty})}, \,  \forall\, t\in (0, t_0], \,  \forall\, h_0 \in L^2(\mathbb{R}^{2d}, f_{\infty}).
\end{equation}
\item[(iv)] \begin{equation}\label{e:H^1_v0}
|| e^{-t K }h_0||_{H^{1 }_v(\mathbb{R}^{2d}, f_{\infty})}\leq  \mathscr{C}_4
||h_0||_{H^{1 }_v(\mathbb{R}^{2d}, f_{\infty})}, \, \forall\, t\in [0, t_0],  \,  \forall\, h_0 \in  H^{1}_v(\mathbb{R}^{2d}, f_{\infty}).
\end{equation}
\item[(v)]
\begin{equation}\label{ee:H_v^10}
||e^{-t K }h_0||_{H^1_v(\mathbb{R}^{2d}, f_{\infty})}\leq  \mathscr{C}_5(1+ t^{-\frac{1}{2}}) ||h_0||_{L^2(\mathbb{R}^{2d}, f_{\infty})},   \,  \forall\, t\in (0, t_0], \,  \forall\, h_0 \in L^2(\mathbb{R}^{2d}, f_{\infty}).
\end{equation}
\end{itemize}
\end{lemma}
\begin{proof}
$(i)$ Lemma \ref{lemma ||h||^2} shows that $$||e^{-t K }h_0||\leq ||h_0||, \,\,\, \, \, t\geq 0.$$ The inequality $||e^{-t K }h_0||_{L^2(\mathbb{R}^{2d},f_{\infty})}\leq ||e^{-t K }h_0||$ and \eqref{gr.psi in L^s} with $p=\frac{2d}{d-2}$ yield
   \begin{equation}\label{L^2110}||e^{-t K }h_0||_{L^2(\mathbb{R}^{2d},f_{\infty})}\leq \sqrt{1+\theta^2_1} ||h_0||_{L^2(\mathbb{R}^{2d},f_{\infty})}, \,\,\, \, \, t\geq 0.
   \end{equation}
   Thus, \eqref{e:L^20} holds with the constant $\mathscr{C}_1\colonequals \sqrt{1+\theta^2_1}.$

   $(ii)$  Since $e^{-tK}$ generates a $C_0$ semigroup on $ L^{2}(\mathbb{R}^{2d}, f_{\infty})$  and $e^{-t K }h_0\in H^1(\mathbb{R}^{2d},\\ f_{\infty})$ for all $t>0,$ (see Theorem \ref{th:lin}$\,(ii)$), we conclude  $e^{-tK}$ also generates a $C_0$ semigroup on $ H^{\alpha}_x(\mathbb{R}^{2d}, f_{\infty})$ (see \cite[Theorem 0.1]{intS}).
 By the semigroup property \cite[Theorem 1.2.2]{Pazy} there exist constants $\omega\geq 0$ and $C\geq 1$ such that
$$||e^{-tK}h_0||_{H^{\alpha}_x(\mathbb{R}^{2d}, f_{\infty})}\leq C e^{\omega t} ||h_0||_{H^{\alpha}_x(\mathbb{R}^{2d}, f_{\infty})}$$
for all $t\geq 0.$ This implies that \eqref{e:H^alpha0} holds with $\mathscr{C}_2\colonequals C e^{\omega t_0}. $

$(iii)$ \eqref{ee:H_x^alpha0} coincides with \eqref{e:L^20} when $\alpha=0.$ We  prove \eqref{ee:H_x^alpha0} when $\alpha=1.$  By Theorem \ref{th:lin}$\,(ii)$ we have, for all $h_0 \in \mathcal{H},$ \begin{equation*}
\int_{\mathbb{R}^{2d}}|\nabla_x e^{-t K }h_0|^2f_{\infty}dxdv\leq \frac{C_1}{t^3} \int_{\mathbb{R}^{2d}}h^2_0f_{\infty}dxdv,  \, \, \, \, \, \, t \in (0,t_0].
\end{equation*}
We add the square of \eqref{L^2110} to this estimate to get
\begin{multline*}
\int_{\mathbb{R}^{2d}}|e^{-t K }h_0|^2f_{\infty}dxdv+\int_{\mathbb{R}^{2d}}|\nabla_x e^{-t K }h_0|^2f_{\infty}dxdv\\
\leq \left(1+\theta^2_1 + \frac{C_1}{t^3}\right) \int_{\mathbb{R}^{2d}}h^2_0f_{\infty}dxdv\\
\leq \max\{1+\theta_1^2, C_1\}(1+t^{-\frac{3}{2}})^2\int_{\mathbb{R}^{2d}}h^2_0f_{\infty}dxdv,  \, \, \, \, \, \,\forall t \in (0,t_0].
\end{multline*}
Then by \eqref{norm H_x^1} $$||e^{-tK}h_0||_{H^1_x(\mathbb{R}^{2d}, f_{\infty})}\leq {\kappa_5}\sqrt{\max\{1+\theta_1^2, C_1\}}(1+t^{-\frac{3}{2}})||h_0||_{L^2(\mathbb{R}^{2d}, f_{\infty})}.$$
 This proves \eqref{ee:H_x^alpha0} when $\alpha=1.$ The complete proof follows by interpolation.

 The proofs of $(iv)$ and $(v)$ follow by similar arguments which we did in $(ii)$ and $(iii).$
\end{proof}

\begin{lemma}\label{lem: e^tK}
Let $d\geq 3,$ the assumption $(A1),\, (A2)$ and $ (A3)$ hold. Let $\alpha \in [0,1],$  $\lambda>0$ be the constant appearing in Theorem \ref{th:lin} $(iii)$
 and $\lambda_1\in (0,\lambda).$ There are positive constants $\mathcal{C}_1,$ $\mathcal{C}_2,$  $\mathcal{C}_3,$ $\mathcal{C}_4$ and $\mathcal{C}_5$ such that
\begin{itemize}
\item[(i)]\begin{equation}\label{e:L^2}
||e^{-t K }h_0||_{L^2(\mathbb{R}^{2d}, f_{\infty})}\leq  \mathcal{C}_1e^{-\lambda t} ||h_0||_{L^2(\mathbb{R}^{2d}, f_{\infty})}, \,  \forall\, t\geq 0, \, \forall\, h_0 \in \mathcal{H}.
\end{equation}
\item[(ii)]
\begin{equation}\label{e:H^alpha}
|| e^{-t K }h_0||_{H^{\alpha }_x(\mathbb{R}^{2d}, f_{\infty})}\leq  \mathcal{C}_2 e^{-\lambda t}
||h_0||_{H^{\alpha }_x(\mathbb{R}^{2d}, f_{\infty})}, \, \forall\, t\geq 0, \,  \forall\, h_0 \in \mathcal{H}\cap H^{\alpha}_x(\mathbb{R}^{2d}, f_{\infty}).
\end{equation}
\item[(iii)]
 \begin{equation}\label{ee:H_x^alpha}
||e^{-t K }h_0||_{H^{\alpha}_x(\mathbb{R}^{2d}, f_{\infty})}\leq  \mathcal{C}_3 (1+t^{-\frac{3\alpha}{2}})e^{-\lambda_1 t}||h_0||_{L^2(\mathbb{R}^{2d}, f_{\infty})},    \,  \forall\, t> 0, \,   \forall\, h_0 \in \mathcal{H}.
\end{equation}
\item[(iv)]
\begin{equation}\label{e:H^1_v}
|| e^{-t K }h_0||_{H^{1 }_v(\mathbb{R}^{2d}, f_{\infty})}\leq  \mathcal{C}_4 e^{-\lambda t}
||h_0||_{H^{1 }_v(\mathbb{R}^{2d}, f_{\infty})}, \,  \forall\, t\geq 0, \,  \forall\, h_0 \in \mathcal{H}\cap H^{1}_v(\mathbb{R}^{2d}, f_{\infty}).
\end{equation}
\item[(v)]
 \begin{equation}\label{ee:H_v^1}
||e^{-t K }h_0||_{H^1_v(\mathbb{R}^{2d}, f_{\infty})}\leq  \mathcal{C}_5(1+ t^{-\frac{1}{2}}) e^{-\lambda_1 t}||h_0||_{L^2(\mathbb{R}^{2d}, f_{\infty})},  \,  \forall\, t> 0, \,   \forall\, h_0 \in \mathcal{H}.
\end{equation}
\end{itemize}
\end{lemma}
\begin{proof}
$(i)$ Let $t_0>0.$   
  Theorem \ref{th:lin}$\,(iii)$ implies, for all $h_0 \in \mathcal{H},$
\begin{equation*}\label{e3'}
||e^{-t K }h_0||_{L^2(\mathbb{R}^{2d},f_{\infty})}\leq ||e^{-t K }h_0||_{H^1(\mathbb{R}^{2d},f_{\infty})}\leq C_3 e^{-\lambda t}||h_0||_{L^2(\mathbb{R}^{2d},f_{\infty})}, \, \forall\, t\geq t_0.
\end{equation*}
 We combine this inequality and \eqref{e:L^20} to get
$$||e^{-t K }h_0||_{L^2(\mathbb{R}^{2d},f_{\infty})}\leq \max\{C_3, \mathscr{C}_1\}  e^{-\lambda t} ||h_0||_{L^2(\mathbb{R}^{2d},f_{\infty})}  $$
for all $t\geq 0.$
This inequality implies \eqref{e:L^2} with the constant  $\mathcal{C}_1\colonequals   \max\{C_3, \mathscr{C}_1\}.$


$(ii)$ \eqref{e:H^alpha} coincides with \eqref{e:L^2} when $\alpha=0.$
 Let $\alpha=1.$  \eqref{e:H^alpha0} lets us  write for $t \in [0, t_0]$
$$||e^{-tK}h_0||_{H^{1}_x(\mathbb{R}^{2d}, f_{\infty})}\leq \mathscr{C}_2 e^{\lambda  t_0}  e^{-\lambda t} ||h_0||_{H^{1}_x(\mathbb{R}^{2d}, f_{\infty})}.$$
 For $t\geq t_0,$ we use \eqref{norm H_x^1} and Theorem \ref{th:lin} $(ii)$ to get
 \begin{multline*}||e^{-tK}h_0||_{H^{1}_x(\mathbb{R}^{2d}, f_{\infty})}\leq \kappa_5||e^{-tK}h_0||_{H^{1}(\mathbb{R}^{2d}, f_{\infty})}\\ \leq \kappa_5C_3e^{-\lambda t}||h_0||_{L^{2}(\mathbb{R}^{2d}, f_{\infty})}\leq \kappa_5C_3e^{-\lambda t} ||h_0||_{H^{1}_x(\mathbb{R}^{2d}, f_{\infty})}.
 \end{multline*}
 The combination of these estimates imply \eqref{e:H^alpha} with $\mathcal{C}_2\colonequals \max\{ \mathscr{C}_2 e^{\lambda  t_0}, \kappa_5C_3 \}$  when $\alpha=1.$
 The case of $\alpha\in (0,1)$ follows by interpolation.

$(iii)$ 
 By \eqref{ee:H_x^alpha0} we have \begin{equation}\label{A222}
||e^{-t K }h_0||_{H^{\alpha}_x(\mathbb{R}^{2d}, f_{\infty})}\leq  \mathscr{C}_3 e^{\lambda_1 t_0} e^{-\lambda_1 t}(1+t^{-\frac{3\alpha}{2}})||h_0||_{L^2(\mathbb{R}^{2d}, f_{\infty})},   \,\forall\, t \in (0,t_0].
\end{equation}
For $t\geq t_0,$ we have by \eqref{e:H^alpha}
\begin{equation*}
|| e^{-t K }h_0||_{H^{\alpha }_x(\mathbb{R}^{2d}, f_{\infty})}\leq  \mathcal{C}_2 e^{-\lambda t}
||h_0||_{H^{\alpha }_x(\mathbb{R}^{2d}, f_{\infty})}.
\end{equation*}
 Since  $\lambda>\lambda_1>0,$ there is a constant $C>0$ depending $t_0, \, \lambda_1$ and $\lambda$ such that
$$\mathcal{C}_2 e^{-\lambda t}\leq {C}(1+t^{-\frac{3\alpha}{2}})e^{-\lambda_1 t},  \, \, \, \, \, \, \forall\, t\geq t_0.$$
Thus, we obtain
\begin{equation}\label{A111}
|| e^{-t K }h_0||_{H^{\alpha }_x(\mathbb{R}^{2d}, f_{\infty})}\leq {C}(1+t^{-\frac{3\alpha}{2}})e^{-\lambda_1 t}
||h_0||_{H^{\alpha }_x(\mathbb{R}^{2d}, f_{\infty})},  \, \, \, \, \, \, \forall\, t\geq t_0.
\end{equation}
 \eqref{A222} and \eqref{A111}  show that  \eqref{ee:H_x^alpha} holds with the constant $\mathcal{C}_3\colonequals\max\{C, \mathscr{C}_3 e^{\lambda_1 t_0}\}.  $ 


The proofs of $(iv)$ and $(v)$  follow by similar arguments as we did in $(iii)$ and $(iv).$
\end{proof}

\subsection{Local well-posedness}
In this subsection we prove Theorem \ref{Short time} i.e., the existence of a unique solution to \eqref{INTe} in a (possible short) time interval.   We use $\psi_h$ below to denote the solution of $-\Delta_x \psi=\int_{\mathbb{R}^d}hf_{\infty}dv.$
\begin{proof}[\textbf{Proof of Theorem \ref{Short time}} ] Let $t_0>0$ be a fixed constant as in Lemma \ref{lem: e^tK0}.
For a given $h_0\in H^{\alpha}_x(\mathbb{R}^{6}, f_{\infty})\cap H^{1}_v(\mathbb{R}^{6}, f_{\infty})$ we define a mapping
 $$F: C\left([0, \tau]; H^{\alpha}_x(\mathbb{R}^{6}, f_{\infty})\right)\cap C\left([0, \tau]; H^{1}_v(\mathbb{R}^{6}, f_{\infty}) \right)
 \to  C\left([0, \tau]; H^{\alpha}_x(\mathbb{R}^{6}, f_{\infty})\right)\cap C\left([0, \tau]; H^{1}_v(\mathbb{R}^{6}, f_{\infty}) \right)
 $$
  by
$$F[h]= e^{-t K } h_0+\int_0^t e^{-(t-s) K } \left(\nabla_x \psi_h\cdot \nabla_v h-\frac{\nu}{\sigma}v\cdot\nabla_x \psi_h h\right)  ds, \, \, \, \, t\in [0,\tau],$$
where $\tau\in (0, t_0]$ will be fixed later. We want to show that $F$ has a unique fixed point if $\tau $ is small enough.

 We define
 $$||h||_{\tau,1}\colonequals \sup_{t\in [0, \tau]}\{||h(t)||_{H_x^{\alpha}(\mathbb{R}^6,f_{\infty})}\} ,$$
 $$||h||_{\tau, 2}\colonequals \sup_{t\in [0, \tau]}\{|| h(t)||_{H^1_v(\mathbb{R}^6,f_{\infty})}\}.$$
 We use $$||h||_{\tau}\colonequals\max\{ ||h||_{\tau, 1}, ||h||_{\tau, 2}\} $$ as a norm in $C\left([0, \tau]; H^{\alpha}_x(\mathbb{R}^{6}, f_{\infty})\right)\cap C\left([0, \tau]; H^{1}_v(\mathbb{R}^{6}, f_{\infty}) \right).$
\eqref{gr.psi in L^infty} shows that $\nabla_x \psi_h(t)$ is bounded for all $t\geq 0.$  Hence $$\nabla_x \psi_h(t)\cdot \nabla_v h(t)-\frac{\nu}{\sigma}v\cdot\nabla_x \psi_h(t) h(t)\in L^2(\mathbb{R}^6,f_{\infty})$$ for all  $t\in[0,\tau].$  
 Using the H\"older inequality, \eqref{gr.psi in L^infty} and \eqref{Vil.s lemma2}
 \begin{multline}\label{FR0}
|| \nabla_x \psi_h(t)\cdot \nabla_v h(t)-{\nu}/{\sigma}v\cdot\nabla_x \psi_h(t) h(t)||_{L^2(\mathbb{R}^{6}, f_{\infty})}
\\
 \leq \sqrt{2||\nabla_x \psi_h(t)||^2_{L^{\infty}}\int_{\mathbb{R}^{6}}\left(  |\nabla_v h(t)|^2+{\nu^2}/{\sigma^2}|v|^2| h^2(t)\right)f_{\infty} dxdv}\\
 \leq C_R ||h(t)||_{H_x^{\alpha}(\mathbb{R}^6,f_{\infty})}|| h(t)||_{H^1_v(\mathbb{R}^6,f_{\infty})}
 \end{multline}
 for some constant $C_R>0.$ We  estimate $F[h]$ in $C\left([0, \tau]; H^{\alpha}_x(\mathbb{R}^{6}, f_{\infty})\right):$
 \begin{multline}\label{F10}
||F[h(t)]||_{H^{\alpha}_x(\mathbb{R}^{6}, f_{\infty})}\leq ||e^{-t K } h_0||_{H^{\alpha}_x(\mathbb{R}^{6}, f_{\infty})}\\ +\int_{0}^t ||e^{-(t-s) K } \left(\nabla_x \psi_h\cdot \nabla_v h-\nu/\sigma v\cdot\nabla_x \psi_h h \right) ||_{H^{\alpha}_x(\mathbb{R}^{6}, f_{\infty})} ds.
\end{multline}
\eqref{ee:H_x^alpha0} and \eqref{FR0}  let us estimate
 $$
 \int_{0}^t ||e^{-(t-s) K } \left(\nabla_x \psi_h\cdot \nabla_v h-\nu/\sigma v\cdot\nabla_x \psi_h h\right)||_{H^{\alpha}_x(\mathbb{R}^{6}, f_{\infty})} ds$$
 $$\leq \mathscr{C}_3\int_{0}^t (1+(t-s)^{-\frac{3\alpha}{2}})|| \nabla_x \psi_h(s)\cdot \nabla_v h(s)-\nu/\sigma v\cdot\nabla_x \psi_h(s) h(s)||_{L^2(\mathbb{R}^{6}, f_{\infty})} ds $$
\begin{equation}\label{F30}\leq \mathscr{C}_3C_R\int_{0}^t  (1+(t-s)^{-\frac{3\alpha}{2}})||h(s)||_{H_x^{\alpha}(\mathbb{R}^6,f_{\infty})}||h(s)||_{H_v^{1}(\mathbb{R}^6,f_{\infty})}ds
 \end{equation}
$$\leq \mathscr{C}_3C_R||h||_{\tau,1}||h||_{\tau, 2}\int_{0}^t  (1+(t-s)^{-\frac{3\alpha}{2}}) ds.$$
Then,  \eqref{e:H^alpha0}, \eqref{F10} and \eqref{F30} provide
\begin{equation}\label{||F||_1}
||F[h]||_{\tau, 1}\leq   \mathscr{C}_2|| h_0||_{H^{\alpha}_x(\mathbb{R}^{6}, f_{\infty})}+\mathscr{C}_3C_R||h||_{\tau, 1}||h||_{\tau, 2}\sup_{t\in [0,\tau]}\left\{\int_{0}^t  (1+(t-s)^{-\frac{3\alpha}{2}}) ds\right\}.
\end{equation}
We estimate $F[h]$ in $C\left([0, \tau], H^{1}_v(\mathbb{R}^{6}, f_{\infty}) \right):$
 \begin{multline}\label{FF10}
||F[h(t)]||_{H^{1}_v(\mathbb{R}^{6}, f_{\infty})}\leq ||e^{-t K } h_0||_{H^{1}_v(\mathbb{R}^{6}, f_{\infty})} \\+\int_{0}^t ||e^{-(t-s) K } \left(\nabla_x \psi_h\cdot \nabla_v h- \nu/\sigma v\cdot\nabla_x \psi_h h\right) ||_{H^{1}_v(\mathbb{R}^{6}, f_{\infty})} ds.
\end{multline}
 \eqref{ee:H_v^10} and \eqref{FR0} let us estimate
$$\int_{0}^t ||e^{-(t-s) K } \left(\nabla_x \psi_h(s)\cdot \nabla_v h(s)-\nu/\sigma v\cdot\nabla_x \psi_h(s) h(s)\right) ||_{H^{1}_v(\mathbb{R}^{6}, f_{\infty})} ds$$
$$\leq \mathscr{C}_5\int_{0}^t (1+(t-s)^{-\frac{1}{2}})|| \nabla_x \psi_h(s)\cdot \nabla_v h(s)-\nu/\sigma v\cdot\nabla_x \psi_h(s) h(s)||_{L^2(\mathbb{R}^{6}, f_{\infty})} ds $$
$$\leq \mathscr{C}_5C_R\int_{0}^t  (1+(t-s)^{-\frac{1}{2}})||h(s)||_{H_x^{\alpha}(\mathbb{R}^6,f_{\infty})}||h(s)||_{H_v^{1}(\mathbb{R}^6,f_{\infty})}ds$$
\begin{equation}\label{FF30}
\leq \mathscr{C}_5C_R||h||_{\tau,1}||h||_{\tau, 2}(t+2\sqrt{t}),
\end{equation}
where we used
 $$ \int_{0}^t  (1+(t-s)^{-\frac{1}{2}})ds= t+2\sqrt{t}.$$
\eqref{FF10}, \eqref{FF30} and \eqref{e:H^1_v0} show  $$ ||F[h(t)]||_{H^{1}_v(\mathbb{R}^{6}, f_{\infty})}\leq  \mathscr{C}_4 ||h_0||_{H^{1}_v(\mathbb{R}^{6}, f_{\infty})} +\mathscr{C}_5C_R ||h||_{1}||h||_{2}(t+2\sqrt{t}), $$
and so \begin{equation}\label{||F||_2}
||F[h]||_{\tau, 2}
\leq  \mathscr{C}_4 ||h_0||_{H^{1}_v(\mathbb{R}^{6}, f_{\infty})}+\mathcal{C}_5C_R||h||_{\tau,1}||h||_{\tau,2}(\tau+2\sqrt{\tau}).
\end{equation}

Let $h, g \in  C\left([0, \tau]; H^{\alpha}_x(\mathbb{R}^{6}, f_{\infty})\right)\cap C\left([0, \tau]; H^{1}_v(\mathbb{R}^{6}, f_{\infty}) \right).$  We consider
$$F[h]-F[g]= \int_0^t e^{-(t-s) K } \left(\nabla_x \psi_h\cdot \nabla_v (h-g)-\nu/\sigma v\cdot\nabla_x \psi_h (h-g)\right)  ds$$
$$+\int_0^t e^{-(t-s) K } \left((\nabla_x \psi_h-\nabla_x \psi_g)\cdot \nabla_v g-\nu/\sigma v\cdot (\nabla_x \psi_h-\nabla_x \psi_g) g\right)ds. $$
As we did in \eqref{FR0}, we can show  by using the H\"older inequality, \eqref{gr.psi in L^infty} and \eqref{Vil.s lemma2} that
\begin{multline}\label{wh-wg10}
||\nabla_x \psi_h(t)\cdot \nabla_v (h(t)-g(t))-\nu/\sigma v\cdot\nabla_x \psi_h (t)(h(t)-g(t)) ||_{L^2(\mathbb{R}^6,f_{\infty})}\\ \leq C_R ||h(t)||_{H_x^{\alpha}(\mathbb{R}^6,f_{\infty})}||h(t)-g(t)||_{H_v^{1}(\mathbb{R}^6,f_{\infty})}
\end{multline}
and \begin{multline}\label{wh-wg20}
||(\nabla_x \psi_h(t)-\nabla_x \psi_g(t))\cdot \nabla_v g(t)-\nu/\sigma v\cdot (\nabla_x \psi_h(t)-\nabla_x \psi_g(t)) g(t)||_{L^2(\mathbb{R}^6,f_{\infty})}\\ \leq C_R ||h(t)-g(t)||_{H_x^{\alpha}(\mathbb{R}^6,f_{\infty})}||g(t)||_{H_v^{1}(\mathbb{R}^6,f_{\infty})}
\end{multline}
hold for all $t\in[0,\tau].$
Using  \eqref{ee:H_v^10}, \eqref{wh-wg10} and  \eqref{wh-wg20} we estimate $F[h]-F[g]$ in\\ $ C\left([0, \tau]; H_v^{1}(\mathbb{R}^6,f_{\infty})\right):$
\begin{multline*}
||F[h(t)]-F[g(t)]||_{H_v^{1}(\mathbb{R}^6,f_{\infty})}\\  \leq \int_0^t ||e^{-(t-s) K } \left(\nabla_x \psi_h\cdot \nabla_v (h-g)-\nu/\sigma v\cdot\nabla_x \psi_h (h-g)\right) ||_{H_v^{1}(\mathbb{R}^6,f_{\infty})} ds
\end{multline*}$$+\int_0^t ||e^{-(t-s) K } \left((\nabla_x \psi_h-\nabla_x \psi_g)\cdot \nabla_v g-\nu/\sigma v\cdot (\nabla_x \psi_h-\nabla_x \psi_g) g\right)||_{H_v^{1}(\mathbb{R}^6,f_{\infty})}ds $$
$$\leq \mathscr{C}_5C_R\int_{0}^t  (1+(t-s)^{-\frac{1}{2}})||h(s)||_{H_x^{\alpha}(\mathbb{R}^6,f_{\infty})}||h(s)-g(s)||_{H_v^{1}(\mathbb{R}^6,f_{\infty})}ds$$
$$\quad+ \mathscr{C}_5C_R\int_{0}^t  (1+(t-s)^{-\frac{1}{2}})||h(s)-g(s)||_{H_x^{\alpha}(\mathbb{R}^6,f_{\infty})}||g(s)||_{H_v^{1}(\mathbb{R}^6,f_{\infty})}ds.
$$
This shows that
\begin{multline*}
||F[h(t)]-F[g(t)]||_{H_v^{1}(\mathbb{R}^6,f_{\infty})}
\leq \mathscr{C}_5C_R||h||_{\tau,1}||h-g||_{\tau, 2}\int_{0}^t  (1+(t-s)^{-\frac{1}{2}}) ds\\
+ \mathscr{C}_5C_R ||h-g||_{\tau,1} ||g||_{\tau, 2}\int_{0}^t  (1+(t-s)^{-\frac{1}{2}})ds.
\end{multline*}
Therefore, we get \begin{multline}\label{||F-F||_2}
||F[h(t)]-F[g(t)]||_{\tau,2}
\leq \mathscr{C}_5C_R||h||_{\tau, 1}||h-g||_{\tau, 2}(\tau+2\sqrt{\tau})
\\+ \mathscr{C}_5C_R ||h-g||_{\tau, 1} ||g||_{\tau, 2}(\tau+2\sqrt{\tau}).
\end{multline}
Similarly, we estimate $F[h]-F[g]$ in $C\left([0,\tau]; H_x^{\alpha}(\mathbb{R}^6,f_{\infty})\right)$ using  \eqref{ee:H_x^alpha0}, \eqref{wh-wg10} and  \eqref{wh-wg20}:
\begin{align*}
&||F[h(t)]-F[g(t)]||_{H_x^{\alpha}(\mathbb{R}^6,f_{\infty})}\\  \leq &\int_0^t ||e^{-(t-s) K } \left(\nabla_x \psi_h\cdot \nabla_v (h-g)-\nu/\sigma v\cdot\nabla_x \psi_h (h-g)\right) ||_{H_x^{\alpha}(\mathbb{R}^6,f_{\infty})} ds\\&+\int_0^t ||e^{-(t-s) K } \left((\nabla_x \psi_h-\nabla_x \psi_g)\cdot \nabla_v g-\nu/\sigma v\cdot (\nabla_x \psi_h-\nabla_x \psi_g) g\right)||_{H_x^{\alpha}(\mathbb{R}^6,f_{\infty})}ds \\
\leq& \mathscr{C}_3C_R\int_{0}^t  (1+(t-s)^{-\frac{3\alpha}{2}})||h(s)||_{H_x^{\alpha}(\mathbb{R}^6,f_{\infty})}||h(s)-g(s)||_{H_v^{1}(\mathbb{R}^6,f_{\infty})}ds\\
&+ \mathscr{C}_3C_R\int_{0}^t  (1+(t-s)^{-\frac{3\alpha}{2}})||h(s)-g(s)||_{H_x^{\alpha}(\mathbb{R}^6,f_{\infty})}||g(s)||_{H_v^{1}(\mathbb{R}^6,f_{\infty})}ds.
\end{align*}
This shows that
\begin{multline*}
||F[h(t)]-F[g(t)]||_{H_x^{\alpha}(\mathbb{R}^6,f_{\infty})}
\leq \mathscr{C}_3C_R||h||_{\tau,1} ||h-g||_{\tau, 2}\int_{0}^t  (1+(t-s)^{-\frac{3\alpha}{2}}) ds\\
+ \mathscr{C}_3C_R ||h-g||_{\tau,1} ||g||_{\tau, 2}\int_{0}^t  (1+(t-s)^{-\frac{3\alpha}{2}})ds.
\end{multline*}
We take  the supremum in time
\begin{multline}\label{||F-F||_1}
||F[h]-F[g]||_{\tau, 1}
\leq \mathscr{C}_3C_R  ||h||_{\tau,1} ||h-g||_{\tau, 2}\sup_{t\in [0, \tau]}\left\{\int_{0}^t  (1+(t-s)^{-\frac{3\alpha}{2}}) ds\right\}
\\+ \mathscr{C}_3C_R ||h-g||_{\tau, 1} ||g||_{\tau, 2}\sup_{t\in [0, \tau]}\left\{\int_{0}^t  (1+(t-s)^{-\frac{3\alpha}{2}}) ds\right\}.
\end{multline}
Let $r \colonequals 2 \max\{ \mathscr{C}_2|| h_0||_{H^{\alpha}_x(\mathbb{R}^{6}, f_{\infty})}, \mathscr{C}_4 ||h_0||_{H^{1}_v(\mathbb{R}^{6}, f_{\infty})}\}.$
We choose a small $\tau\in(0,t_0]$ such that
$$\frac{r}{2}+\mathscr{C}_3C_Rr^2\sup_{t\in [0,\tau]}\left\{{\int_{0}^t  (1+(t-s)^{-\frac{3\alpha}{2}})  ds}\right\}\leq r, \, \, \, \, \, \, \, \, \frac{r}{2}+\mathcal{C}_5C_Rr^2(\tau+2\sqrt{\tau})\leq r.$$
Then, \eqref{||F||_1} and \eqref{||F||_2} show that
$$||F[h]||_{\tau}\leq r\, \, \, \, \text{for} \, \, \, \, ||h||_{\tau}\leq r.$$
We choose a   smaller $\tau\in(0,t_0]$ such that
$$2\mathscr{C}_5C_R r(\tau+2\sqrt{\tau})
<\frac{1}{2}, \, \, \, \, \, \, 2\mathscr{C}_3C_R  r \sup_{t\in [0, \tau]}\left\{\int_{0}^t  (1+(t-s)^{-\frac{3\alpha}{2}}) ds\right\}
<\frac{1}{2},$$
then \eqref{||F-F||_2} and \eqref{||F-F||_1} show
$$||F[h]-F[g]||_{\tau}\leq \frac{1}{2}||h-g||_{\tau}\, \, \, \, \text{for} \, \, \, \, ||h||_{\tau}\leq r, \, \, ||g||_{\tau}\leq r.$$
By the well known
contraction principle $F$ has a unique fixed point in\\ $C\left([0, \tau], H^{\alpha}_x(\mathbb{R}^{6}, f_{\infty})\right)\cap C\left([0, \tau], H^{1}_v(\mathbb{R}^{6}, f_{\infty}) \right).$  This
fixed point is the desired solution of the integral equation
$$ h(t)=e^{-t K } h_0+\int_0^t e^{-(t-s) K } \left(\nabla_x \psi_h\cdot \nabla_v h-v\cdot\nabla_x \psi_h h\right)  ds, \, \, \, \, t\in [0,\tau].$$

We note that $\tau$ depends on $|| h_0||_{H^{\alpha}_x(\mathbb{R}^{6}, f_{\infty})}$ and $ ||h_0||_{H^{1}_v(\mathbb{R}^{6}, f_{\infty})}.$ From what we have just proved it follows that if $h$ is a mild solution on the interval $[0,\tau]$ it can be extended on the interval $[0, \tau+\tau_1]$ with $\tau_1 \in(0,t_0]$ by defining on $[\tau, \tau+\tau_1],$ $h(t)=h_1(t)$ where $h_1(t)$ is the solution of $$h_1(t)=e^{-(t-\tau)K}h(\tau)+\int_{\tau}^t e^{-(t-s)K} \left(\nabla_x \psi_{h_1}(s)\cdot \nabla_v h_1(s)-v\cdot\nabla_x \psi_{h_1}(s) h_1(s)\right)  ds.$$
Moreover, $\tau_1 \in(0,t_0]$ depends on $|| h(\tau)||_{H^{\alpha}_x(\mathbb{R}^{6}, f_{\infty})}$ and $ ||h(\tau)||_{H^{1}_v(\mathbb{R}^{6}, f_{\infty})}.$

Let $[0, t_{max})$ be the maximal interval of  existence of the solution. If $t_{max}<\infty$ then at least one of the limits
$$\lim_{t\nearrow t_{max}}||h(t)||_{ H^{\alpha}_x(\mathbb{R}^{6},f_{\infty})}  \, \, \, \,\text{ and } \, \, \, \, \, \, \,  \lim_{t\nearrow t_{max}}||h(t)||_{ H_v^1(\mathbb{R}^{6},f_{\infty})}$$
is infinite. Otherwise there is a sequence $t_n\nearrow t_{max},$ $n\in \mathbb{N},$ such that $||h(t_n)||_{ H^{\alpha}_x(\mathbb{R}^{6},f_{\infty})}$ and $||h(t_n)||_{ H^{1}_v(\mathbb{R}^{6},f_{\infty})}$ is bounded. This would imply what we have just proved that for each $t_n,$ near enough to $t_{max},$ the solution $h$ on the interval $[0,t_n]$ can be extended to the interval $[0, t_n+\delta],$ where $\delta \in(0,t_0]$ is independent of $t_n$ and  hence $h$ can be extended beyond $t_{max}. $ This contradicts the definition of $t_{max}.$

Then \eqref{gr.psi in Lalpha} completes the proof.
\end{proof}
\subsection{Global well-posedness and exponential stability}
Let $\lambda>0$ be the constant appearing in  Theorem \ref{th:lin}$\, (iii),$ $\lambda_1 \in (0,\lambda)$ and $\alpha\in [0,1].$   We define
$$X \colonequals\left\{h\in C\left([0, \infty); H^{\alpha}_x(\mathbb{R}^{6}, f_{\infty}) \right): \, \,\sup_{t\geq 0}\left\{e^{\lambda_1 t}||h(t)||_{ H^{\alpha}_x(\mathbb{R}^{6}, f_{\infty})}\right\}<\infty
\right\},
$$
$$Y\colonequals \left\{h\in C\left([0, \infty); H^{1}_v(\mathbb{R}^{6}, f_{\infty})\right): \, \, \sup_{t\geq 0}\left\{e^{\lambda_1 t}||h(t)||_{ H^{1}_v(\mathbb{R}^{6}, f_{\infty})}\right\}<\infty\right\}$$
with the norms
$$||h||_X\colonequals \sup_{t\geq 0}\left\{e^{\lambda_1 t}||h(t)||_{ H^{\alpha}_x(\mathbb{R}^{6}, f_{\infty})}\right\},$$
$$||h||_Y\colonequals \sup_{t\geq 0}\left\{e^{\lambda_1 t}||h(t)||_{ H^{1}_v(\mathbb{R}^{6}, f_{\infty})}\right\}.$$
 We denote $$I_1\colonequals \sup_{t\geq 0}\left\{\int_{0}^t  (1+(t-s)^{-\frac{3\alpha}{2}})e^{-\lambda_1 s}  ds\right\}, \,  I_2\colonequals \sup_{t\geq 0}\left\{\int_{0}^t  (1+(t-s)^{-\frac{1}{2}})e^{-\lambda_1 s}  ds\right\}.$$ Here $I_2$ is finite, but $I_1$ is finite if $\alpha \in (0,\frac{2}{3}).$

{ We need the following Gr\"onwall type inequality.
\begin{lemma}\label{lemma:Gron}
Let $a$ and $ b$ be positive constants, and $y:[0,\infty)\to [0, \infty)$ be a continuous function  satisfying
\begin{equation*}\label{dif.in}
y(t)\leq a +b \int_{0}^t (1+(t-s)^{-\frac{1}{2}})e^{-\lambda_1 s}y(s)ds, \, \, \, \,\forall\, t\in [0,\infty).
\end{equation*}
  Then there is a positive constant $\Lambda$ depending only on $\lambda_1$ such that
 \begin{equation}\label{adv.Gron}
 y(t)
\leq (a + ab I_2)e^{ b^2\Lambda}, \, \, \, \,\forall\, t\in [0,\infty).
 \end{equation}
\end{lemma}
\begin{proof}
We observe
\begin{align*}
y(t)\leq& a +b \int_{0}^t (1+(t-s)^{-\frac{1}{2}})e^{-\lambda_1 s}y(s)ds \\
\leq &a + ab \int_{0}^t (1+(t-s)^{-\frac{1}{2}})e^{-\lambda_1 s}ds\\&+b^2 \int_{0}^t (1+(t-s)^{-\frac{1}{2}})e^{-\lambda_1 s}\left[\int_{0}^s (1+(s-\tau)^{-\frac{1}{2}})e^{-\lambda_1 \tau}y(\tau)d\tau\right]ds.
\end{align*}
 By Fubini's theorem for computing multiple integrals, we obtain
 \begin{align}\label{y< a}
y(t)
\leq& a + ab \int_{0}^t (1+(t-s)^{-\frac{1}{2}})e^{-\lambda_1 s}ds\\\notag
&+b^2 \int_{0}^t  e^{-2\lambda_1 \tau}y(\tau)\left[\int_{\tau}^t(1+(t-s)^{-\frac{1}{2}})(1+(s-\tau)^{-\frac{1}{2}})e^{-\lambda_1(s-\tau)} ds \right]d\tau.
\end{align}
The integral in the brackets can be written as
 \begin{multline}\label{int mu}
\int_{\tau}^t(1+(t-s)^{-\frac{1}{2}})(1+(s-\tau)^{-\frac{1}{2}})e^{-\lambda_1 (s-\tau)} ds \\
=
\int_{0}^{t-\tau}(1+(t-\tau-s)^{-\frac{1}{2}})(1+s^{-\frac{1}{2}})e^{-\lambda_1 s} ds\\
=\int_{0}^{t-\tau}(1+(t-\tau-s)^{-\frac{1}{2}}+s^{-\frac{1}{2}}+(t-\tau-s)^{-\frac{1}{2}}s^{-\frac{1}{2}})e^{-\lambda_1 s} ds.
\end{multline}
We show this integral is bounded by a constant depending only on $\lambda_1.$ We first compute
$$\int_{0}^{t-\tau}e^{-\lambda_1 s} ds=\frac{1}{\lambda_1}(1-e^{-\lambda_1(t-\tau)})\leq \frac{1}{\lambda_1}$$
and
\begin{multline*}\int_{0}^{t-\tau}(t-\tau-s)^{-\frac{1}{2}}s^{-\frac{1}{2}}e^{-\lambda_1 s} ds\leq \int_{0}^{t-\tau}(t-\tau-s)^{-\frac{1}{2}}s^{-\frac{1}{2}} ds\\=\arcsin{\frac{2s-(t-\tau)}{t-\tau}}\big|^{t-\tau}_0=\pi.
\end{multline*}
If $t-\tau\leq 1,$  then
\begin{equation}\label{in<1}\int_{0}^{t-\tau} s^{-\frac{1}{2}}e^{-\lambda_1 s} ds\leq  \int_{0}^{t-\tau} s^{-\frac{1}{2}} ds=2(t-\tau)^{\frac{1}{2}}\leq 2.
\end{equation}
If $t-\tau>1,$ then
\begin{equation}\label{in0-1}\int_{0}^{1}s^{-\frac{1}{2}}e^{-\lambda_1 s}ds\leq  \int_{0}^{1}s^{-\frac{1}{2}} ds=2
\end{equation}
and
\begin{equation}\label{in1-t}\int_{1}^{t-\tau}s^{-\frac{1}{2}} e^{-\lambda_1 s} ds
\leq \int_{1}^{t-\tau} e^{-\lambda_1 s} ds=\frac{e^{-\lambda_1}(1-e^{-\lambda_1 (t-\tau -1)})}{\lambda_1} \leq \frac{e^{-\lambda_1}}{\lambda_1}.
\end{equation}
\eqref{in<1}, \eqref{in0-1}, and \eqref{in1-t} show
$$\int_{0}^{t-\tau} s^{-\frac{1}{2}}e^{-\lambda_1 s} ds\leq 2+\frac{e^{-\lambda_1}}{\lambda_1}, \, \, \, \, \, \, \,  \forall\, t-\tau\geq 0.$$
Similarly, if $t-\tau\leq 1,$  then
\begin{equation}\label{in<11}\int_{0}^{t-\tau} (t-\tau-s)^{-\frac{1}{2}}e^{-\lambda_1 s} ds\leq  \int_{0}^{t-\tau} (t-\tau-s)^{-\frac{1}{2}} ds=2(t-\tau)^{\frac{1}{2}}\leq 2.
\end{equation}
If $t-\tau>1,$ then
\begin{equation}\label{in0-11}\int_{0}^{t-\tau-1}(t-\tau-s)^{-\frac{1}{2}}e^{-\lambda_1 s}ds\leq  \int_{0}^{t-\tau-1} e^{-\lambda_1 s}ds\leq \frac{1}{\lambda_1}
\end{equation}
and
\begin{equation}\label{in1-t1}
\int_{t-\tau-1}^{t-\tau}(t-\tau-s)^{-\frac{1}{2}}e^{-\lambda_1 s} ds
\leq \int_{t-\tau-1}^{t-\tau}(t-\tau-s)^{-\frac{1}{2}}ds= 2.
\end{equation}
\eqref{in<1}, \eqref{in0-1}, and \eqref{in1-t} show
$$\int_{0}^{t-\tau} (t-\tau-s)^{-\frac{1}{2}}e^{-\lambda_1 s} ds\leq 2+\frac{1}{\lambda_1}, \, \, \, \, \, \, \,  \forall\, t-\tau\geq 0.$$
The estimates above shows that the integral in \eqref{int mu} is bounded by  $\frac{e^{-\lambda_1}+2}{\lambda_1}+\pi+4.$ Then we get from \eqref{y< a}
 $$
y(t)
\leq a + ab I_2
+b^2\left(\frac{e^{-\lambda_1}+2}{\lambda_1}+\pi+4\right)  \int_{0}^t  e^{-2\lambda_1 \tau}y(\tau)d\tau.
$$
The Gr\"onwall inequality yields
 $$
y(t)
\leq (a + ab I_2)e^{ b^2(\frac{e^{-\lambda_1}+2}{\lambda_1}+\pi+4) \int_{0}^t  e^{-2\lambda_1 \tau}d\tau}\leq (a + ab I_2)e^{ \frac{b^2}{2\lambda_1}(\frac{e^{-\lambda_1}+2}{\lambda_1}+\pi+4)}.
$$
This proves \eqref{adv.Gron} with the constant $\Lambda\colonequals \frac{1}{2\lambda_1}(\frac{e^{-\lambda_1}+2}{\lambda_1}+\pi+4).$
\end{proof}}
\begin{lemma}\label{Lem: Fix} Let $h_0\in H^{\alpha}_x(\mathbb{R}^{6}, f_{\infty})\cap H^{1}_{v}(\mathbb{R}^{6}, f_{\infty}),$ the assumptions $(A1),$ $(A2)$ and $(A3)$ hold.  Then, for any  $h \in X,$ there is a unique $w\in X\cap Y$ satisfying
\begin{equation}\label{u} w(t)= e^{-t K } h_0+\int_0^t e^{-(t-s) K } \left(\nabla_x \psi_h(s)\cdot \nabla_v w(s)-v\cdot\nabla_x \psi_h(s) w(s)\right)  ds, \, \, \, \, \forall t\geq 0.
\end{equation}
\end{lemma}
\begin{proof}
Let $\tau>0.$  We define a mapping  $$G: C\left([0, \tau]; H^{1}_v(\mathbb{R}^{6}, f_{\infty}) \right)\to  C\left([0, \tau]; H^{1}_v(\mathbb{R}^{6}, f_{\infty}) \right)$$ by
$$G[w]= e^{-t K } h_0+\int_0^t e^{-(t-s) K } \left(\nabla_x \psi_h\cdot \nabla_v w-\nu/\sigma v\cdot\nabla_x \psi_h w\right)  ds, \, \, \, \, t\in [0,\tau].$$
$w$ is a solution of \eqref{u}  if and only if it is a fixed point of $G.$ We first show that, if  $\tau>0$ is small enough, then  there is a unique fixed point of $G$  in $C\left([0, \tau]; H^{1}_v(\mathbb{R}^{6}, f_{\infty}) \right).$
 We use the norm  $||h||_{\tau,2}\colonequals \sup_{t\in [0,\tau]}\left\{||h(t)||_{ H^{1}_v(\mathbb{R}^{6}, f_{\infty})}\right\}$ in $C\left([0, \tau]; H^{1}_v(\mathbb{R}^{6}, f_{\infty}) \right).$
 As we assume $h \in X,$  \eqref{gr.psi in L^infty} shows that $\nabla_x \psi_h(t)$ is bounded for all $t\geq 0.$  Therefore, $$\nabla_x \psi_h(t)\cdot \nabla_v w(t)-\nu/\sigma v\cdot\nabla_x \psi_h(t) w(t)\in L^2(\mathbb{R}^6,f_{\infty})$$ for all $w \in C\left([0, \tau]; H^{1}_v(\mathbb{R}^{6}, f_{\infty}) \right)$ and $t\geq 0.$  
 Using the H\"older inequality, \eqref{gr.psi in L^infty} and \eqref{Vil.s lemma2}
 \begin{multline}\label{FR}
|| \nabla_x \psi_h(t)\cdot \nabla_v w(t)-\nu/\sigma v\cdot\nabla_x \psi_h(t) w(t)||_{L^2(\mathbb{R}^{6}, f_{\infty})}
\\
 \leq \sqrt{2||\nabla_x \psi_h(t)||^2_{L^{\infty}}\int_{\mathbb{R}^{6}}\left(  |\nabla_v w(t)|^2+\nu^2/\sigma^2|v|^2| w^2(t)\right)f_{\infty} dxdv}
 \\\leq C_R ||h(t)||_{H_x^{\alpha}(\mathbb{R}^6,f_{\infty})}|| w(t)||_{H^1_v(\mathbb{R}^6,f_{\infty})}
 \end{multline}
 for some constant $C_R>0.$ Also, we can check by integration by parts that
 \begin{equation}\label{incH}\nabla_x \psi_h(t)\cdot \nabla_v w(t)-\nu/\sigma v\cdot\nabla_x \psi_h(t) w(t)\in \mathcal{H}
 \end{equation} for all $t\geq 0.$

   We estimate $G[w]$ in $C\left([0,\tau]; H^{1}_v(\mathbb{R}^{6}, f_{\infty})\right):$
 \begin{multline}\label{FF1}
||G[w(t)]||_{H^{1}_v(\mathbb{R}^{6}, f_{\infty})}\leq ||e^{-t K } h_0||_{H^{1}_v(\mathbb{R}^{6}, f_{\infty})} \\+\int_{0}^t ||e^{-(t-s) K } \left(\nabla_x \psi_h\cdot \nabla_v w-v\cdot\nabla_x \psi_h w\right) ||_{H^{1}_v(\mathbb{R}^{6}, f_{\infty})} ds.
\end{multline}
 \eqref{ee:H_v^1}, \eqref{incH} and \eqref{FR} let us estimate the second term on the right side of \eqref{FF1}
$$\int_{0}^t ||e^{-(t-s) K } \left(\nabla_x \psi_h(s)\cdot \nabla_v w(s)-\nu/\sigma v\cdot\nabla_x \psi_h(s) w(s)\right) ||_{H^{1}_v(\mathbb{R}^{6}, f_{\infty})} ds$$
$$\leq \mathcal{C}_5\int_{0}^t (1+(t-s)^{-\frac{1}{2}})e^{-\lambda_1 (t-s)}|| \nabla_x \psi_h(s)\cdot \nabla_v w(s)-\nu/\sigma v\cdot\nabla_x \psi_h(s) w(s)||_{L^2(\mathbb{R}^{6}, f_{\infty})} ds $$
\begin{equation}\label{FF3}\leq \mathcal{C}_5C_R\int_{0}^t  (1+(t-s)^{-\frac{1}{2}})e^{-\lambda_1 (t-s)}||h(s)||_{H_x^{\alpha}(\mathbb{R}^6,f_{\infty})}||w(s)||_{H_v^{1}(\mathbb{R}^6,f_{\infty})}ds
\end{equation}
$$\leq \mathcal{C}_5C_Re^{-\lambda_1 t}||h||_{X}||w||_{\tau,2}\int_{0}^t  (1+(t-s)^{-\frac{1}{2}})ds=\mathcal{C}_5C_Re^{-\lambda_1 t}||h||_{X}||w||_{\tau,2}(t+2\sqrt{t}),$$
where we used
 $$ \int_{0}^t  (1+(t-s)^{-\frac{1}{2}})ds= t+2\sqrt{t}.$$
The estimates above show  $$ ||G[w(t)]||_{H^{1}_v(\mathbb{R}^{6}, f_{\infty})}\leq   ||e^{-t K } h_0||_{H^{1}_v(\mathbb{R}^{6}, f_{\infty})} +\mathcal{C}_5C_R e^{-\lambda_1 t}||h||_{X}||w||_{\tau,2}(t+2\sqrt{t}), $$
and so \begin{equation}\label{F_Y}
||G[w]||_{\tau,2}
\leq  ||e^{-t K } h_0||_{\tau,2}+\mathcal{C}_5C_R||h||_{X}||w||_{\tau,2}(\tau+2\sqrt{\tau}),
\end{equation}
 Let $r\colonequals  2||e^{-t K } h_0||_{\tau,2}.$ If $\tau$ is small enough so that
\begin{equation}\label{tau}\mathcal{C}_5C_R||h||_{X}(\tau+2\sqrt{\tau})\leq \frac{1}{2} ,
\end{equation}  then \eqref{F_Y} shows that
\begin{equation}\label{F<r}||G[w]||_{\tau,2}\leq r \, \, \, \, \, \, \, \text{ for any  } \, \, \, \, ||w||_{\tau, 2}\leq r.
\end{equation}\\
Similar computations  show
\begin{equation*}\label{F_Y-F_Y}
||G[w]-G[u]||_{\tau,2}
  \leq \mathcal{C}_5C_R||h||_X ||w-u||_{\tau,2}( \tau+2\sqrt{\tau})
\end{equation*}
 for all $w, u\in C\left([0, \tau]; H^{1}_v(\mathbb{R}^{6}, f_{\infty}) \right).$
If $\tau$ satisfies \eqref{tau}, then  \begin{equation}\label{F<1}
||G[w]-G[u]||_{\tau,2}
  \leq  \frac{1}{2} ||w-u||_{\tau,2}.
\end{equation}
We fix $\tau$ so that \eqref{tau} holds. Then, \eqref{F<r} and \eqref{F<1} shows that $F$ has a fixed point $w$ in $C\left([0, \tau]; H^{1}_v(\mathbb{R}^{6}, f_{\infty}) \right).$
 We note that $\tau$ only depends on the product $\mathcal{C}_4C_R|| h||_{X}<\infty.$
From what we have just proved it follows that if $w$ is a solution of \eqref{u} on the interval $[0,\tau],$ it can be extended to the interval $[0,2\tau]$ by defining on $[\tau, 2\tau],$ $w(t)=w_1(t)$ where $w_1$ is the solution of the integral equation
\begin{equation*} w_1(t)= e^{-(t-\tau) K } w(\tau)+\int_{\tau}^t e^{-(t-s) K } \left(\nabla_x \psi_h\cdot \nabla_v w_1-\nu/\sigma v\cdot\nabla_x \psi_h w_1\right)  ds, \, \, \, \, t\in [\tau, 2\tau].
\end{equation*}
After that, we extend this solution to the interval $[0,3 \tau]$ and so on. Thus, we can prove that there is  a unique global in time solution.

Next, we show that the solution $w$ is in $Y:$
\begin{multline*}||w(t)||_{H^{1}_v(\mathbb{R}^{6}, f_{\infty})}\leq ||e^{-t K } h_0||_{H^{1}_v(\mathbb{R}^{6}, f_{\infty})} \\+\int_{0}^t ||e^{-(t-s) K } \left(\nabla_x \psi_h\cdot \nabla_v w-\nu/\sigma v\cdot\nabla_x \psi_h w\right) ||_{H^{1}_v(\mathbb{R}^{6}, f_{\infty})} ds.
\end{multline*}
 Then, \eqref{e:H^1_v} and similar estimate as we did in  \eqref{FF3} show that
$$||w(t)||_{H^{1}_v(\mathbb{R}^{6}, f_{\infty})}\leq \mathcal{C}_4e^{-\lambda_1 t}||h_0||_{H^{1}_v (\mathbb{R}^{6}, f_{\infty})}$$
$$ +\mathcal{C}_5C_R\int_{0}^t  (1+(t-s)^{-\frac{1}{2}})e^{-\lambda_1 (t-s)}||h(s)||_{H_x^{\alpha}(\mathbb{R}^6,f_{\infty})}||w(s)||_{H_v^{1}(\mathbb{R}^6,f_{\infty})}ds$$
$$\leq  \mathcal{C}_4 e^{-\lambda_1 t}||h_0||_{H^{1}_v(\mathbb{R}^{6}, f_{\infty})} +\mathcal{C}_5 C_R ||h||_{X} e^{-\lambda_1 t}\int_{0}^t  (1+(t-s)^{-\frac{1}{2}})||w(s)||_{H_v^{1}(\mathbb{R}^6,f_{\infty})}ds.$$
It yields
\begin{multline*}e^{\lambda_1 t}||w(t)||_{H^{1}_v(\mathbb{R}^{6}, f_{\infty})}\leq  \mathcal{C}_4||h_0||_{H^{1}_v(\mathbb{R}^{6}, f_{\infty})}\\ +\mathcal{C}_5C_R ||h||_{X} \int_{0}^t  (1+(t-s)^{-\frac{1}{2}})e^{-\lambda_1 s}e^{\lambda_1 s}||w(s)||_{H_v^{1}(\mathbb{R}^6,f_{\infty})}ds.
\end{multline*}
{Using Lemma \ref{lemma:Gron} we obtain
$$e^{\lambda_1 t}||w(t)||_{H^{1}_v(\mathbb{R}^{6}, f_{\infty})}\leq  \mathcal{C}_4(1+\mathcal{C}_5C_RI_2 ||h||_{X} )e^{\mathcal{C}_5^2C_R^2\Lambda ||h||^2_{X} } ||h_0||_{H^{1}_v(\mathbb{R}^{6}, f_{\infty})}
$$
After taking the supremum in time
 \begin{equation}\label{w_Y}
||w||_{Y}\leq   \mathcal{C}_4(1+\mathcal{C}_5C_RI_2 ||h||_{X} )e^{\mathcal{C}_5^2C_R^2\Lambda ||h||^2_{X} } ||h_0||_{H^{1}_v(\mathbb{R}^{6}, f_{\infty})}<\infty.
\end{equation}
}

We show that the solution $w$ is also in $X:$
 \begin{multline}\label{F1}
||w(t)||_{H^{\alpha}_x(\mathbb{R}^{6}, f_{\infty})}\leq ||e^{-t K } h_0||_{H^{\alpha}_x(\mathbb{R}^{6}, f_{\infty})}\\ +\int_{0}^t ||e^{-(t-s) K } \left(\nabla_x \psi_h\cdot \nabla_v w-\nu/\sigma v\cdot\nabla_x \psi_h w\right) ||_{H^{\alpha}_x(\mathbb{R}^{6}, f_{\infty})} ds.
\end{multline}
\eqref{ee:H_x^alpha} and \eqref{FR}  let us estimate
 $$
 \int_{0}^t ||e^{-(t-s) K } \left(\nabla_x \psi_h\cdot \nabla_v w-\nu/\sigma v\cdot\nabla_x \psi_h w\right)||_{H^{\alpha}_x(\mathbb{R}^{6}, f_{\infty})} ds$$
 $$\leq \mathcal{C}_3\int_{0}^t (1+(t-s)^{-\frac{3\alpha}{2}})e^{-\lambda_1 (t-s)}|| \nabla_x \psi_h\cdot \nabla_v w-\nu/\sigma v\cdot\nabla_x \psi_h w||_{L^2(\mathbb{R}^{6}, f_{\infty})} ds $$
\begin{equation}\label{F3}\leq \mathcal{C}_3C_R\int_{0}^t  (1+(t-s)^{-\frac{3\alpha}{2}})e^{-\lambda_1 (t-s)}||h(s)||_{H_x^{\alpha}(\mathbb{R}^6,f_{\infty})}||w(s)||_{H_v^{1}(\mathbb{R}^6,f_{\infty})}ds
 \end{equation}
$$\leq \mathcal{C}_3C_Re^{-\lambda_1 t}||h||_{X}||w||_{Y}\int_{0}^t  (1+(t-s)^{-\frac{3\alpha}{2}})e^{-\lambda_1 s} ds.$$
Then, \eqref{F1}, \eqref{e:H^alpha} and \eqref{F3} provide
$$ e^{\lambda_1 t}||w(t)||_{H^{\alpha}_x(\mathbb{R}^{6}, f_{\infty})}\leq   \mathcal{C}_2|| h_0||_{H^{\alpha}_x(\mathbb{R}^{6}, f_{\infty})}+\mathcal{C}_3C_R||h||_{X}||w||_Y\int_{0}^t  (1+(t-s)^{-\frac{3\alpha}{2}})e^{-\lambda_1 s}  ds. $$
We take the supremum in time and use
\eqref{w_Y}
 \begin{multline}\label{w_X}
||w||_{X}\leq   \mathcal{C}_2|| h_0||_{H^{\alpha}_x(\mathbb{R}^{6}, f_{\infty})}\\+ \mathcal{C}_3\mathcal{C}_4 C_RI_1  (||h||_{X}+\mathcal{C}_5C_RI_2 ||h||^2_{X} )e^{\mathcal{C}_5^2C_R^2\Lambda ||h||^2_{X} } ||h_0||_{H^{1}_v(\mathbb{R}^{6}, f_{\infty})}<\infty.
\end{multline}
Finally, \eqref{w_Y} and \eqref{w_X} show that $w\in X\cap Y.$
\end{proof}


\begin{proof}[\textbf{Proof of Theorem \ref{th:VPFP}} ]

We construct a solution to \eqref{INTe} with a fixed point argument, and therefore we define the mapping $U:X\to X$ such that $U[h]$ (which is the value of $U$ at $h$) is the solution of
\begin{equation}\label{Th} U[h(t)]= e^{-t K } h_0+\int_0^t e^{-(t-s) K } \left(\nabla_x \psi_h(s)\cdot \nabla_v U[h(s)]-\nu/\sigma v\cdot\nabla_x \psi_h(s) U[h(s)]\right)  ds.
\end{equation}
Lemma \ref{Lem: Fix} provides that, for any $h\in X,$ there is a unique  $U[h]\in  X\cap Y$ which satisfies \eqref{Th}. Therefore,  this mapping  is well-defined.  $h\in X$ solves \eqref{INTe} if and only if $h=U[h].$ We will show that  $U$ has a unique fixed point in $X.$

 The estimates \eqref{w_Y} and \eqref{w_X} provide
{
\begin{equation}\label{W_Y}
||U[h]||_{Y}\leq   \mathcal{C}_4(1+\mathcal{C}_5C_RI_2 ||h||_{X} )e^{\mathcal{C}_5^2C_R^2\Lambda ||h||^2_{X} } ||h_0||_{H^{1}_v(\mathbb{R}^{6}, f_{\infty})},
\end{equation}
 \begin{multline}\label{W_X}
||U[h]||_{X}\leq   \mathcal{C}_2|| h_0||_{H^{\alpha}_x(\mathbb{R}^{6}, f_{\infty})}\\+ \mathcal{C}_3\mathcal{C}_4 C_RI_1 (||h||_{X}+\mathcal{C}_5C_RI_2 ||h||^2_{X} )e^{\mathcal{C}_5^2C_R^2\Lambda ||h||^2_{X} } ||h_0||_{H^{1}_v(\mathbb{R}^{6}, f_{\infty})}.
\end{multline}
}
For $h, g \in X,$ we have
$$U[h]-U[g]= \int_0^t e^{-(t-s) K } \left(\nabla_x \psi_h\cdot \nabla_v (U[h]-U[g])-\nu/\sigma v\cdot\nabla_x \psi_h (U[h]-U[g])\right)  ds$$
$$+\int_0^t e^{-(t-s) K } \left((\nabla_x \psi_h-\nabla_x \psi_g)\cdot \nabla_v U[g]-\nu/\sigma v\cdot (\nabla_x \psi_h-\nabla_x \psi_g) U[g]\right)ds. $$
As we did in \eqref{FR}, we can show  by using the H\"older inequality, \eqref{gr.psi in L^infty} and \eqref{Vil.s lemma2} that
\begin{multline}\label{wh-wg1}
||\nabla_x \psi_h(t)\cdot \nabla_v (U[h(t)]-U[g(t)])-\nu/\sigma v\cdot\nabla_x \psi_h (t)(U[h(t)]-U[g(t)]) ||_{L^2(\mathbb{R}^6,f_{\infty})}\\ \leq C_R ||h(t)||_{H_x^{\alpha}(\mathbb{R}^6,f_{\infty})}||U[h(t)]-U[g(t)]||_{H_v^{1}(\mathbb{R}^6,f_{\infty})}
\end{multline}
and \begin{multline}\label{wh-wg2}
||(\nabla_x \psi_h(t)-\nabla_x \psi_g(t))\cdot \nabla_v U[g(t)]-\nu/\sigma v\cdot (\nabla_x \psi_h(t)-\nabla_x \psi_g(t)) U[g(t)]||_{L^2(\mathbb{R}^6,f_{\infty})}\\ \leq C_R ||h(t)-g(t)||_{H_x^{\alpha}(\mathbb{R}^6,f_{\infty})}||U[g(t)]||_{H_v^{1}(\mathbb{R}^6,f_{\infty})}
\end{multline}
hold for all $t\geq 0.$ Integrating by parts we can check \begin{equation}\label{incHWh}\nabla_x \psi_h(t)\cdot \nabla_v (U[h(t)]-U[g(t)])-\nu/\sigma v\cdot\nabla_x \psi_h(t) (U[h(t)]-U[g(t)]) \in \mathcal{H}\end{equation}
and
\begin{equation}\label{incHWg}(\nabla_x \psi_h(t)-\nabla_x \psi_g(t))\cdot \nabla_v U[g(t)]-\nu/\sigma v\cdot (\nabla_x \psi_h(t)-\nabla_x \psi_g(t)) U[g(t)] \in \mathcal{H}\end{equation} for all $t\geq 0.$
Using  \eqref{ee:H_v^1}, \eqref{incHWh}, \eqref{incHWg}, \eqref{wh-wg1} and  \eqref{wh-wg2} we estimate $U[h]-U[g]$ in $Y:$ \begin{small}
\begin{multline*}
||U[h(t)]-U[g(t)]||_{H_v^{1}(\mathbb{R}^6,f_{\infty})}\\  \leq \int_0^t ||e^{-(t-s) K } \left(\nabla_x \psi_h\cdot \nabla_v (U[h]-U[g])-\nu/\sigma v\cdot\nabla_x \psi_h (U[h]-U[g])\right) ||_{H_v^{1}(\mathbb{R}^6,f_{\infty})} ds\\+\int_0^t ||e^{-(t-s) K } \left((\nabla_x \psi_h-\nabla_x \psi_g)\cdot \nabla_v U[g]-\nu/\sigma v\cdot (\nabla_x \psi_h-\nabla_x \psi_g) U[g]\right)||_{H_v^{1}(\mathbb{R}^6,f_{\infty})}ds \\
\leq \mathcal{C}_5C_R\int_{0}^t  (1+(t-s)^{-\frac{1}{2}})e^{-\lambda_1 (t-s)}||h(s)||_{H_x^{\alpha}(\mathbb{R}^6,f_{\infty})}||U[h(s)]-U[g(s)]||_{H_v^{1}(\mathbb{R}^6,f_{\infty})}ds\\
+ \mathcal{C}_5C_R\int_{0}^t  (1+(t-s)^{-\frac{1}{2}})e^{-\lambda_1 (t-s)}||h(s)-g(s)||_{H_x^{\alpha}(\mathbb{R}^6,f_{\infty})}||U[g(s)]||_{H_v^{1}(\mathbb{R}^6,f_{\infty})}ds.
\end{multline*}
\end{small}
This shows that
\begin{multline*}
e^{\lambda_1 t}||U[h(t)]-U[g(t)]||_{H_v^{1}(\mathbb{R}^6,f_{\infty})}\\
\leq \mathcal{C}_5C_R||h||_X\int_{0}^t  (1+(t-s)^{-\frac{1}{2}})e^{-\lambda_1 s} e^{\lambda_1 s}||U[h(s)]-U[g(s)]||_{H_v^{1}(\mathbb{R}^6,f_{\infty})}ds\\
+ \mathcal{C}_5C_R ||h-g||_{X} ||U[g]||_{Y}\int_{0}^t  (1+(t-s)^{-\frac{1}{2}})e^{-\lambda_1 s}ds\\
\leq \mathcal{C}_5C_R||h||_X\int_{0}^t  (1+(t-s)^{-\frac{1}{2}})e^{-\lambda_1 s} e^{\lambda_1 s}||U[h(s)]-U[g(s)]||_{H_v^{1}(\mathbb{R}^6,f_{\infty})}ds\\
+ \mathcal{C}_5C_R I_2 ||h-g||_{X} ||U[g]||_{Y}.
\end{multline*}
We apply Lemma \ref{lemma:Gron} to this inequality
\begin{multline*}
e^{\lambda_1 t}||U[h(t)]-U[g(t)]||_{H_v^{1}(\mathbb{R}^6,f_{\infty})}\\
\leq
\mathcal{C}_5C_R I_2 (1+ \mathcal{C}_5C_RI_2||h||_X)e^{ \mathcal{C}^2_5C^2_R\Lambda||h||^2_X} ||h-g||_{X}  ||U[g]||_{Y}.
\end{multline*}
We take  the supremum in time to get
\begin{equation*}
||U[h]-U[g]||_{Y}
\leq \mathcal{C}_5C_R I_2(1+ \mathcal{C}_5C_RI_2||h||_X)e^{ \mathcal{C}^2_5C^2_R\Lambda||h||^2_X}  ||h-g||_{X} ||U[g]||_{Y} .
\end{equation*}
The estimate on $||U[g]||_Y$  in \eqref{W_Y} shows
\begin{small}\begin{multline}\label{Wg-Wg_Y}
||U[h]-U[g]||_{Y}
\leq \\  \mathcal{C}_4 \mathcal{C}_5C_R I_2 (1+ \mathcal{C}_5C_RI_2||h||_X) (1+\mathcal{C}_5C_RI_2 ||g||_{X} )e^{ \mathcal{C}^2_5C^2_R\Lambda(||h||^2_X+||g||^2_{X})} ||h_0||_{H^{1}_v(\mathbb{R}^{6}, f_{\infty})}||h-g||_{X}
\end{multline}
\end{small}

Using  \eqref{ee:H_x^alpha}, \eqref{wh-wg1} and  \eqref{wh-wg2} we estimate $U[h]-U[g]$ in $X:$
\begin{small}
\begin{multline*}
||U[h(t)]-U[g(t)]||_{H_x^{\alpha}(\mathbb{R}^6,f_{\infty})}\\  \leq \int_0^t ||e^{-(t-s) K } \left(\nabla_x \psi_h\cdot \nabla_v (U[h]-U[g])-\nu/\sigma v\cdot\nabla_x \psi_h (U[h]-U[g])\right) ||_{H_x^{\alpha}(\mathbb{R}^6,f_{\infty})} ds\\+\int_0^t ||e^{-(t-s) K } \left((\nabla_x \psi_h-\nabla_x \psi_g)\cdot \nabla_v U[g]-\nu/\sigma v\cdot (\nabla_x \psi_h-\nabla_x \psi_g) U[g]\right)||_{H_x^{\alpha}(\mathbb{R}^6,f_{\infty})}ds \\
\leq \mathcal{C}_3C_R\int_{0}^t  (1+(t-s)^{-\frac{3\alpha}{2}})e^{-\lambda_1 (t-s)}||h(s)||_{H_x^{\alpha}(\mathbb{R}^6,f_{\infty})}||U[h(s)]-U[g(s)]||_{H_v^{1}(\mathbb{R}^6,f_{\infty})}ds\\
+ \mathcal{C}_3C_R\int_{0}^t  (1+(t-s)^{-\frac{3\alpha}{2}})e^{-\lambda_1 (t-s)}||h(s)-g(s)||_{H_x^{\alpha}(\mathbb{R}^6,f_{\infty})}||U[g(s)]||_{H_v^{1}(\mathbb{R}^6,f_{\infty})}ds.
\end{multline*}
\end{small}
This shows that
\begin{multline*}
e^{\lambda_1 t}||U[h(t)]-U[g(t)]||_{H_x^{\alpha}(\mathbb{R}^6,f_{\infty})}
\\
\leq \mathcal{C}_3C_R||h||_X ||U[h]-U[g]||_{Y}\int_{0}^t  (1+(t-s)^{-\frac{3\alpha}{2}})e^{-\lambda_1 s} ds\\
+ \mathcal{C}_3C_R ||h-g||_{X} ||U[g]||_{Y}\int_{0}^t  (1+(t-s)^{-\frac{3\alpha}{2}})e^{-\lambda_1 s}ds.
\end{multline*}
We take  the supremum in time and  obtain
\begin{equation*}
||U[h]-U[g]||_{X}\\
\leq \mathcal{C}_3C_R I_1 ||h||_X ||U[h]-U[g]||_{Y}\\
+ \mathcal{C}_3C_RI_1 ||h-g||_{X} ||U[g]||_{Y}.
\end{equation*}
The estimates \eqref{W_Y} and \eqref{Wg-Wg_Y} provide
{
\begin{multline}\label{W_h-W_g<h_0}
||U[h]-U[g]||_{X}\\
\leq \left[\mathcal{C}_3
\mathcal{C}_4 \mathcal{C}_5C_R^2 I_1 I_2 (||h||_X+ \mathcal{C}_5C_RI_2||h||_X^2) (1+\mathcal{C}_5C_RI_2 ||g||_{X} )e^{ \mathcal{C}^2_5C^2_R\Lambda(||h||^2_X+||g||^2_{X})} \right.
\\
\left.+ \mathcal{C}_3  \mathcal{C}_4 C_RI_1(1+\mathcal{C}_5C_RI_2 ||g||_{X} )e^{\mathcal{C}_5^2C_R^2\Lambda ||g||^2_{X} } \right]||h_0||_{H^{1}_v(\mathbb{R}^{6}, f_{\infty})}||h-g||_{X}
\end{multline}
}

 Let $r>0,$ then there exist $\delta_1=\delta_1(r)>0$ and $\delta_2=\delta_2(r)>0$ such that
 \begin{small}
  \begin{equation}\label{int2}
  \begin{cases}
 \mathcal{C}_2\delta_1+ \mathcal{C}_3\mathcal{C}_4 C_RI_1 (r+\mathcal{C}_5C_RI_2 r^2 )e^{\mathcal{C}_5^2C_R^2\Lambda r^2 } \delta_2\leq r\\
 \left[\mathcal{C}_3
\mathcal{C}_4 \mathcal{C}_5C_R^2 I_1 I_2 r(1+\mathcal{C}_5C_RI_2 r )^2e^{2 \mathcal{C}^2_5C^2_R\Lambda r^2}+ \mathcal{C}_3  \mathcal{C}_4 C_RI_1(1+\mathcal{C}_5C_RI_2 r )e^{\mathcal{C}_5^2C_R^2\Lambda r^2 }
\right]\delta_2 <1.
\end{cases}
\end{equation}
\end{small}
If $h_0$ satisfies \begin{equation}\label{intcond}
|| h_0||_{H^{\alpha}_x(\mathbb{R}^{6}, f_{\infty})}\leq \delta_1\, \, \, \, \text{and} \, \, \, \, \,  ||h_0||_{H^{1}_v(\mathbb{R}^{6}, f_{\infty})}\leq \delta_2,
\end{equation}  
 then \eqref{W_X},  \eqref{W_h-W_g<h_0} and \eqref{int2} show
\begin{equation*}\label{W[h]<1}
 ||U[h]||_{X}\leq r \, \, \, \, \text{ for all } \, \, ||h||_{X}\leq r
\end{equation*}
and
\begin{equation*}\label{Wh-Wg<h-g}
||U[h]-U[g]||_{X}
\leq \epsilon||h-g||_{X}, \, \, \, \, \, \, \,  \text{ for all } \, \, ||h||_{X}\leq r,\, \, ||g||_{X}\leq r,
\end{equation*}
where \begin{small}$$\epsilon\colonequals \left[\mathcal{C}_3
\mathcal{C}_4 \mathcal{C}_5C_R^2 I_1 I_2 r(1+\mathcal{C}_5C_RI_2 r )^2e^{2 \mathcal{C}^2_5C^2_R\Lambda r^2}+ \mathcal{C}_3  \mathcal{C}_4 C_RI_1(1+\mathcal{C}_5C_RI_2 r )e^{\mathcal{C}_5^2C_R^2\Lambda r^2 }
\right]\delta_2$$\end{small}is a constant in $ (0,1). $
Then, the contraction principle yields that $U$ has a unique fixed point $h \in  X$ such that $||h||_X\leq r.$ Moreover,  \eqref{W_Y} shows  the fixed point $h$ is also in $Y$ and
$$||h||_{Y}\leq   \mathcal{C}_4(1+\mathcal{C}_5C_RI_2 ||h||_{X} )e^{\mathcal{C}_5^2C_R^2\Lambda ||h||^2_{X} } ||h_0||_{H^{1}_v(\mathbb{R}^{6}, f_{\infty})}\leq   \mathcal{C}_4(1+\mathcal{C}_5C_RI_2 r )e^{\mathcal{C}_5^2C_R^2\Lambda r^2 }  \delta_2.$$
Therefore, Theorem \ref{th:VPFP} holds with $C_5\colonequals r,$  $C_6\colonequals \mathcal{C}_4(1+\mathcal{C}_5C_RI_2 r )e^{\mathcal{C}_5^2C_R^2\Lambda r^2 }  \delta_2$ and $C_7\colonequals \theta_2 r$ by \eqref{gr.psi in Lalpha}.

Here, $r$ can be any positive number and  there always exist $\delta_1=\delta_1(r)>0$ and $\delta_2=\delta_2(r)>0$ such that \eqref{int2} holds. To have the condition \eqref{intcond} with larger $\delta_1$ and $\delta_2,$ we fix $r>0$ so that   $\delta_1=\delta_1(r)>0$ and $\delta_2=\delta_2(r)>0$ are as large as possible.
\end{proof}






\section*{Acknowledgments}
The author thanks his supervisor  Anton Arnold for stimulating discussions. The author is supported by the Austrian Science Fund (FWF) project \href{https://doi.org/10.55776/F65}{10.55776/F65}.









\medskip

\end{document}